\documentclass{amsart}
\usepackage[T1]{fontenc}
\usepackage[utf8]{inputenc}
\usepackage{amsmath,amsfonts,amssymb,amsthm}
\usepackage{mathtools}
\usepackage{tensor}
\usepackage{pgf,tikz}
\usetikzlibrary{matrix,arrows,calc}
\usepackage{tikz-cd}
\usepackage{wrapfig}
\usepackage{shuffle}
\usepackage{url}
\usepackage{xcolor}
\numberwithin{equation}{section}
\numberwithin{table}{section}

\usepackage{caption} 

\usepackage{enumitem}

\usepackage[colorinlistoftodos]{todonotes}

\definecolor{darkgreen}{rgb}{0,0.5,0}
\usepackage[
colorlinks, citecolor=darkgreen,
colorlinks=true,
]{hyperref}
\usepackage{cleveref}
\usepackage{comment}

\newtheorem{thm}{Theorem}
\newtheorem{prop}[thm]{Proposition}
\newtheorem{lemma}[thm]{Lemma}
\newtheorem{cor}[thm]{Corollary}
\newtheorem{conj}[thm]{Conjecture}

\theoremstyle{remark}
\newtheorem{rem}[thm]{Remark}

\newtheorem{example}[thm]{Example}

\theoremstyle{definition}
\newtheorem{defn}[thm]{Definition}
\newtheorem{cond}[thm]{Condition}

\AddToHook{env/defn/begin}{\crefalias{thm}{defn}}
\AddToHook{env/prop/begin}{\crefalias{thm}{prop}}
\AddToHook{env/lemma/begin}{\crefalias{thm}{lemma}}
\AddToHook{env/conjecture/begin}{\crefalias{thm}{conjecture}}
\AddToHook{env/cor/begin}{\crefalias{thm}{cor}}
\AddToHook{env/rem/begin}{\crefalias{thm}{rem}}
\AddToHook{env/remarks/begin}{\crefalias{thm}{remarks}}
\AddToHook{env/example/begin}{\crefalias{thm}{example}}
\AddToHook{env/cond/begin}{\crefalias{thm}{cond}}
\AddToHook{env/conj/begin}{\crefalias{thm}{conj}}

\numberwithin{thm}{section}

\mathchardef\mhyphen="2D %

\newcommand{\bA}{\mathbb A}

\newcommand{\cB}{\mathcal B}
\newcommand{\cC}{\mathcal C}

\newcommand{\bC}{\mathbb C}
\newcommand{\rd}{\mathrm d}
\newcommand{\cE}{\mathcal E}

\newcommand{\bF}{\mathbb F}

\newcommand{\bG}{\mathbb G}

\newcommand{\rH}{\mathrm H}

\newcommand{\fl}{\mathfrak l}

\newcommand{\fn}{\mathfrak n}
\newcommand{\MT}{\mathrm{MT}}

\newcommand{\bP}{\mathbb P}

\newcommand{\fp}{\mathfrak p}
\newcommand{\fP}{\mathfrak P}
\newcommand{\fq}{\mathfrak q}
\newcommand{\bQ}{\mathbb Q}
\newcommand{\bZ}{\mathbb Z}
\newcommand{\cO}{\mathcal{O}}

\newcommand{\fu}{\mathfrak{u}}
\newcommand{\rZ}{\mathrm{Z}}

\newcommand{\llangle}{\langle\!\langle}
\newcommand{\rrangle}{\rangle\!\rangle}
\newcommand{\an}{{\mathrm{an}}}

\newcommand{\ur}{{\mathrm{ur}}}
\newcommand{\prim}{{\mathrm{prim}}}

\newcommand{\mot}{{\mathrm{mot}}}

\newcommand{\Fib}{{\mathrm{Fib}}}
\newcommand{\Gon}{{\mathrm{Gon}}}
\newcommand{\ev}{{\mathrm{ev}}}
\newcommand{\univ}{{\mathrm{univ}}}
\newcommand{\red}{{\mathrm{red}}}

\newcommand{\eps}{\varepsilon}

\newcommand{\et}{\operatorname{\acute{e}t}}
\newcommand{\Gal}{\mathrm{Gal}}

\newcommand{\Vect}{\mathrm{Vect}}

\newcommand{\lto}{\longrightarrow}
\newcommand\fcov{\mathrm{f\mhyphen cov}}

\DeclareMathOperator{\ab}{ab}

\DeclareMathOperator{\ad}{ad}

\DeclareMathOperator{\Aut}{Aut}

\DeclareMathOperator{\cris}{cris}

\DeclareMathOperator{\Ext}{Ext}

\DeclareMathOperator{\gr}{gr}
\DeclareMathOperator{\crs}{cr}
\DeclareMathOperator{\dR}{dR}
\DeclareMathOperator{\pr}{pr}

\DeclareMathOperator{\res}{res}
\DeclareMathOperator{\loc}{loc}
\DeclareMathOperator{\li}{li}
\DeclareMathOperator{\Li}{Li}
\DeclareMathOperator{\Lie}{Lie}

\DeclareMathOperator{\Res}{Res}

\DeclareMathOperator{\Hom}{Hom}
\DeclareMathOperator{\Nm}{Nm}
\DeclareMathOperator{\per}{per}
\DeclareMathOperator{\PL}{PL}
\DeclareMathOperator{\Rep}{Rep}
\DeclareMathOperator{\Sel}{Sel}
\DeclareMathOperator{\Spec}{Spec}

\mathchardef\mhyphen="2D %

\usepackage[
    backend=biber,
    style=alphabetic,
    minalphanames=3,
    maxnames=99,
    maxalphanames=4,
    giveninits=true,
    isbn=false,
    doi=true,
]{biblatex}
\title{Polylogarithmic Chabauty--Kim loci over number fields}

\author{Xiang Li}
\address{Xiang Li,
	School of Mathematics,
	The University of Edinburgh,
    James Clerk Maxwell Building,
    Peter Guthrie Tait Road,
    Edinburgh,
    EH9 3FD, United Kingdom}
\email{X.Li-198@sms.ed.ac.uk}

\author{Martin Lüdtke}
\address{Martin Lüdtke,
	Institut für Mathematik,
	Carl von Ossietzky Universität Oldenburg,
	26111 Oldenburg,
	Germany
}
\email{martin.luedtke@uol.de}

\begin{document}

    \begin{abstract}
We study polylogarithmic Chabauty--Kim loci for $S$-integral points on $\mathbb{P}^1\smallsetminus\left\{0,1,\infty\right\}$ over number fields. We compare motivic and étale Selmer schemes, describe Galois actions on Selmer schemes and period rings, and give an explicit formula for the localisation map on the polylogarithmic Selmer scheme. For imaginary and real quadratic fields, we derive equations and determine several polylogarithmic and full Chabauty--Kim loci and show that Kim's Conjecture holds in several new cases. In other cases, we show that the polylogarithmic Chabauty--Kim method is insufficient to cut out precisely the $S$-integral points, even when combined with $S_3$-symmetrisation, due to additional points arising from $p$-adic roots of unity. As a further application, we give a Chabauty--Kim theoretic proof of the $S$-Selmer Section Conjecture for imaginary quadratic fields when $S$ is empty or $S$ consists of a single prime not fixed by complex conjugation.
\end{abstract}

	\maketitle

\tableofcontents
    
	\thispagestyle{empty}

    \section{Introduction}

For a number field~$K$ and a finite set~$S$ of primes of~$K$, it is a classical theorem of Siegel and Mahler that the set $X(\cO_{K,S})$ of $S$-integral points of the thrice-punctured line
\[
    X=\bP^1\smallsetminus\{0,1,\infty\}
\]
is finite.  In \cite{kim_2005_the}, Kim gave a new proof of this result for $K=\bQ$ by developing a non-abelian generalisation of Chabauty's method.

Fix a rational prime~$p$ such that no prime in~$S$ lies above~$p$, and a quotient $\pi_1^{\mot}(X,0)\twoheadrightarrow\Pi$ of the motivic fundamental group of Deligne--Goncharov at the tangential base point $\vec{1}_0$ at~$0$ \cite{deligne-goncharov}. Central to our approach over number fields is the motivic Chabauty--Kim diagram
\begin{equation}
\label{eq:ck-diagram-intro}
\begin{tikzcd}[column sep=large,row sep=large]
X(\cO_{K,S}) \rar[hook] \dar["j_S^{\mot}"'] &
\displaystyle\prod_{\fp\mid p}X(\cO_{\fp})
    \dar["\prod_{\fp}j_{\fp}^{\dR}"] \\
\Sel_{S,\Pi}^{\mot}(X)_{\bQ_p}
    \rar["\loc_p"'] &
\displaystyle\prod_{\fp\mid p}\Res_{K_{\fp}/\bQ_p}\Pi_{K_{\fp}}^{\dR}.
\end{tikzcd}
\end{equation}
Here $\Sel_{S,\Pi}^{\mot}(X)$ is the motivic Selmer scheme.  As summarised in \Cref{sec:intro-framework}, this diagram is canonically identified with the usual étale Chabauty--Kim diagram, in which the bottom arrow is the product of the localisation maps for all primes $\fp \mid p$.  The Chabauty--Kim locus is defined as the inverse image under the right vertical map of the scheme-theoretic image of the bottom horizontal map.  Taking for $\Pi$ the depth-$n$ quotient, i.e.\ the $n$-th quotient along the descending central series, gives a subset
\[
    X(\cO_K\otimes\bZ_p)_{S,n}\subseteq
    X(\cO_K\otimes\bZ_p)=\prod_{\fp\mid p}X(\cO_{\fp})
\]
containing the diagonal image of $X(\cO_{K,S})$.  These Chabauty--Kim loci form a descending sequence as~$n$ increases.  Kim's Conjecture predicts that
\[
    X(\cO_K\otimes\bZ_p)_{S,n}=X(\cO_{K,S})
\]
for sufficiently large~$n$.  Over general number fields, Dogra proved the eventual finiteness of (refined) Chabauty--Kim loci and thereby obtained a new proof of the Siegel--Mahler theorem \cite{dogra_2023_unlikely}.  Much of the existing explicit and motivic theory has nevertheless been developed with $K=\bQ$ as its principal setting. The presence of several primes above~$p$, together with the arithmetic of the ground field and its Galois action, makes explicit Chabauty--Kim theory over number fields substantially richer. 
In this paper we restrict attention to the \emph{polylogarithmic quotient} of the fundamental group, whose associated Chabauty--Kim loci
\[
    X(\cO_K\otimes\bZ_p)_{S,\PL,n}
    \supseteq X(\cO_K\otimes\bZ_p)_{S,n}
\]
are more amenable to computation. They are cut out by functions involving the $p$-adic logarithm and polylogarithms
\[ \log(z) = \int_{\vec{1}_0}^z \frac{\rd t}{t}, \quad \Li_m(z) \coloneqq \int_{\vec{1}_0}^z \underbrace{\frac{{\rd}t}{t} \cdots \frac{{\rd}t}{t} \frac{{\rd}t}{1-t}}_m   \quad (1 \leq m \leq n), \]
rather than more general iterated Coleman integrals from the tangential basepoint $\vec{1}_0$ to~$z$. A first objective of this paper is to develop the motivic, Galois-equivariant, and computational foundations needed to derive equations defining the polylogarithmic Chabauty--Kim loci $X(\cO_K\otimes\bZ_p)_{S,\PL,n}$ over general number fields. A second is to use this framework to compute these loci and test Kim's Conjecture over quadratic number fields. 

Before discussing the details of the method, we state some of the main results obtained by applying it in the case of quadratic number fields. The first non-trivial setting is the case where $K$ is imaginary quadratic and $S = \emptyset$. 

\begin{thm}[Equations for integral points over imaginary quadratic fields, {\Cref{integral-points-equations}}]
    \label{mainthm:integral-points-equations}
    Let $K$ be an imaginary quadratic field and let $p$ be a prime which splits completely in~$K$. For $1 \leq n \leq \infty$, the depth-$n$ polylogarithmic Chabauty--Kim locus $X(\cO_K \otimes \bZ_p)_{\emptyset,\PL,n}$ is contained in the set of pairs $(z_1,z_2) \in X(\bZ_p) \times X(\bZ_p)$ satisfying the equations
    \begin{align}
        \log(z_1) = \log(z_2) &= 0, %
        \\
        \Li_1(z_1) = \Li_1(z_2) &= 0, %
        \\
        \Li_m(z_1) + (-1)^m \Li_m(z_2) &= 0 \qquad \text{for $2\leq m \leq n$}. %
    \end{align}
    The containment is an equality if the $p$-adic period conjecture (\Cref{conj:period-conjecture}) holds for the primes dividing $p$.
\end{thm}

\Cref{mainthm:integral-points-equations} shows a typical structure for many of our results: we unconditionally obtain equations which hold on the polylogarithmic Chabauty--Kim locus, and thereby on the $S$-integral points, while we need to assume a $p$-adic period conjecture, asserting that certain $p$-adic periods are nonzero, in order to know that these equations precisely cut out the locus. We obtain both positive and negative results in this paper, i.e.\ we verify Kim's Conjecture in many cases, while in others we show that the polylogarithmic quotient is insufficient. Determining the solution set to the equations from \Cref{mainthm:integral-points-equations}, which notably do not depend on~$K$, we find: 

\begin{thm}[Loci for integral points over imaginary quadratic fields]
    \label{mainthm:imagquad-integral-locus}
    Let $K$ be an imaginary quadratic field, let $S=\emptyset$, and let $p$ be a rational prime which splits in~$K$. If $K=\bQ(\zeta_3)$, then Kim's Conjecture holds for the polylogarithmic quotient in depth $p-3$. If $K\neq\bQ(\zeta_3)$ and $p\not\equiv1\bmod3$, it holds in depth~1.  If $K\neq\bQ(\zeta_3)$ and $p\equiv1\bmod3$, then, assuming the $p$-adic period conjecture,
    \[
        X(\cO_K)=\emptyset
        \subsetneq X(\mathcal{O}_K\otimes\mathbb{Z}_p)_{\emptyset,\PL,\infty}
        =\{(\zeta_6,\zeta_6^{-1}),(\zeta_6^{-1},\zeta_6)\},
    \]
    so the polylogarithmic quotient is insufficient to verify Kim's Conjecture.
\end{thm}

\Cref{mainthm:imagquad-integral-locus} is the combination of \Cref{integral-points-locus}, Corollaries~\ref{kims-conjecture-for-Qzeta3} and \ref{kims-conjecture-p-not-1-mod-3}, and \Cref{integral-points-kims-conjecture-failure}. 

An example over~$\bQ$ where the polylogarithmic quotient is insufficient was already found by Corwin and Dan-Cohen \cite{CDC:polylog1}. As a remedy, they proposed to consider an intermediate ``$S_3$-symmetrised'' locus
\[
    X(\cO_K\otimes\bZ_p)_{S,n}
    \subseteq X(\cO_K\otimes\bZ_p)_{S,\PL,n}^{S_3}
    \subseteq X(\cO_K\otimes\bZ_p)_{S,\PL,n},
\]
defined as the largest subset of $X(\cO_K\otimes\bZ_p)_{S,\PL,n}$ which is stable under the natural action of $S_3=\Aut(X)$. However, the set $\{(\zeta_6,\zeta_6^{-1}),(\zeta_6^{-1},\zeta_6)\}$ is already $S_3$-stable, so nothing is gained by $S_3$-symmetrisation and \Cref{mainthm:imagquad-integral-locus} provides infinitely many counterexamples to the ``$S_3$-symmetrised polylogarithmic Chabauty--Kim Conjecture'' \cite[Conjecture~8.2.1]{ishaidancohen_2020_mixed}.

We also obtain negative results for non-empty~$S$ and larger fields~$K$:

\begin{thm}[{Roots-of-unity obstruction, \Cref{cor:CMfails}}]
\label{mainthm:roots-of-unity}
Let $K$ be a number field containing an imaginary quadratic field and suppose that $S$ contains all primes of~$K$ above a rational prime~$l$. Let $p\neq l$ split completely in~$K$, let $\fp\mid p$, and let $X(\cO_{\fp})_{S,\PL,\infty}$ be the projection of $X(\cO_K \otimes \bZ_p)_{S,\PL,\infty}$ onto $X(\cO_{\fp})$.  Assume the $p$-adic period conjecture. Then
\[
    \mu(\bZ_p)\smallsetminus\{1\}
    \subseteq X(\cO_{\fp})_{S,\PL,\infty}.
\]
In particular, for all sufficiently large~$p$, the locus $X(\cO_K \otimes \bZ_p)_{S,\PL,\infty}$ contains points not coming from $X(\cO_{K,S})$, so Kim's Conjecture fails for the polylogarithmic quotient in infinite depth.
\end{thm}

If $K$ is itself an imaginary quadratic field, a stronger version of \Cref{mainthm:roots-of-unity} states that $X(\cO_K\otimes\bZ_p)_{S,\PL,\infty}$ contains every pair $(\zeta,\zeta^{-1})$ with $1\neq\zeta\in\mu(\bZ_p)$; see \Cref{thm:roots-of-unity}. The general statement follows from this by functoriality under field extension (\Cref{field-extension-functoriality}). 

We also study the case where $K$ is an imaginary quadratic field and $S$ contains a single prime of~$K$. The sets of $S$-integral points for all such $K$ and $S$ are given in \Cref{S=1-stable-solutions}. The behaviour of the Chabauty--Kim loci depends on whether $S$ is stable under complex conjugation. In each case, we derive equations defining the polylogarithmic Chabauty--Kim locus $X(\cO_K \otimes \bZ_p)_{S,\PL,n}$ in depth $n \leq 4$ inside $X(\cO_K \otimes \bZ_p) \cong X(\bZ_p) \times X(\bZ_p)$, see Theorems~\ref{S=1-stable-equations} and~\ref{S=1-unstable-equations}. They are given by polynomials in $\log(z_i)$ and $L_m(z_i)$ for $i=1,2$ and $1 \leq m \leq n$, where $L_m(z)$ denotes the modified polylogarithms
\[
    L_m(z)=\sum_{k=0}^{m-1}\frac{B_k}{k!}\log(z)^k\Li_{m-k}(z),
\]
and $B_k$ denotes the $k$-th Bernoulli number with $B_1=-1/2$. The coefficients of these polynomials are certain $p$-adic periods which are difficult to determine in practice. However, in several cases where we know sufficiently many $S$-integral points, we manage to make these equations completely explicit and compute their solution sets:

\begin{thm}[Loci over imaginary quadratic fields, $S$ of size one]
    \label{mainthm:imagquad-singleton}
    The following statements hold.
    \begin{enumerate}[label=(\alph*)]
        \item \label{item:Qzeta3-Qi-depth2}
        Kim's Conjecture holds in depth~$2$ in the following two cases:
        \[
            (K,S,p)=\bigl(\bQ(\sqrt{-2}),\{(\sqrt{-2})\},3\bigr)
            \quad\text{or}\quad
            (K,S,p)=\bigl(\bQ(i),\{(1-i)\},5\bigr).
        \]

        \item \label{item:Qi-Qzeta3-depth4}
        Let $(K,S)$ be either $(\bQ(i),\{(1-i)\})$ or $(\bQ(\zeta_3),\{(1-\zeta_3)\})$. For every prime $p<2{,}000$ which splits in~$K$, we have
        \begin{align*}
            X(\cO_K\otimes\bZ_p)_{S,\PL,4}
            &=X(\cO_{K,S})\cup
            \{(\zeta,\zeta^{-1}):1\neq\zeta\in\mu(\bZ_p)\}.
            \\
            X(\cO_K\otimes\bZ_p)_{S,\PL,4}^{S_3}
            &=
            \begin{cases}
                X(\cO_{K,S})\cup
                \{(\zeta,\zeta^{-1}):\zeta\in\mu_6^{\prim}(\bZ_p)\},
                    & K=\bQ(i),\\[1mm]
                X(\cO_{K,S}),
                    & K=\bQ(\zeta_3).
            \end{cases}
        \end{align*}
        For $K=\bQ(i)$, the two displayed equalities are optimal in view of the roots-of-unity obstruction (\Cref{mainthm:roots-of-unity}). For $K=\bQ(\zeta_3)$, Kim's Conjecture holds for the full quotient in depth~$4$ for all $p < 2{,}000$; for $p=7$, it already holds for the polylogarithmic quotient in depth~$4$.
        \item \label{item:imagquad-unstable}
        Let $K$ be an imaginary quadratic field and assume $S=\{\fl\}$ is not stable under complex conjugation. Then $X(\cO_K\otimes\bZ_p)_{S,\PL,1}$ is finite for every split prime~$p$ not divisible by~$\fl$; moreover, the $S$-Selmer Section Conjecture \cite[Conj.~1.6]{BKL:chabauty-kim-sc} holds for $(X,S)$.
    \end{enumerate}
\end{thm}

Part~\ref{item:Qzeta3-Qi-depth2} follows from \Cref{thm:locusforimquadSsingle} and the corollary following it. Part~\ref{item:Qi-Qzeta3-depth4} combines \Cref{S=1-stable-depth4} with \Cref{Qzeta3-kim-conjecture}, and its equalities were verified computationally for all split primes in the stated range. The finiteness assertion in part~\ref{item:imagquad-unstable} is \Cref{S=1-finiteness}, whose proof uses the Ax--Schanuel theorem; the Selmer Section Conjecture assertion, whose proof follows the strategy of \cite{BKL:chabauty-kim-sc} and involves determining refined Chabauty--Kim loci for infinitely many primes~$p$, is \Cref{selmer-section-conjecture}. These instances of the $S$-Selmer Section Conjecture are not new, see \Cref{rem:selmer-section-conjecture}.

We now turn to the case of a real quadratic field. For $S = \emptyset$, we derive equations for $X(\cO_K \otimes \bZ_p)_{\emptyset,\PL,n}$ for arbitrary~$n$, see Theorems~\ref{real-integral-equations} and~\ref{realquad:higher-depth}. In the case that $S$ contains a single prime which is Galois-stable, depth~4 equations are derived in \Cref{thm:realquadS=1eqs}; depth~2 equations in the Galois-unstable case are derived in \Cref{thm:realquad_unstable}. We manage to compute Chabauty--Kim loci by solving these equations in several cases:

\begin{thm}[Real quadratic fields]
\label{mainthm:real-quadratic}
Let $K$ be a real quadratic field and let $p<2{,}000$ split in~$K$.  Then
\[
X(\cO_K\otimes\bZ_p)_{\emptyset,2}
=
\begin{cases}
S_3.\!\left(\dfrac{-1+\sqrt5}{2},\dfrac{-1-\sqrt5}{2}\right),
    & p\equiv\pm1\bmod5,\\[2mm]
\emptyset, & p\not\equiv\pm1\bmod5.
\end{cases}
\]
Consequently, Kim's Conjecture holds in depth~2 for $K=\bQ(\sqrt5)$ and every such~$p$, and for every real quadratic $K\neq\bQ(\sqrt5)$ when $p\not\equiv\pm1\bmod5$.  If $K\neq\bQ(\sqrt5)$ and $p\equiv\pm1\bmod5$, Kim's Conjecture does not hold in depth~2.  For $K=\bQ(\sqrt2)$ and $S=\emptyset$, Kim's Conjecture holds in depth~4 for every split $p<2{,}000$; it also holds in depth~4 for $S=\{(\sqrt2)\}$ and every split $p<100$.
\end{thm}

The description of the depth-2 locus for $S = \emptyset$ is \Cref{real-integral-depth2-locus}; the depth-4 computations for $\bQ(\sqrt2)$ are Theorems~\ref{thm:Qsqrt2SemptyKimholds} and~\ref{thm:Qsqrt2S=1Kimholds}.

The following table summarises the results concerning Kim's Conjecture established in Theorems~\ref{mainthm:imagquad-integral-locus}, \ref{mainthm:imagquad-singleton}, and~\ref{mainthm:real-quadratic}, together with the roots-of-unity obstruction in Theorem~\ref{mainthm:roots-of-unity}.

\begin{table}[ht]
\centering
\footnotesize
\renewcommand{\arraystretch}{1.15}
\setlength{\tabcolsep}{2.75pt}
\begin{tabular}{@{}lllcccc@{}}
\hline
$K$ & $S$ & condition on $p$ & depth & $\Pi_{\PL}$ & $S_3$-sym. & $\Pi$ \\
\hline
$\bQ(\zeta_3)$ & $\emptyset$ & $p\equiv1\bmod3$ & $p-3$ & $\checkmark$ & $\checkmark$ & $\checkmark$ \\
imag.\ quad. $\neq\bQ(\zeta_3)$ & $\emptyset$ & $p\equiv1\bmod3$ & $p-3$ & $\times^{*}$ & $\times^{*}$ & \\
imag.\ quad. $\neq\bQ(\zeta_3)$ & $\emptyset$ & $p\not\equiv1\bmod3$ & $1$ & $\checkmark$ & $\checkmark$ & $\checkmark$ \\
$\bQ(\sqrt{-2})$ & $\{(\sqrt{-2})\}$ & $p=3$ & $2$ & $\checkmark$ & $\checkmark$ & $\checkmark$ \\
$\bQ(i)$ & $\{(1-i)\}$ & $p=5$ & $2$ & $\checkmark$ & $\checkmark$ & $\checkmark$ \\
$\bQ(i)$ & $\{(1-i)\}$ & $5<p<2{,}000,\ p\not\equiv1\bmod3$ & $4$ & $\times^{**}$ & $\checkmark$ &  $\checkmark$ \\
$\bQ(i)$ & $\{(1-i)\}$ & $5<p<2{,}000,\ p\equiv1\bmod3$ & $4$ & $\times^{**}$ & $\times^{*}$ &  \\
$\bQ(\zeta_3)$ & $\{(1-\zeta_3)\}$ & $p=7$ & $4$ & $\checkmark$ & $\checkmark$ & $\checkmark$ \\
$\bQ(\zeta_3)$ & $\{(1-\zeta_3)\}$ & $7<p<2{,}000$ & $4$ & $\times^{**}$ & $\checkmark$ & $\checkmark$ \\
\hline
$\bQ(\sqrt5)$ & $\emptyset$ & $p<2{,}000$ & $2$ & $\checkmark$ & $\checkmark$ & $\checkmark$ \\
real quad. $\neq\bQ(\sqrt5)$ & $\emptyset$ & $p<2{,}000$, $p\not\equiv\pm1\bmod5$ & $2$ & $\checkmark$ & $\checkmark$ & $\checkmark$ \\
$\bQ(\sqrt2)$ & $\emptyset$ & $p<2{,}000$ & $4$ & $\checkmark$ & $\checkmark$ & $\checkmark$ \\
$\bQ(\sqrt2)$ & $\{(\sqrt2)\}$ & $p<100$ & $4$ & $\checkmark$ & $\checkmark$ & $\checkmark$ \\
\hline
$K \supseteq $ imag.\ quad. & $S \supseteq \{\fl \mid l\}$ & $\mu_{p-1} \not \subseteq K$ & $\infty$ & $\times$ & & \\
\hline
\end{tabular}
\caption{Kim's Conjecture in the cases treated in this paper. The columns $\Pi_{\PL}$, $S_3$-sym., and $\Pi$ concern the polylogarithmic quotient, its $S_3$-symmetrisation, and the full quotient, respectively. A checkmark indicates that the conjecture holds, a cross that it fails. The symbol $\times^*$ means that it fails but the only additional points in the locus are $(\zeta_6,\zeta_6^{-1})$ and $(\zeta_6^{-1},\zeta_6)$; the symbol $\times^{**}$ means that the only additional points are $(\zeta,\zeta^{-1})$ with $\zeta \neq 1$ a root of unity in~$\bZ_p$. A blank entry means that no conclusion is asserted. Throughout, $p$ is assumed to split in~$K$. The entries in the second and last row are conditional on the $p$-adic period conjecture.}
\label{tab:intro-kim}
\end{table}

We also derive equations for \emph{refined Chabauty--Kim loci} in the sense of Betts--Dogra \cite{BD:refined} and compute the loci in several cases, see Theorems~\ref{thm:refinedkim-functionsforS=1}, \ref{S=1-refined1-equations}, \ref{S=1-refined0-equations}, \ref{S=1-refined0-depth1and2}, \ref{Qzeta3-refined0-depth3and4}, \ref{Qi-refined0-depth3and4}, \ref{real-refined1-equations}, \ref{thm:realquad_unstable}.

Code in SageMath \cite{sagemath} for computing polylogarithmic Chabauty--Kim loci over number fields is available at \url{https://github.com/martinluedtke/PolylogNF}.

\subsection{Polylogarithmic Chabauty--Kim theory over number fields}
\label{sec:intro-framework}

We now give a brief summary of the theory developed in Sections~\ref{sec:ck-loci}--\ref{sec:periods}. As before, let $K$ be a number field, let $S$ be a finite set of primes of~$K$, and let $p$ be a rational prime not divisible by a prime in~$S$.  
We work with the motivic formulation of the Chabauty--Kim method developed by Hadian~\cite{Hadian2011} and Dan-Cohen and Wewers \cite{DCW:explicitCK, ishaidancohen_2016_mixed}. Here, the pro-unipotent étale fundamental group of the thrice-punctured line is replaced by the motivic fundamental group $\pi_1^{\mot}(X,0)$ of Deligne and Goncharov \cite{deligne-goncharov}. The equivalence between the two formulations is established by a motivic--étale comparison theorem, see Theorem~\ref{motivic-etale-comparison}.

Write $\MT(\cO_{K,S},\bQ)$ for the category of $\bQ$-linear mixed Tate motives over~$\cO_{K,S}$. It carries a canonical $\bQ$-linear fibre functor~$\omega$ whose base change to~$K$ is the de Rham realisation functor. Let $G_S^{\MT}$ be the Tannaka group of $\MT(\cO_{K,S},\bQ)$ with respect to~$\omega$, and let $U_S^{\MT}$ be its unipotent part. The \emph{motivic Selmer scheme} $\Sel_{S,\Pi}^{\mot}(X)$ associated to a motivic quotient $\pi_1^{\mot}(X,0) \twoheadrightarrow \Pi$ is defined as the moduli space of $G_S^{\MT}$-equivariant $\Pi^{\omega}$-torsors.
For the depth-$n$ polylogarithmic quotient $\pi_1^{\mot}(X,0) \twoheadrightarrow \Pi_{\PL,n}$, there are canonical identifications
\[
    \Sel_{S,\PL,n}^{\mot}(X)
    \cong
    \Hom_{\gr}\bigl(\Lie(U_S^{\MT}),\Lie(\Pi_{\PL,n}^{\omega})\bigr),
    \qquad
    \log\colon \Pi_{\PL,n}^{\dR}\xrightarrow{\sim}\Lie(\Pi_{\PL,n}^{\dR}),
\]
where $\Hom_{\gr}$ means weight-graded Lie algebra homomorphisms. Under these identifications, the localisation map $\loc_p$ in the Chabauty--Kim diagram~\eqref{eq:ck-diagram-intro} is the product 
\begin{equation}
\label{eq:loc-evaluation}
    \Hom_{\gr}(\Lie(U_S^{\MT}),\Lie(\Pi_{\PL,n}^{\omega}))
    \longrightarrow
    \prod_{\fp\mid p}\Lie(\Pi_{\PL,n}^{\dR})_{K_{\fp}},
    \qquad
    c\longmapsto(c(\eps_{\fp}))_{\fp\mid p}.
\end{equation}
of the evaluation maps $\ev_{\eps_{\fp}}$ at certain $\fp$-adic period elements $\eps_{\fp}\in\Lie(U_S^{\MT})(K_{\fp})$ for $\fp \mid p$, which were constructed by Chatzistamatiou--Ünver \cite{chatzistamatiou-unver:p-adic_periods}. 

The Selmer scheme $\Sel_{S,\PL,n}^{\mot}(X)$ is an affine space with an explicit set of coordinates relative to a choice of free generators of $\Lie(U_S^{\MT})$.  Using these coordinates on the Selmer scheme and a natural set of coordinates on $\Lie(\Pi_{\PL,n}^{\omega})$, we derive explicit polynomials describing the evaluation map
\[ \ev_{\eps}\colon \Sel_{S,\PL,n}^{\mot}(X)_R \to \Lie(\Pi_{\PL,n}^{\omega})_R \]
at an arbitrary element $\eps \in \Lie(U_S^{\MT})(R)$ for a $\bQ$-algebra~$R$, see \Cref{cocycle-evaluation-map}. The coefficients of these polynomials depend on $\eps$ in the following way. All graded Lie algebra homomorphisms factor through the so-called \emph{Goncharov quotient} $\Lie(U_S^{\MT}) \twoheadrightarrow \Lie(U_S^{\MT})_{\Gon}$ (\Cref{def:goncharov-quotient}). Using the free Lie algebra generators for $U_S^{\MT}$ again, we describe a basis of $\Lie(U_S^{\MT})_{\Gon}$ in \Cref{lie-algebra-element-expansion}. The coordinates of $\eps$ with respect to this basis are called its ``Goncharov coefficients'' (\Cref{def:goncharov-representation}). Now $\ev_{\eps}$ is given by polynomials whose coefficients are explicit $\bZ$-linear combinations of the Goncharov coefficients of $\eps$.

In order to make the localisation map~\eqref{eq:loc-evaluation} explicit, we need to study the Goncharov coefficients of the $\fp$-adic period elements $\eps_{\fp}$ for $\fp\mid p$. These are elements of~$K_{\fp}$ in the image of the \emph{$\fp$-adic period homomorphism} (\Cref{def:motivic-periods})
\[
    \per_{\fp}\colon \cO(U_S^{\MT}) \to K_{\fp}.
\]
The ring $\cO(U_S^{\MT})$ is called the ring of (unipotent mixed Tate) motivic periods, and the $\fp$-adic period conjecture is the statement that $\per_{\fp}$ is injective. By lifting $\fp$-adic periods to the motivic level, we are able to exploit their functorial behaviour with respect to field homomorphisms. More precisely, if $\sigma\colon K\to K'$ is a homomorphism of number fields, and $\sigma_* S$ denotes the set of primes $\fp'$ of~$K'$ such that $\sigma^*\fp' \in S$, we have surjective resp.\ injective homomorphisms
\begin{align*}
    \sigma^*\colon U_{\sigma_* S}^{\MT} &\twoheadrightarrow U_S^{\MT},\\
    \sigma^*\colon \Lie(U_{\sigma_* S}^{\MT}) &\twoheadrightarrow \Lie(U_S^{\MT}),\\
    \sigma_*\colon \cO(U_{\sigma_* S}^{\MT}) &\hookrightarrow \cO(U_S^{\MT}),
\end{align*}
see Sections~\ref{sec:galois-action-on-selmer-scheme} and~\ref{sec:galois-action-on-periods}.
In particular, if $K/\bQ$ is Galois, the group $\Gal(K/\bQ)$ acts on $U_S^{\MT}$, its Lie algebra, and on the motivic period ring $\cO(U_S^{\MT})$, and the $\fp$-adic period elements $\eps_{\fp}$ for the various $\fp \mid p$ are permuted by the Galois action. As a consequence, the Goncharov coefficients of a single $\eps_{\fp}$ determine all others.  

Determining the $\fp$-adic periods appearing as Goncharov coefficients of the $\eps_{\fp}$ is a difficult problem even after exploiting the Galois action. In half-weight $1$ they are always given by $\fp$-adic logarithms (\Cref{tau-coeffs}) but in order to compute the higher-weight coefficients one usually needs a sufficiently large supply of known $S$-integral points. The general procedure is to derive equations for the Chabauty--Kim loci with $\fp$-adic periods appearing as unknown coefficients, which one can then try to solve for by plugging in known solutions.

\subsection{Organisation of the paper}

This paper is organised in two parts.  Sections~\ref{sec:ck-loci}--\ref{sec:periods} develop the theory, while Sections~\ref{sec:integral-points-imag}--\ref{sec:real-quadratic} contain applications. 
In Section~\ref{sec:ck-loci} we recall the definition of Chabauty--Kim loci in the number field setting. In Section \ref{sec:motivic-ck}, we discuss the motivic Chabauty--Kim method, the motivic--étale comparison and functoriality. In Section~\ref{sec:ck-theory-for-polylog-quotient} we specialise to the polylogarithmic quotient and derive a formula for the localisation map. Section~\ref{sec:periods} contains a discussion of motivic and $p$-adic periods and their Galois action. 
In Sections~\ref{sec:integral-points-imag}--\ref{sec:S-of-size-one}, we apply the theory to imaginary quadratic fields, including the roots-of-unity obstruction and refined loci. Finally, \Cref{sec:real-quadratic} treats real quadratic fields.

\subsection*{AI disclosure}
The authors used GPT and other large language models to assist in writing some of the SageMath code for the computations in this paper. GPT-5.5 was also used to discover the equations and coefficient formulas in \Cref{realquad:higher-depth} and to assist in developing the proof. All AI-assisted output was checked and verified by the authors.

\subsection*{Acknowledgements}
We are particularly grateful to Minhyong Kim for proposing this project. We thank Netan Dogra for clarifying the definition of the Chabauty--Kim locus in \cite{dogra_2023_unlikely}, which helped us distinguish between two possible definitions of the single-prime locus (see \Cref{rem:alternative-single-place-locus}), and for suggesting the use of Ax--Schanuel in the proof of \Cref{depth-1-finiteness}. We also thank Ishai Dan-Cohen for helpful discussions about the construction of bases of the extension spaces $E_n(Z)$; see \Cref{sec:generators}.

X.\ L. is supported by a Carnegie PhD Scholarship from the Carnegie Trust for the Universities of Scotland. M.\ L. is supported by a Minerva Fellowship of the Minerva Stiftung Gesellschaft für die Forschung mbH and also acknowledges support through a guest postdoc fellowship at the Max Planck Institute for Mathematics in Bonn and through an NWO Grant, project number VI.Vidi.192.106.

\section{Chabauty--Kim loci and Kim's Conjecture over number fields}
\label{sec:ck-loci}

We define Chabauty--Kim loci and state two versions of Kim's Conjecture in the number field setting. Our definition of the Selmer scheme uses Bloch--Kato local conditions as in \cite{kim:albanese}; we later also consider refined Selmer schemes as defined in \cite{BD:refined} (over~$\bQ$) and \cite{dogra_2023_unlikely} (over general number fields). We restrict the discussion to the case of the thrice-punctured line for simplicity but remark that everything in this section admits straightforward generalisations to general hyperbolic curves.

Let $K$ be a number field, let $S$ be a finite set of primes of~$K$, and let $X = \bP^1 \smallsetminus \{0,1,\infty\}$ be the thrice-punctured line over $\cO_{K,S}$. Fix an auxiliary prime~$p$ not divisible by a prime in~$S$. Let $\overline{K}/K$ be an algebraic closure and $G_K = \Gal(\overline{K}/K)$. Fix a $G_K$-equivariant quotient $\pi_1^{\et,\bQ_p}(X_{\overline{K}}, 0) \twoheadrightarrow \Pi^{\et}$ of the $\bQ_p$-prounipotent étale fundamental group of $X_{\overline{K}}$. Here, $0$ in $\pi_1^{\et,\bQ_p}(X_{\overline{K}}, 0)$ is shorthand for the tangential base point $\vec{1}_0$ of~$X$ given by the tangent vector $\partial/\partial t$ at~$0$, where $t$ is the standard coordinate on $\bA^1$. We have the following \emph{Chabauty--Kim diagram:}
\begin{equation}
\label{eq:CK-diagram}
\begin{tikzcd}[column sep=huge]
    X(\cO_{K,S}) \rar[hook] \dar["j_S"] & X(\mathcal{O}_{K}\otimes_{\mathbb{Z}}\mathbb{Z}_p)=\prod_{\fp\mid p} X(\mathcal{O}_\fp) \dar["j_p = \prod_\fp j_\fp"] \\
    \rH^1_{f,S}(G_K, \Pi^{\et}) \rar["{\loc_p} = \prod_{\fp} \loc_{\fp}"]  & \prod_{\fp\mid p} \rH^1_f(G_\fp, \Pi^{\et})
\end{tikzcd}
\end{equation}
Here, the objects in the bottom row are affine $\bQ_p$-schemes and the localisation map $\loc_p$ is algebraic. The object $\rH^1_{f,S}(G_K, \Pi^{\et})$ is called the (global) \emph{Selmer scheme}. It parametrises $G_K$-equivariant $\Pi^{\et}$-torsors which are crystalline at places dividing~$p$ and unramified at all places outside~$S$ not dividing~$p$. (This is the special case of \cite[Definition~3.2.2]{betts:weight-filtrations} when applied to the Selmer structure described in Example~3.2.3 of loc.\ cit.)
The objects $\rH^1_f(G_\fp, \Pi^{\et})$ are the \emph{local Selmer schemes} which parametrise crystalline $G_\fp$-equivariant $\Pi^{\et}$-torsors, where $G_\fp \subseteq G_K$ is the local Galois group of a prime dividing~$p$. The vertical maps $j_S$ and $j_\fp$ are the non-abelian Kummer maps which take a global resp.\ local point $x$ of~$X$ to the $\bQ_p$-point representing the path torsor $\pi_1^{\et,\bQ_p}(X_{\overline{K}}; 0, x)$, pushed out to~$\Pi^{\et}$, equipped with the natural $G_K$- resp.\ $G_\fp$-action. When $\Pi^{\et}$ is finite-dimensional, the global and local Selmer scheme are of finite type over~$\bQ_p$. 

\begin{defn}
\label{def:ck-locus}
	The \emph{Chabauty--Kim locus} for the quotient $\Pi^{\et}$ is defined as
	\[ X(\cO_K \otimes_{\bZ} \bZ_p)_{S,\Pi^{\et}}\coloneq j_p^{-1}(\loc_p(\rH^1_{f,S}(G_K, \Pi^{\et})))\subseteq X(\cO_K\otimes_{\mathbb{Z}}\mathbb{Z}_p), \]
    in other words, as the set of all tuples $(x_\fp)_{\fp \mid p}$ for which $(j_\fp(x_\fp))_{\fp \mid p}$ lies in the scheme-theoretic image of the Selmer scheme under the localisation map.
    The \emph{Chabauty--Kim locus at a place} $\fp \mid p$ for the quotient $\Pi^{\et}$ is 
    \[ X(\cO_{\fp})_{S,\Pi^{\et}}\coloneq \mathrm{pr}_\fp(X(\cO_K \otimes_{\bZ} \bZ_p)_{S,\Pi^{\et}})\subseteq X(\mathcal{O}_\fp), \]
    where $\mathrm{pr}_\fp\colon  X(\mathcal{O}_{K}\otimes_{\mathbb{Z}}\mathbb{Z}_p)=\prod_{\fp'\mid p} X(\mathcal{O}_{\fp'})\to X(\mathcal{O}_{\fp})$ is the projection to the $\fp$-component.
\end{defn}

When $\Pi^{\et} = \Pi^{\et}_n$ is the maximal $n$-step nilpotent quotient of $\pi_1^{\et,\bQ_p}(X_{\overline{K}}, 0)$, where $1 \leq n \leq \infty$, the Chabauty--Kim loci are simply denoted by $X(\cO_K \otimes_{\bZ} \bZ_p)_{S,n}$ resp.\ $X(\cO_{\fp})_{S,n}$, and~$n$ is called the \emph{depth}. In this paper we mostly work with the \emph{polylogarithmic} depth-$n$ quotient $\Pi^{\et}_{\PL,n}$ (see Definition \ref{def:PL}), whose associated Chabauty--Kim loci we denote by $X(\cO_K \otimes_{\bZ} \bZ_p)_{S,\PL,n}$ resp.\ $X(\cO_{\fp})_{S,\PL,n}$. 

By the commutativity of the Chabauty--Kim diagram~\eqref{eq:CK-diagram}, we have inclusions
\begin{align}
    \label{eq:kim-inclusion-full}
    X(\cO_{K,S}) &\subseteq X(\cO_K \otimes \bZ_p)_{S,\Pi^{\et}}, \\
    \label{eq:kim-inclusion-single}
    X(\cO_{K,S}) &\subseteq X(\cO_\fp)_{S,\Pi^{\et}}. %
\end{align}

Kim \cite{kim_2005_the} showed (in the case $K = \bQ$) that the Chabauty--Kim loci are finite for sufficiently large fundamental group quotients~$\Pi^{\et}$, thereby giving a new proof of the Siegel--Mahler theorem, i.e.\ finiteness of $X(\cO_{K,S})$. Kim's Conjecture \cite[§3.1 \& §8.1]{BDCKW}, generalised to the number field setting, states that the inclusions~\eqref{eq:kim-inclusion-full}, \eqref{eq:kim-inclusion-single} become even equalities for sufficiently large $\Pi^{\et}$.

\begin{conj}[Kim's Conjecture]
\label{conj:kim-full}
    For sufficiently large quotients $\pi_1^{\et,\bQ_p}(X_{\overline{K}},0)\twoheadrightarrow \Pi^{\et}$, we have
    \[ X(\cO_K \otimes_{\bZ} \bZ_p)_{S,\Pi^{\et}} = X(\mathcal{O}_{K,S}). \]
\end{conj}

\begin{conj}[Kim's Conjecture, single prime version]
\label{conj:kim-single}
    Let $\fp$ be a prime of~$K$ dividing~$p$. For sufficiently large fundamental group quotients $\pi_1^{\et,\bQ_p}(X_{\overline{K}},0) \twoheadrightarrow \Pi^{\et}$, we have
    \[ X(\cO_{\fp})_{S,\Pi^{\et}} = X(\mathcal{O}_{K,S}). \]
\end{conj}

Note that \Cref{conj:kim-full} implies \Cref{conj:kim-single} for all $\fp \mid p$. The converse is not true since an element of $X(\cO_K \otimes \bZ_p)_{S,\Pi^{\et}} \subseteq \prod_{\fp \mid p} X(\cO_\fp)$ could have all components lying in $X(\cO_{K,S})$ without the whole tuple being in the image of the diagonal embedding $X(\cO_{K,S}) \hookrightarrow \prod_{\fp \mid p} X(\cO_\fp)$. An example of this occurs for the depth-1 loci for $K = \bQ(\zeta_3)$, $S = \emptyset$, $p \equiv 1 \bmod 3$, see \Cref{Kim-conjecture-full-vs-single-place} below.

\begin{rem}
    \label{rem:alternative-single-place-locus}
    Another possible definition of the Chabauty--Kim locus at a single prime $\fp \mid p$ that has been considered in the literature is
    $$X(\cO_\fp)_{S,\Pi^{\et}}'\coloneq j_\fp^{-1}(\loc_{\fp}(\rH^1_{f,S}(G_K, \Pi^{\et})))\subseteq X(\mathcal{O}_\fp).$$
    It is clear that $X(\cO_{\fp})_{S,\Pi^{\et}}\subseteq X(\cO_\fp)_{S,\Pi^{\et}}'$, so this latter variant is weaker compared to the one we are using. The advantage is that finiteness of $X(\cO_\fp)_{S,\Pi^{\et}}'$ is guaranteed as soon as $\loc_{\fp}$ is non-dominant for a single prime~$\fp \mid p$. Kim used this to prove that $X(\cO_\fp)_{S,\PL,\infty}'$ is finite for any set of primes~$S$ when $K$ is totally real \cite[§3]{kim:tangential}. Moreover, in the case that $K$ is a real quadratic field and $S = \emptyset$, Dan-Cohen and Wewers have computed an equation defining the depth-2 locus $X(\cO_\fp)_{\emptyset,2}'$ \cite[§12.1]{DCW:explicitCK}. On the other hand, assuming a $p$-adic period conjecture, it is shown in \cite[Proposition~8.1]{ishaidancohen_2020_mixed} that when $K$ is \emph{not} totally real then for any set of primes $S$, the polylogarithmic locus in infinite depth $X(\cO_\fp)_{S,\PL,\infty}'$ is all of $X(\mathcal{O}_\fp)$ (except when $K$ is imaginary quadratic and $S =\emptyset$).\footnote{\cite[Proposition~8.1]{ishaidancohen_2020_mixed} is false when $K$ is imaginary quadratic and $S = \emptyset$. The proof in loc.\ cit. claims that the motivic Ext group $\Ext^1_{\MT(\cO_{K,S},\bQ)}(\bQ(0), \bQ(n))$ is nonzero for all $n \geq 1$ but for $n = 1$ this group is $\cO_{K,S}^\times \otimes \bQ$, which is trivial in the aforementioned case.} In contrast we will see that $X(\cO_\fp)_{S,\PL,2}$ is finite when $K$ is imaginary quadratic and $S$ is a Galois-stable singleton set (Theorem \ref{thm:locusforimquadSsingle}). This shows that the inclusion $X(\cO_{\fp})_{S,\Pi^{\et}}\subseteq X(\cO_\fp)_{S,\Pi^{\et}}'$ can be strict.
\end{rem}

We note the following functorial behaviour for Chabauty--Kim loci under field extensions for later use.

\begin{lemma}
    \label{field-extension-functoriality}
    Let $K \subseteq K'$ be an extension of number fields, let $S$ be a finite set of primes of~$K$ and let $S'$ be the set of all primes of~$K'$ lying over~$S$. Let $p$ be a rational prime not divisible by any prime in~$S$. Then for any $G_K$-equivariant quotient $\pi_1^{\et,\mathbb{Q}_p}(X_{\overline{K}},0) \twoheadrightarrow \Pi^{\et}$ we have the inclusions
    \begin{align*}
        X(\mathcal{O}_K \otimes \mathbb{Z}_p)_{S,\Pi^{\et}} &\subseteq X(\mathcal{O}_{K'} \otimes \mathbb{Z}_p)_{S',\Pi^{\et}},\\
        X(\cO_{\fp})_{S,\Pi^{\et}} &\subseteq X(\cO_{\fp'})_{S',\Pi^{\et}} \quad \text{for $\fp'\mid\fp\mid p$}.
    \end{align*}
\end{lemma}

\begin{proof}
    We have the following commutative diagram:
    \begin{center}
    \begin{tikzcd}[row sep=small]
         & & & X(\mathcal{O}_K \otimes \mathbb{Z}_p) \arrow[dd,hook] \arrow[dl] \\
         \rH^1_{f,S}(G_K, \Pi^{\et}) \rar[hook] \arrow[dd,dashed] & \rH^1(G_K, \Pi^{\et}) \rar \arrow[dd,"\res"] & \prod_{\mathfrak{p}\mid p} \rH^1(G_{\fp}, \Pi^{\et}) \arrow[dd,"\res"] & \\
         & & & X(\mathcal{O}_{K'} \otimes \mathbb{Z}_p) \arrow[dl] \\
         \rH^1_{f,S'}(G_{K'}, \Pi^{\et}) \rar[hook] & \rH^1(G_{K'}, \Pi^{\et}) \rar & \prod_{\mathfrak{p}'\mid p} \rH^1(G_{\fp'}, \Pi^{\et}) &
    \end{tikzcd}
    \end{center}
    The vertical restriction maps are those for the inclusions $G_{K'} \subseteq G_K$ and $G_{\fp'} \subseteq G_{\fp}$ where $\mathfrak{p}'\mid \mathfrak{p}\mid p$. The restriction map on the global side is a map of functors on $\bQ$-algebras. It suffices to show that it maps the Selmer subfunctor $\rH^1_{f,S}(G_K, \Pi^{\et})$ into $\rH^1_{f,S'}(G_{K'}, \Pi^{\et})$, for then we have the dashed morphism of $\bQ_p$-schemes in the diagram, implying the claimed inclusion of Chabauty--Kim loci. But for $\mathfrak{p}'\mid \mathfrak{p}$ with $\mathfrak{p}' \not \in S'$, we also have $\mathfrak{p} \not\in S$, and the restriction morphism of local Selmer schemes $\rH^1(G_{\fp}, \Pi^{\et}) \to \rH^1(G_{\fp'}, \Pi^{\et})$ maps unramified classes to unramified classes (if $\mathfrak{p}\nmid p$) and crystalline classes to crystalline classes (if $\mathfrak{p}\mid p$). Thus, the local Selmer conditions are preserved under the global restriction map.
\end{proof}

In this paper we also consider the \emph{refined} Chabauty--Kim method as introduced by Betts and Dogra \cite{BD:refined}. It is based on the definition of a refined Selmer scheme which is a closed subscheme of the original Selmer scheme but imposes stricter local conditions. More precisely, the \emph{refined Selmer scheme} $\Sel_{S,\Pi^{\et}}^{\min}(X)$ is the $\bQ_p$-scheme parametrising cohomology classes $\xi \in \rH^1(G_K,\Pi^{\et})$ such that
\begin{equation}
\label{eq:refined-local-condition}
    \loc_{\mathfrak{q}}(\xi)\in \begin{cases}
        j_{\mathfrak{q}}(X(\mathcal{O}_{\mathfrak{q}}))^{\mathrm{Zar}}, & \text{ if }\mathfrak{q}\notin S;\\
        j_{\mathfrak{q}}(X(K_{\mathfrak{q}}))^{\mathrm{Zar}}, & \text{ if }\mathfrak{q}\in S,
    \end{cases}
\end{equation}
for all primes~$\fq$ of~$K$, where $(-)^{\mathrm{Zar}}$ denotes Zariski closure inside $\rH^1(G_{\fq},\Pi^{\et})$.
This implies that $\loc_{\fq}(\xi)$ is unramified for $\fq \not\in S \cup \{\fp \mid p\}$ and $\loc_{\fp}(\xi)$ crystalline for $\fp \mid p$, so the refined Selmer scheme is a closed subscheme of the original one:
\begin{equation}
\label{eq:refined-selmer-scheme-is-subscheme}
    \Sel_{S,\Pi^{\et}}^{\min}(X) \subseteq \rH^1_{f,S}(G_K, \Pi^{\et}).
\end{equation}
The Kummer map $j_S\colon X(\cO_{K,S}) \to \rH^1_{f,S}(G_K,\Pi^{\et})$ factors through the refined Selmer scheme, so we still have a commutative diagram when $\rH^1_{f,S}(G_K,\Pi^{\et})$ is replaced with $\Sel_{S,\Pi^{\et}}^{\min}(X)$ in the Chabauty--Kim diagram \eqref{eq:CK-diagram}.

\begin{defn}
	The \emph{refined Chabauty--Kim locus} for the quotient $\Pi^{\et}$ is defined as
	\[ X(\cO_K \otimes_{\bZ} \bZ_p)_{S,\Pi^{\et}}^{\min}\coloneq j_p^{-1}(\loc_p(\Sel_{S,\Pi^{\et}}^{\min}(X))) \subseteq X(\cO_K\otimes_{\mathbb{Z}}\mathbb{Z}_p). \]
    The \emph{refined Chabauty--Kim locus at a place} $\fp \mid p$ for the quotient $\Pi^{\et}$ is 
    \[ X(\cO_{\fp})_{S,\Pi^{\et}}^{\min}\coloneq \mathrm{pr}_\fp(X(\cO_K \otimes_{\bZ} \bZ_p)_{S,\Pi^{\et}}^{\min})\subseteq X(\mathcal{O}_\fp). \]
\end{defn}

We have inclusions
\begin{align*}
    X(\cO_{K,S}) &\subseteq X(\cO_K \otimes_{\bZ} \bZ_p)_{S,\Pi^{\et}}^{\min} \subseteq X(\cO_K \otimes_{\bZ} \bZ_p)_{S,\Pi^{\et}},\\
    X(\cO_{K,S}) &\subseteq X(\cO_{\fp})_{S,\Pi^{\et}}^{\min} \subseteq X(\cO_{\fp})_{S,\Pi^{\et}},
\end{align*}
so the refined Chabauty--Kim method is at least as strong as the original one. Dogra \cite[Theorem~1.1\,(1)]{dogra_2023_unlikely} proved that the depth-$n$ refined Chabauty--Kim locus $X(\cO_K \otimes_{\bZ} \bZ_p)_{S,n}^{\min}$ is finite for $n \gg 0$, which provides a new proof of the Siegel--Mahler theorem over number fields.

We also have refined variants of Kim's Conjecture:

\begin{conj}[Refined Kim's Conjecture]
\label{conj:refined-kim-full}
    For sufficiently large fundamental group quotients $\pi_1^{\et,\bQ_p}(X_{\overline{K}},0) \twoheadrightarrow \Pi^{\et}$, we have
    \[ X(\cO_K \otimes_{\bZ} \bZ_p)_{S,\Pi^{\et}}^{\min} = X(\mathcal{O}_{K,S}). \]
\end{conj}

\begin{conj}[Refined Kim's Conjecture, single prime version]
\label{conj:refined-kim-single}
    Let $\fp$ be a prime of~$K$ dividing~$p$. For sufficiently large fundamental group quotients $\pi_1^{\et,\bQ_p}(X_{\overline{K}},0) \twoheadrightarrow \Pi^{\et}$, we have
    \[ X(\cO_{\fp})_{S,\Pi^{\et}}^{\min} = X(\mathcal{O}_{K,S}). \]
\end{conj}

It is clear that the refined variant of Kim's Conjecture is implied by the unrefined one. One reason for investigating \Cref{conj:refined-kim-single} is its relation with Grothendieck's Section Conjecture:

\begin{thm}[{\cite[Theorem~2.21]{BKL:chabauty-kim-sc}}]
\label{kim-implies-selmer-section-conjecture}
    Assume that $X(\cO_{\fp})_{S,\infty}^{\min} = X(\mathcal{O}_{K,S})$ holds for a Dirichlet-dense set of primes~$\fp$ of~$K$. Then $X$ satisfies the Section Conjecture for locally geometric $S$-integral Galois sections. 
\end{thm}

\Cref{kim-implies-selmer-section-conjecture} (which is valid for any hyperbolic curve) was used in \cite{BKL:chabauty-kim-sc} to verify the ``$S$-Selmer Section Conjecture'' in the case $K = \bQ$, $X = \bP^1 \smallsetminus \{0,1,\infty\}$, and $S = \{2\}$. We obtain new examples of this strategy over imaginary quadratic fields in this paper; see \Cref{selmer-section-conjecture}.

It can happen that there is a local obstruction to the existence of $S$-integral points: when there is a prime $\fq \not\in S$ with $X(\cO_{\fq}) = \emptyset$ then automatically $X(\cO_{K,S}) = \emptyset$. The refined Chabauty--Kim method captures this obstruction since the local condition~\eqref{eq:refined-local-condition} at~$\fq$ in the definition of the refined Selmer scheme is unsatisfiable, so that the refined Selmer scheme and all refined Chabauty--Kim loci are empty and \Cref{conj:refined-kim-full} holds vacuously. We have $X(\cO_{\fq}) = \emptyset$ if and only if $X(\bF_{\fq}) = \emptyset$, which happens if and only if the residue $\bF_{\fq}$ of~$\fq$ is $\bF_2$. Thus, in order to avoid this trivial situation in the refined Chabauty--Kim method we will often impose the following condition.

\begin{cond}
\label{cond:F2}
    $S$ contains all primes of~$K$ with residue field~$\bF_2$.
\end{cond}

The refined Chabauty--Kim method is only interesting when $S \neq \emptyset$. When $S = \emptyset$ and the non-degeneracy condition \ref{cond:F2} is satisfied (i.e.\ the residue fields of the even primes of~$K$ are not~$\bF_2$), then the refined Selmer scheme agrees with the non-refined one, so nothing is gained.

\section{Mixed Tate motivic Chabauty--Kim}
\label{sec:motivic-ck}

As before, let $K$ be a number field, let $S$ be a finite set of primes of~$K$, let $X = \bP^1 \smallsetminus \{0,1,\infty\}$ be the thrice-punctured line over $\cO_{K,S}$, and let $\vec{1}_0 = \partial t/t$ be the standard tangential base point at~$0$. In this section we describe the motivic formulation of the Chabauty--Kim method, which has been studied by Hadian, Dan-Cohen, Wewers, Corwin, Brown, and others \cite{Hadian2011,DCW:explicitCK,ishaidancohen_2016_mixed,brown:integral_points,ishaidancohen_2020_mixed,CDC:polylog1,CDC:polylog2,BKL:chabauty-kim-sc,motivic-selmer-scheme}. We also discuss the motivic--étale comparison theorem in the number field setting, the Lie algebra variant for semisimple fundamental group quotients, and the Galois action on Selmer schemes.

\subsection{The motivic Chabauty--Kim diagram}
\label{sec:motivic-ck-diagram}

Denote by $\MT(\cO_{K,S}, \bQ)$ the category of mixed Tate motives over $\cO_{K,S}$, which is defined to be the full subcategory of the category $\MT(K, \bQ)$ of mixed Tate motives over $K$ consisting of those which are unramified outside $S$ \cite[§1.6]{deligne-goncharov}. It carries various fibre functors, notably the \emph{canonical fibre functor} $\omega\colon \MT(\cO_{K,S}, \bQ)\to \Vect(\mathbb{Q})$ defined by
$$\omega(M) \coloneqq M^{\omega} \coloneqq \bigoplus_{n\in\mathbb{Z}}\omega_n(M),\quad  \omega_n(M) \coloneqq \Hom(\mathbb{Q}(-n),\mathrm{gr}_{2n}^W(M)),$$
the \emph{de Rham fibre functor} $\omega^{\dR}\colon \MT(\cO_{K,S}, \bQ)\to \Vect(K)$, $M \mapsto M^{\dR}$ given by de Rham realisation, and for a prime~$p$ the \emph{$p$-adic étale fibre functor} $\omega^{\et} \colon \MT(\cO_{K,S}, \bQ) \to \Vect(\bQ_p)$, $M \mapsto M^{\et}$ given by étale realisation. There is a natural isomorphism $\omega\otimes_{\mathbb{Q}} K\cong \omega^{\dR}$ \cite[Proposition~2.10]{deligne-goncharov}.
The canonical fibre functor $\omega$ makes the category $\MT(\cO_{K,S}, \bQ)$ a neutral $\bQ$-linear Tannakian category. The corresponding Tannakian fundamental group is denoted by $G_S^{\MT}$. It has a semi-direct product decomposition 
\[ G_S^{\MT}=U_S^{\MT}\rtimes \mathbb{G}_m \]
with $U_S^{\MT}$ pro-unipotent \cite[§2.1]{deligne-goncharov}. The \emph{motivic fundamental group} $\pi_1^{\mot}(X,0)$ of $X$ at the tangential base point $\vec{1}_0$ constructed by Deligne--Goncharov is a pro-unipotent group in $\MT(\cO_{K,S}, \bQ)$, meaning that its ring of functions is an algebra object of the ind-category of $\MT(\cO_{K,S}, \bQ)$. Its de Rham realisation is canonically isomorphic to the unipotent de Rham fundamental group $\pi_1^{\dR}(X,0)$ of~$X$, and its $p$-adic étale realisation is canonically isomorphic to the $p$-adic étale fundamental group $\pi_1^{\et,\bQ_p}(X_{\overline{K}}, 0)$. For any quotient $\pi_1^{\mot}(X,0) \twoheadrightarrow \Pi$ in $\MT(\cO_{K,S},\bQ)$, the motivic Selmer scheme is defined as follows.

\begin{defn}[{\cite[Definition~2.4]{motivic-selmer-scheme}}]
    \label{defn:motivic-selmer-scheme}
    The \emph{motivic Selmer scheme} is the affine $\bQ$-scheme representing the functor on $\bQ$-algebras
    \[ \Sel_{S,\Pi}^{\mot}(X)\colon R \mapsto \{\text{$\Pi$-torsors over~$R$ in $\MT(\cO_{K,S},\bQ)$}\}/\text{iso}. \]
\end{defn}

Elements of the motivic Selmer scheme can more concretely be identified with $\bG_m$-equivariant algebraic cocycles \cite[Theorem~4.7]{motivic-selmer-scheme}:
\begin{equation}
    \label{eq:selmer-scheme-via-cocycles}
    \Sel_{S,\Pi}^{\mot}(X)(R) \cong \rZ^1_{\bG_m}(U_{S,R}^{\MT}, \Pi^{\omega}_R).
\end{equation}
The isomorphism is given by mapping a $\Pi$-torsor $P$ over a $\bQ$-algebra~$R$ to the cocycle $u \mapsto p_0^{-1} u(p_0)$, where $p_0 \in P^{\omega}(R)$ is the unique $\bG_m$-fixed point. For each $S$-integral point $x \in X(\cO_{K,S})$, Deligne and Goncharov construct the motivic path torsor $\pi_1^{\mot}(X;0,x)$, a $\pi_1^{\mot}(X,0)$-torsor over~$\bQ$ in $\MT(\cO_{K,S}, \bQ)$ \cite{deligne-goncharov}. Pushing out along the quotient map $\pi_1^{\mot}(X,0) \twoheadrightarrow \Pi$ yields a $\Pi$-torsor, i.e.\ a $\bQ$-point of $\Sel_{\Pi,S}^{\mot}(X)$. We thus have a \emph{motivic Kummer map}
\begin{equation}
\label{eq:motivic-kummer-map}
    j_S\colon X(\cO_{K,S}) \to \Sel_{\Pi,S}^{\mot}(X)(\bQ), \quad x \mapsto [\pi_1^{\mot}(X;0,x) \overset{\pi_1^{\mot}(X,0)}{\times} \Pi].
\end{equation}

Let $p$ be a rational prime not divisible by a prime in~$S$. We do not assume that $p$ splits completely in~$K$. For any $\fp \mid p$, Chatzistamatiou--Ünver \cite{chatzistamatiou-unver:p-adic_periods} construct a distinguished \emph{$\fp$-adic period point} $\eta_{\fp} \in U_S^{\MT}(K_{\fp})$ whose definition we recall in \Cref{def:period-element} below. Evaluation of $\bG_m$-equivariant cocycles at this point defines a map of $\bQ_p$-schemes
\begin{equation}
    \label{eq:cocycle-evaluation-at-eta-p}
    \ev_{\eta_{\fp}}\colon \Sel_{S,\Pi}^{\mot}(X)_{\bQ_p} = \rZ^1_{\bG_m}(U_S^{\MT}, \Pi^{\omega})_{\bQ_p} \to \Res_{K_{\fp}/\bQ_p} \Pi^{\dR}_{K_{\fp}}, \quad c \mapsto c(\eta_{\fp}),
\end{equation}
where we have used the identification $\Pi^{\omega} \otimes_{\bQ} K \cong \Pi^{\dR}$. We also have the local \emph{de Rham Kummer map}
\begin{equation}
    \label{eq:de-Rham-Kummer-map}
    j_{\fp}^{\dR}\colon X(\cO_{\fp}) \to \Pi^{\dR}(K_{\fp}).
\end{equation}
It is defined by mapping $x \in X(\cO_{\fp})$ to $(\gamma_x^H)^{-1} \gamma_x^{\cris} \in \pi_1^{\dR}(X,0)(K_{\fp})$ and composing with the quotient map $\pi_1^{\dR}(X,0) \to \Pi^{\dR}$, where $\gamma_x^H,\gamma_x^{\cris} \in \pi_1^{\dR}(X;0,x)(K_{\fp})$ are the unique Hodge-filtered path and the unique Frobenius-invariant path, respectively. The various maps are combined in the \emph{motivic Chabauty--Kim diagram}
\begin{equation}
    \label{eq:motivic-ck-diagram}
    \begin{tikzcd}[column sep = huge]
        X(\cO_{K,S}) \rar[hook] \dar["j_S^{\mot}"] & \prod_{\fp\mid p} X(\cO_{\fp}) \dar["j_p^{\dR} = \prod_{\fp} j_{\fp}^{\dR}"] \\
        \Sel_{S,\Pi}^{\mot}(X)_{\bQ_p} \rar["\prod_{\fp} \ev_{\eta_{\fp}}"] & \prod_{\fp \mid p} \Res_{K_{\fp}/\bQ_p} \Pi^{\dR}_{K_{\fp}}.
    \end{tikzcd}
\end{equation}
The commutativity of \eqref{eq:motivic-ck-diagram} is shown in \Cref{evaluation-localisation} below.

\subsection{Motivic-étale comparison}
\label{sec:motivic-etale-comparison}

In this subsection we prove that the motivic Chabauty--Kim diagram~\eqref{eq:motivic-ck-diagram} is canonically isomorphic to the classical (étale) Chabauty--Kim diagram~\eqref{eq:CK-diagram}. In the case $K = \bQ$ this was previously shown in \cite[§3]{BKL:chabauty-kim-sc}. We follow the same strategy, with the necessary modifications for general number fields.

If $R$ is a $\bQ_p$-algebra and $P$ is a $\Pi$-torsor over~$R$ in $\MT(\cO_{K,S}, \bQ)$, then the $p$-adic étale realisation $P^{\et}$ is a $G_K$-equivariant $\Pi^{\et}$-torsor over~$R \otimes_{\bQ} \bQ_p$. Base-changing along the multiplication map $R \otimes_{\bQ} \bQ_p \to R$ yields a $\Pi^{\et}$-torsor over~$R$. This construction defines a morphism of $\bQ_p$-schemes from the motivic to the étale Selmer scheme:
\begin{equation}
    \label{eq:motivic-to-etale}
    \Sel_{S,\Pi}^{\mot}(X)_{\bQ_p} \to \rH^1_{f,S}(G_K, \Pi^{\et}).
\end{equation}
The following is the generalisation of \cite[Theorem~3.2]{BKL:chabauty-kim-sc} to number fields.

\begin{thm}
    \label{motivic-etale-comparison}
    Let $p$ be a rational prime not divisible by a prime in~$S$. Let $\pi_1(X,0) \twoheadrightarrow \Pi$ be a quotient in $\MT(\cO_{K,S},\bQ)$ and $\Pi^{\et}$ its $p$-adic étale realisation. Then the map~\eqref{eq:motivic-to-etale} is an isomorphism, fitting into a commutative diagram for each $\fp \mid p$:
    \begin{equation}
        \label{eq:comparison-diagram}
        \begin{tikzcd}[column sep = small]
			{} & X(\cO_{K,S}) \arrow[rr,hook]\arrow[dl,"j_S^{\mot}"] \arrow[dd,"j_S"', pos=0.3] && X(\cO_{\fp}) \arrow[dl,"j_{\fp}^{\dR}"]\arrow[dd,"j_{\fp}"] \\
			\Sel_{S,\Pi}^{\mot}(X)(\bQ_p) \arrow[rr,"\ev_{\eta_{\fp}}",pos=0.7, crossing over]\arrow[dr,equal,"\eqref{eq:motivic-to-etale}"',"\sim"] && \Pi^{\dR}(K_{\fp}) \arrow[dr,equals,"\sim"] & {} \\
			{} & \rH^1_{f,S}(G_K,\Pi^{\et}(\bQ_p)) \arrow[rr,"\loc_{\fp}"]  && \rH^1_f(G_{\fp},\Pi^{\et}(\bQ_p)),
		\end{tikzcd}
    \end{equation}
    where the identification $\Pi^{\dR}(K_{\fp}) \cong \rH^1_f(G_{\fp}, \Pi^{\et}(\bQ_p))$ is the one coming from the non-abelian Bloch--Kato exponential \cite[Proposition~1.4]{kim:tangential}.
    In particular the étale Chabauty--Kim diagram~\eqref{eq:CK-diagram} and the motivic Chabauty--Kim diagram~\eqref{eq:motivic-ck-diagram} are isomorphic.
\end{thm}

The proof of \Cref{motivic-etale-comparison} occupies the remainder of this subsection. The comparison isomorphism~\eqref{eq:motivic-to-etale} between the motivic and étale Selmer scheme is a straightforward generalisation of the proof given over~$\bQ$ in \cite[§3.3]{BKL:chabauty-kim-sc}; for a sketch of the argument in the number field case see \cite[§2.6]{luedtke_2025_refined}.
It remains to check the commutativity of the diagram~\eqref{eq:comparison-diagram}. The left triangle commutes because the $p$-adic étale realisation of the motivic path torsor $\pi_1^{\mot}(X;0,x)$ agrees with the étale path torsor $\pi_1^{\et,\bQ_p}(X_{\overline{K}}; 0, x)$.

\begin{lemma}
    \label{right-triangle-commutes}
    The right triangle in Diagram~\eqref{eq:comparison-diagram} commutes.
\end{lemma}

\begin{proof}
    The space $\rH^1_f(G_{\fp}, \Pi^{\et}(\bQ_p))$ parametrises crystalline $G_{\fp}$-equivariant $\Pi^{\et}$-torsors over~$\bQ_p$. On the other hand, by \cite[Proposition~1]{kim:albanese}, $\Pi^{\dR}(K_{\fp})$ parametrises admissible $\Pi^{\dR}_{K_{\fp}}$-torsors over~$K_{\fp}$ in the sense of loc.\ cit., i.e.\ $\Pi^{\dR}$-torsors equipped with a Hodge filtration and a Frobenius automorphism compatible with the corresponding structures on $\Pi^{\dR}_{K_{\fp}}$. The parametrisation is given by sending an admissible $\Pi^{\dR}_{K_{\fp}}$-torsor~$T^{\dR}$ to $(t^H)^{-1} t^{\crs} \in \Pi^{\dR}(K_{\fp})$ where $t^H$ is the unique element in the Hodge subscheme $F^0 T^{\dR}$ and $t^{\crs}$ is the unique Frobenius-invariant element. (Kim actually uses the inverse element $(t^{\crs})^{-1} t^{H}$ due to differing conventions.) As explained in \cite[Remark~3.1.3]{betts:weight-filtrations}, the non-abelian Bloch--Kato logarithm (inverse of the non-abelian Bloch--Kato exponential) 
    \[ \log_{\mathrm{BK}}\colon \rH^1_f(G_{\fp}, \Pi^{\et}(\bQ_p)) \to \Pi^{\dR}(K_{\fp}) \] 
    sends the class of a crystalline $G_{\fp}$-equivariant $\Pi^{\et}$-torsor~$T^{\et}$ to the admissible $\Pi^{\dR}_{K_{\fp}} = D_{\dR}(\Pi^{\et})$-torsor $D_{\dR}(T^{\et})$, where $D_{\dR}$ denotes the Dieudonné functor from $p$-adic Hodge theory. The claim then follows from the comparison theorem of path torsors
    \[ D_{\dR}(\pi_1^{\et,\bQ_p}(X_{\overline{K}}; 0,x)) \cong \pi_1^{\dR}(X; 0,x)_{K_{\fp}} \]
    and the definition of the de Rham Kummer map~\eqref{eq:de-Rham-Kummer-map}.
\end{proof}

\begin{defn}
    \label{def:period-element}
    For a prime $\fp \not\in S$ of~$K$, the \emph{$\fp$-adic period point} $\eta_{\fp} \in U_S^{\MT}(K_{\fp})$ is the inverse of the point $\eta_{\fp}^{\ur}$ defined in \cite[Lemma~2.2.5]{chatzistamatiou-unver:p-adic_periods}.
\end{defn}

We recall the construction of~$\eta_{\fp}$ as a $\otimes$-automorphism of $\omega \otimes_{\bQ} K_{\fp} \cong \omega^{\dR} \otimes_K K_{\fp}$ converting between two natural splittings of the weight filtration. Given a mixed Tate motive~$M$ in $\MT(\cO_{K,S},\bQ)$, two splittings
\[  s_{\dR}, s_{\cris} \colon \bigoplus_{i \in \bZ} \gr^{2i}_W M^{\dR}_{K_{\fp}} \xrightarrow{\sim} M^{\dR}_{K_{\fp}} \]
of the weight filtration on $M^{\dR}_{K_{\fp}} \coloneqq M^{\dR} \otimes_K K_{\fp}$ are given as follows. The first one comes from the fact that the Hodge and weight filtrations are opposed, so
\[ M^{\dR} = \bigoplus_{i \in \bZ} (W_{2i} \cap F^i) M^{\dR}. \]
The second one comes from the decomposition of~$M_{K_{\fp}}^{\dR}$ into Frobenius eigenspaces:
\[ M^{\dR}_{K_{\fp}} = \bigoplus_{i \in \bZ} (M^{\dR}_{K_{\fp}})^{\varphi^f = q^{-i}}. \]
Here, $q = p^f$ is the cardinality of the residue field of~$\fp$ and $\varphi^f$ refers to the $K_{\fp}$-linear Frobenius on $M^{\dR}_{K_{\fp}}$ induced via $M^{\cris} \otimes_{K_{\fp,0}} K_{\fp} \cong M^{\dR}_{K_{\fp}}$ by the $f$th power of the semilinear Frobenius on $M^{\cris}$, the vector space over the maximal unramified subfield $K_{\fp,0} \subseteq K_{\fp}$ underlying the filtered $\varphi$-module structure on $M^{\dR}_{K_{\fp}}$, cf.\ \cite[§4.2.1]{betts:weight-filtrations}. Now $\eta_{\fp}\colon M^{\dR}_{K_{\fp}} \to M^{\dR}_{K_{\fp}}$ is given by the composite
\begin{equation}
\label{eq:eta-p-splittings}
    \eta_{\fp}\colon M^{\dR}_{K_{\fp}} \xrightarrow[\sim]{s_{\dR}^{-1}} \bigoplus_i \gr^{2i}_{W} M^{\dR}_{K_{\fp}} \xrightarrow[\sim]{s_{\cris}} M^{\dR}_{K_{\fp}}. 
\end{equation}
It is an element of $G_S^{\MT}(K_{\fp})$ since this is $\otimes$-natural in~$M$, and it belongs to the subgroup $U_S^{\MT}(K_{\fp})$ since it acts trivially on~$\bQ(1)$.

The following is the generalisation of \cite[Lemma~3.9]{BKL:chabauty-kim-sc} to number fields.

\begin{lemma}
    \label{eta-of-hodge-equals-crystalline}
    Let $P$ be a $\Pi$-torsor over a $\bQ$-algebra~$R$ in $\MT(\cO_{K,S},\bQ)$. Let $\gamma^H \in F^0 P^{\dR}(R \otimes_{\bQ} K_{\fp})$ and $\gamma^{\crs} \in P^{\dR}(R \otimes_{\bQ} K_{\fp})^{\varphi^f = 1}$ be the Hodge element and Frobenius-invariant element of $P^{\dR}_{K_{\fp}}$, respectively. Then we have
    $\eta_{\fp}(\gamma^H) = \gamma^{\crs}$.
\end{lemma}

\begin{proof}
    The Hodge element $\gamma^H$ corresponds to the algebra homomorphism
    \[ (\gamma^H)^{\sharp}\colon \cO(P^{\dR}_{K_{\fp}}) \xrightarrow{s_{\dR}^{-1}} \bigoplus_{i \in \bZ} \gr_W^{2i} \cO(P^{\dR}_{K_{\fp}}) \to \gr_W^{0} \cO(P^{\dR}_{K_{\fp}}) = R \otimes_{\bQ} K_{\fp} \]
    given by projecting to the weight-zero part for the Hodge splitting of the weight filtration. The element $\eta_{\fp}(\gamma^H)$ by definition corresponds to $(\gamma^H)^{\sharp} \circ \eta_{\fp}^{-1}$, which, since $\eta_{\fp}$ converts between the splittings $s_{\dR}$ and $s_{\cris}$, equals
    \[ \cO(P^{\dR}_{K_{\fp}}) \xrightarrow{s_{\cris}^{-1}} \bigoplus_{i \in \bZ} \gr_W^{2i} \cO(P^{\dR}_{K_{\fp}}) \to \gr_W^{0} \cO(P^{\dR}_{K_{\fp}}) = R \otimes_{\bQ} K_{\fp}, \]
    the projection onto the Frobenius-invariant subspace. But this is precisely the algebra homomorphism $(\gamma^{\crs})^{\sharp}$ corresponding to the Frobenius-invariant element~$\gamma^{\crs}$.
\end{proof}

\begin{lemma}
    \label{evaluation-localisation}
    The square of $\bQ_p$-schemes
    \begin{equation}
        \label{eq:evaluation-localisation}
        \begin{tikzcd}
            \Sel_{S,\Pi}^{\mot}(X)_{\bQ_p} \rar["\ev_{\eta_{\fp}}"] \dar[equals, "\eqref{eq:motivic-to-etale}"] & \Res_{K_{\fp}/\bQ_p} \Pi^{\dR}_{K_{\fp}} \dar["\exp_{\mathrm{BK}}"] \\
            \rH^1_{f,S}(G_K, \Pi^{\et}) \rar["\loc_{\fp}"] & \rH^1_f(G_{\fp}, \Pi^{\et})
        \end{tikzcd}
    \end{equation}
    commutes. In particular, the bottom square in Diagram~\eqref{eq:comparison-diagram} commutes. 
\end{lemma}

\begin{proof}
    Let $R$ be a $\bQ_p$-algebra. We have to show that the composition
    \begin{equation}
        \label{eq:localisation-composition}
        \Sel_{S,\Pi}^{\mot}(X)(R) \xrightarrow{\eqref{eq:motivic-to-etale}} \rH^1_{f,S}(G_K, \Pi^{\et}(R)) \xrightarrow{\loc_{\fp}} \rH^1_f(G_{\fp}, \Pi^{\et}(R)) \xrightarrow{\log_{\mathrm{BK}}} \Pi^{\dR}(R \otimes_{\bQ_p} K_{\fp})
    \end{equation}
    agrees with the cocycle evaluation map~$\ev_{\eta_{\fp}}$. Let $P$ be an element of $\Sel_{S,\Pi}^{\mot}(X)(R)$, i.e.\ a $\Pi$-torsor over~$R$ in $\MT(\cO_{K,S}, \bQ)$. Its étale realisation $P^{\et}$ is a $\Pi^{\et}$-torsor over~$R \otimes_{\bQ} \bQ_p$ and, by definition, its base change $P^{\et}_R$ along the multiplication map $R \otimes_{\bQ} \bQ_p \to R$ represents the image of $[P]$ under the map~\eqref{eq:motivic-to-etale}. Let $D_{\dR}$ be the de Rham Dieudonné functor for $K_{\fp}$, then $D_{\dR}(P^{\et}_R)$ is canonically isomorphic to $P^{\dR}_{R \otimes_{\bQ_p} K_{\fp}}$, the base change of the de Rham realisation $P^{\dR}$ along $R \otimes_{\bQ} K \to R \otimes_{\bQ_p} K_{\fp}$. Let $p^H$ be the unique element in $F^0 P^{\dR}(R \otimes_{\bQ_p} K_{\fp})$ and $p^{\crs}$ the unique Frobenius-invariant element in $P^{\dR}(R \otimes_{\bQ_p} K_{\fp})$. Then the image of~$[P]$ under the composition~\eqref{eq:localisation-composition} equals $(p^H)^{-1} p^{\crs}$ by definition of the non-abelian Bloch--Kato logarithm. On the other hand, the image of $[P]$ under the cocycle evaluation map $\ev_{\eta_{\fp}}$ is $p_0^{-1} \eta_{\fp}(p_0)$ where $p_0 \in P^{\omega}(R)$ is the unique $\bG_m$-invariant point. It thus suffices to show that $p_0 = p^H$ in $P^{\omega}(R \otimes_{\bQ_p} K_{\fp})$ and $\eta_{\fp}(p_0) = p^{\crs}$. The identity $p_0 = p^H$ follows from the fact that the identification of fibre functors $\omega \otimes_{\bQ} K \cong \omega^{\dR}$ uses the Hodge splitting of the weight filtration \cite[Proposition~2.10]{deligne-goncharov}, and then $\eta_{\fp}(p_0) = p^{\crs}$ follows from \Cref{eta-of-hodge-equals-crystalline}.
\end{proof}

The commutativity of the top square in~\eqref{eq:comparison-diagram} now follows from the commutativity of all other faces. This finishes the proof of \Cref{motivic-etale-comparison}.

\subsection{Lie algebra variant}
\label{sec:lie-algebra-variant}

Assume that $\pi_1(X,0)^{\mot} \twoheadrightarrow \Pi$ is a motivic fundamental quotient such that $U_S^{\MT}$ acts trivially on~$\Pi^{\omega}$. This will be the case for the polylogarithmic quotient considered later on (see \Cref{trivial-action-on-polylogarithmic-quotient}). In this case $\bG_m$-equivariant cocycles $U_S^{\MT} \to \Pi^{\omega}$ are the same as $\bG_m$-equivariant homomorphisms, and it becomes particularly convenient to work with a Lie algebra version of the motivic Chabauty--Kim diagram. Since $U_S^{\MT}$ and $\Pi^{\omega}$ are pro-unipotent, the logarithm provides isomorphisms of pro-affine $\bQ$-schemes
\begin{align*}
    \log\colon U_S^{\MT} &\xrightarrow{\sim} \Lie(U_S^{\MT}),\\
    \log\colon \Pi^{\omega} &\xrightarrow{\sim} \Lie(\Pi^{\omega}).
\end{align*}
Moreover, we have natural isomorphisms for $\bQ$-algebras~$R$:
\begin{align*}
    \Sel_{S,\Pi}^{\mot}(X)(R) &\overset{\mathclap{\eqref{eq:selmer-scheme-via-cocycles}}}{\cong} \rZ^1(U_{S,R}^{\MT}, \Pi^{\omega}_R) \\
    &\cong \Hom_{\bG_m}(U_{S,R}^{\MT},\Pi^{\omega}_R) \\
    &\cong \Hom_{\gr}(\Lie(U_{S,R}^{\MT}),\Lie(\Pi_R^{\omega})),
\end{align*}
so the motivic Selmer scheme can be identified with the space of graded Lie algebra homomorphisms $\Lie(U_S^{\MT}) \to \Lie(\Pi^{\omega})$.

\begin{defn}
    \label{def:p-adic-period-Lie-algebra-element}
    Let $\fp \not\in S$ be a prime of~$K$. The $\fp$-adic period Lie algebra element $\eps_{\fp} \in \Lie(U_S^{\MT})(K_{\fp})$ is defined as 
    \[ \eps_{\fp} \coloneqq \log(\eta_{\fp}), \]
    where $\eta_{\fp} \in U_S^{\MT}(K_{\fp})$ is the $\fp$-adic period element defined in \Cref{def:period-element}.
\end{defn}

Evaluation of $\bG_m$-equivariant cocycles at $\eps_{\fp}$ defines a map of $\bQ_p$-schemes
\begin{equation}
    \label{eq:lie-homomorphism-evaluation-map}
    \ev_{\eps_{\fp}}\colon \Sel_{S,\Pi}^{\mot}(X)_{\bQ_p} = \Hom_{\gr}(\Lie(U_S^{\MT}),\Lie(\Pi^{\omega}))_{\bQ_p} \to \Res_{K_{\fp}/\bQ_p} \Lie(\Pi^{\dR}_{K_{\fp}}).
\end{equation}

Now let $p$ be a rational prime not divisible by a prime in~$S$. The Lie algebra variant of the motivic Chabauty--Kim diagram~\eqref{eq:motivic-ck-diagram} looks as follows:
\begin{equation}
    \label{eq:lie-algebra-ck-diagram}
    \begin{tikzcd}[column sep = huge]
        X(\cO_{K,S}) \rar[hook] \dar["j_S^{\mot}"] & \prod_{\fp\mid p} X(\cO_{\fp}) \dar["\prod_{\fp} (\log \circ j_{\fp}^{\dR})"] \\
        \Sel_{S,\Pi}^{\mot}(X)_{\bQ_p} \rar["\prod_{\fp} \ev_{\eps_{\fp}}"] & \prod_{\fp \mid p} \Res_{K_{\fp}/\bQ_p} \Lie(\Pi^{\dR}_{K_{\fp}}).
    \end{tikzcd}
\end{equation}

\begin{prop}
    \label{lie-algebra-ck-diagram}
    The diagrams \eqref{eq:lie-algebra-ck-diagram} and \eqref{eq:motivic-ck-diagram} are canonically isomorphic.
\end{prop}

\begin{proof}
    It suffices to verify the commutativity of the diagram
    \begin{equation}
        \begin{tikzcd}[column sep = huge]
            \Hom_{\bG_m}(U_{S,\bQ_p}^{\MT}, \Pi^{\omega}_{\bQ_p})\rar["\ev_{\eta_{\fp}}"] \dar["\sim"] & \Pi^{\dR}(K_{\fp}) \dar["\log", "\sim"'] \\
            \Hom_{\gr}(\Lie(U_{S,\bQ_p}^{\MT}), \Lie(\Pi^{\omega}_{\bQ_p}))\rar["\ev_{\eps_{\fp}}"] & \Lie(\Pi^{\dR})(K_{\fp}).
        \end{tikzcd}
    \end{equation}
    This follows from the commutativity of
    \begin{equation*}
        \begin{tikzcd}[column sep = huge]
            U_S^{\MT}(K_{\fp}) \dar["\log"] \rar["c"] & \Pi^{\dR}(K_{\fp}) \dar["\log"] \\
            \Lie(U_S^{\MT})(K_{\fp}) \rar["\Lie(c)"] & \Lie(\Pi^{\dR})(K_{\fp})
        \end{tikzcd}
    \end{equation*}
    for $\bG_m$-equivariant homomorphisms $c\colon U_{S,K_{\fp}}^{\MT} \to \Pi^{\dR}_{K_{\fp}}$.
\end{proof}

\subsection{The Galois action on the Selmer scheme}
\label{sec:galois-action-on-selmer-scheme}

We show that the motivic Selmer scheme is functorial with respect to field homomorphisms. This implies in particular that when $K$ is Galois over~$\bQ$ and $S$ is Galois-stable, the Galois group $\Gal(K/\bQ)$ acts on $\Sel_{\Pi,S}^{\mot}(X)$.

Let $\sigma\colon K \to K'$ be a homomorphism of number fields and let $S$ be a set of primes of~$K$. Define the set $\sigma_*S$ of primes of~$K'$ by
\[ \sigma_*S \coloneqq (\sigma^*)^{-1}(S) \coloneqq \{ \fp' \,\vert\, \sigma^*\fp' \in S \}. \]
By \cite[§2.16]{deligne-goncharov} and $\cO_{K,S}^{\times} \subseteq \cO_{K',\sigma_*S}^{\times}$, extension of scalars along~$\sigma$ induces a fully faithful, exact $\otimes$-functor 
\begin{equation}
\label{eq:sigma-on-MT}
    \sigma_*\colon \MT(\cO_{K,S},\bQ) \to \MT(\cO_{K',\sigma_*S},\bQ) 
\end{equation}
satisfying $\sigma_* \bQ(1) = \bQ(1)$. It is thus a \emph{Tate functor} in the sense of \cite[§3.2.1]{BKL:chabauty-kim-sc}. In particular, it is automatically compatible with the canonical fibre functors, i.e.\ there exists a natural isomorphism $\omega \cong \omega \circ \sigma_*$ of functors $\MT(\cO_{K,S},\bQ) \to \bQ\mhyphen\Vect$. 

Let $\pi_1^{\mot}(X,0) \twoheadrightarrow \Pi$ be a quotient of the motivic fundamental group of~$X$ in $\MT(\cO_{K,S},\bQ)$. We identify $\Pi$ with its image $\sigma_* \Pi$ under the fully faithful functor~\eqref{eq:sigma-on-MT}. If $P$ is a $\Pi$-torsor over some $\bQ$-algebra~$R$ in $\MT(\cO_{K,S},\bQ)$, then $\sigma_* P$ is a $\Pi$-torsor over~$R$ in $\MT(\cO_{K',\sigma_*S},\bQ)$, thus we get a morphism of motivic Selmer schemes
\begin{equation}
\label{eq:sigma-on-selmer-scheme}
    \sigma_*\colon \Sel^{\mot}_{S,\Pi}(X) \to \Sel^{\mot}_{\sigma_*S,\Pi}(X).
\end{equation}
This is compatible with the Kummer map~\eqref{eq:motivic-kummer-map}:
\begin{lemma}
    \label{sigma-on-kummer}
    The following diagram commutes:
    \[
    \begin{tikzcd}
        X(\cO_{K,S}) \dar["j_S"] \rar["\sigma"] & X(\cO_{K',\sigma_*S}) \dar["j_{\sigma_*S}"] \\
        \Sel^{\mot}_{S,\Pi}(X) \rar["\sigma_*"] & \Sel^{\mot}_{\sigma_*S,\Pi}(X).
    \end{tikzcd}
    \]
\end{lemma}
\begin{proof}
    The construction of the motivic path torsor in \cite{deligne-goncharov} is compatible with extension of scalars: $\sigma_* \pi_1^{\mot}(X;0,x) \cong \pi_1^{\mot}(X; 0, \sigma x)$.
\end{proof}

The functor~\eqref{eq:sigma-on-MT} and its compatibility with the canonical fibre functors gives rise to a morphism of Tannaka groups $\sigma^*\colon G^{\MT}_{\sigma_*S} \to G_S^{\MT}$, hence to a $\bG_m$-equivariant homomorphism
\begin{equation}
    \label{eq:sigma-on-U}
    \sigma^*\colon U_{\sigma_*S}^{\MT} \to U_S^{\MT}.
\end{equation}
We have the following description of the map~\eqref{eq:sigma-on-selmer-scheme} when the motivic Selmer schemes are viewed as spaces of $\bG_m$-equivariant cocycles via~\eqref{eq:selmer-scheme-via-cocycles}.

\begin{lemma}
\label{sigma-on-cocycle-space}
    The diagram
    \[
    \begin{tikzcd}
        \Sel^{\mot}_{S,\Pi}(X) \dar["\sim"] \rar["\sigma_*"] & \Sel^{\mot}_{\sigma_*S,\Pi}(X) \dar["\sim"] \\
        \rZ^1_{\bG_m}(U_S^{\MT}, \Pi^{\omega}) \rar["\sigma_*"] & \rZ^1_{\bG_m}(U_{\sigma_*S}^{\MT}, \Pi^{\omega})
    \end{tikzcd}
    \]
    in which the bottom map sends a cocycle $c$ to its precomposition $c \circ \sigma^*$ with~\eqref{eq:sigma-on-U}, is commutative.
\end{lemma}

\begin{proof}
    Let $P$ be a $\Pi$-torsor over a $\bQ$-algebra~$R$ in $\MT(\cO_{K,S},\bQ)$. The Tannaka group $G^{\MT}_S$ acts on the canonical realisation $P^{\omega}$. The image $\sigma_*P$ of $P$ in $\MT(\cO_{K',\sigma_*S},\bQ)$ has the same canonical realisation $(\sigma_*P)^{\omega} = P^{\omega}$, and the $G^{\MT}_{\sigma_*S}$-action on it is the one obtained by pullback along $\sigma^*\colon G_{\sigma_*S}^{\MT} \to G_S^{\MT}$. In particular, if $p_0 \in P^{\omega}(R)$ denotes the $\bG_m$-fixed point, then the cocycle $c_P\colon U_{S,R}^{\MT} \to \Pi_R^{\omega}$ associated to $[P]$ under~\eqref{eq:selmer-scheme-via-cocycles} is given by $c_P(u') = p_0^{-1} u'(p_0)$, and the cocycle associated to $[\sigma_*P]$ is $c\circ \sigma^*$.
\end{proof}

\section{Chabauty--Kim theory for the polylogarithmic quotient}
\label{sec:ck-theory-for-polylog-quotient}

As before, let $K$ be a number field, let $S$ be a finite set of primes of~$K$, let $X = \bP^1 \smallsetminus \{0,1,\infty\}$ be the thrice-punctured line over $\cO_{K,S}$, and let $\vec{1}_0$ be the standard tangential base point at~$0$.

\subsection{The polylogarithmic quotient}
\label{sec:polylogarithmic-quotient}

In this paper we focus on the following quotient of the motivic fundamental group of~$X$.

\begin{defn}
\label{def:PL}
    The \emph{polylogarithmic quotient} $\Pi_{\PL}$ is defined as $\pi_1^{\mot}(X,0)/[N_1,N_1]$, where $N_1$ is the kernel of the map $\pi_1^{\mot}(X,0)\to \pi_1^{\mot}(\mathbb{G}_m)$ induced by $X\hookrightarrow \mathbb{G}_m$.
\end{defn}

For $1 \leq N \leq \infty$, the \emph{polylogarithmic depth-$N$ quotient} $\Pi_{\PL,N}$ is the maximal $N$-step nilpotent quotient of $\Pi_{\PL}$. It agrees with the depth-$N$ quotient $\Pi_N$ of the full fundamental group for $N \leq 2$.

A more concrete description of the polylogarithmic quotient can be given in terms of its Lie algebra. The Lie algebra $\Lie(\pi_1^{\omega}(X,0))$ of the full fundamental is free pro-nilpotent on two generators $e_0,e_1$; the polylogarithmic Lie algebra $\Lie(\Pi_{\PL}^{\omega})$ is the quotient by all Lie monomials of $e_1$-degree~$> 1$. A basis is given by $e_0$ and $e_1$ (in degree~$-1$) and by $\ad(e_0)^{n-1} e_1$ (in degree~$-n$) for $n \geq 2$. The depth-$N$ quotient $\Lie(\Pi_{\PL,N}^{\omega})$ is obtained by modding out all $\ad(e_0)^{n-1} e_1$ with $n > N$.

Every sequence $\omega_{i_1}\cdots \omega_{i_n}$ in the differential forms $\omega_0 = \frac{\rd t}{t}$ and $\omega_1 = \frac{\rd t}{1-t}$ defines a function on $\pi_1^{\omega}(X,0)$ sending a loop $\gamma$ to the iterated integral
\[ \int_{\gamma} \omega_{i_1}\cdots \omega_{i_n}. \]
A priori, this definition makes sense on $\pi_1^{\omega}(X,0) \otimes_{\bQ} K = \pi_1^{\dR}(X,0)$, but it descends to $\pi_1^{\omega}(X,0)$ by Galois-invariance. When the differential form $\omega_1$ occurs at most once in $\omega_{i_1}\cdots\omega_{i_n}$, the iterated integral is well-defined on the polylogarithmic quotient. We therefore have functions $\log, \Li_1, \Li_2,\ldots \in \cO(\Pi_{\PL}^{\omega})$, defined by
\[ \log(\gamma) \coloneqq \int_{\gamma} \frac{\rd t}{t}, \quad \Li_n(\gamma) \coloneqq \int_{\gamma} \underbrace{\frac{{\rd}t}{t} \cdots \frac{{\rd}t}{t} \frac{{\rd}t}{1-t}}_n. \]
If $n \leq N$ then $\Li_n$ descends to the depth-$N$ quotient $\Pi_{\PL,N}$. For a prime~$\fp$ of~$K$, the pullback of the functions $\log$ and $\Li_n$ along the $\fp$-adic Kummer map $j_{\fp}^{\dR}\colon X(\cO_{\fp}) \to \Pi^{\dR}(K_{\fp})$ yields the $\fp$-adic logarithm and polylogarithms given by the iterated Coleman integrals
\[ \log^{\fp}(z) = \int_{\vec{1}_0}^z \frac{\rd t}{t}, \quad \Li_n^{\fp}(z) = \int_{\vec{1}_0}^z \underbrace{\frac{{\rd}t}{t} \cdots \frac{{\rd}t}{t} \frac{{\rd}t}{1-t}}_n. \]
We usually omit the superscript~$(-)^{\fp}$. For later use we note the following functional equation \cite[Proposition~6.4(i)]{coleman:dilogarithms}.

\begin{lemma}[Coleman--Sinnott]
    \label{coleman-sinnott}
    For every $z \in X(\cO_{\fp})$ and $n \geq 1$ we have
    \[ \Li_n(z) + (-1)^n \Li_n(z^{-1}) = -\frac1{n!} \log(z)^n. \]
\end{lemma}

We will also need the following modified polylogarithm functions:
\begin{defn}
    \label{def:Ln}
    For $n \geq 1$, define $L_n^{\fp}\colon X(\cO_{\fp}) \to K_{\fp}$ by
    \begin{equation*}
    \label{eq:Ln}
        L_n^{\fp}(z) \coloneqq \sum_{k=0}^{n-1} \frac{B_k}{k!} \log(z)^k \Li_{n-k}(z),
    \end{equation*}
    where $B_k$ is the $k$-th Bernoulli number (with $B_1 = -1/2$).
\end{defn}
For example:
\begin{align*}
    L_1^{\fp}(z) &= \Li_1(z) = -\log(1-z),\\
    L_2^{\fp}(z) &= \Li_2(z) - \tfrac12 \log(z)\Li_1(z),\\
    L_3^{\fp}(z) &= \Li_3(z) - \tfrac12 \log(z)\Li_2(z)+ \tfrac1{12} \log(z)^2 \Li_1(z),\\
    L_4^{\fp}(z) &= \Li_4(z)- \tfrac12 \log(z)\Li_3(z)+ \tfrac1{12} \log(z)^2 \Li_2(z).
\end{align*}

The functions $L_n^{\fp}$ arise naturally as pullbacks of linear functions on $\Lie(\Pi_{\PL})$ via the logarithm:

\begin{lemma}
    \label{Ln-as-pullback}
    Let $L_0,L_1,L_2,\ldots \in \Lie(\Pi_{\PL}^{\dR})^{\vee}$ denote the dual basis of $e_0, e_1, [e_0,e_1],\ldots$ Then the pullback of $L_n$ along the composition
    \[ X(\cO_{\fp}) \xrightarrow{j_{\fp}^{\dR}} \Pi_{\PL}^{\dR}(K_{\fp}) \xrightarrow[\sim]{\log_{\Pi_{\PL}^{\dR}}} \Lie(\Pi_{\PL}^{\dR})(K_{\fp}) \]
    equals $L_n^{\fp}$.
\end{lemma}

\begin{proof}
    The function $L_n \in \Lie(\Pi_{\PL}^{\dR})^{\vee}$ agrees with the \emph{coefficient extraction functional} $f_{e_0^{n-1}e_1}$ from \cite[§10]{motivic-selmer-scheme}, as can be checked on the basis $e_0$, $e_1$, $[e_0,e_1],\ldots$ Hence, by Proposition~10.8 of loc.\ cit., the pullback of $L_n$ along the logarithm isomorphism equals
    \[ \log_{\Pi_{\PL}^{\dR}}^{\sharp} L_n = \sum_{k=0}^{n-1} \frac{B_k}{k!} \log^k \cdot \Li_{n-k}, \]
    and pulling back along $j_{\fp}^{\dR}$ yields the claim.
\end{proof}

As in the case of $\Li_n$, we usually omit the superscript $(-)^{\fp}$ from $L_n^{\fp}(z)$ and simply write $L_n(z)$. The functional equations satisfied by $L_n(z)$ are sometimes cleaner than for $\Li_n$:

\begin{lemma}
\label{coleman-sinnott-for-Ln}
    For every $z \in X(\cO_{\fp})$ and $n \geq 2$ we have
    \[ L_n(z) + (-1)^n L_n(z^{-1}) = 0. \]
\end{lemma}

\begin{proof}
    We compute
    \begin{align*}
        L_n(z) &+ (-1)^n L_n(z^{-1}) \\
        &= \sum_{k=0}^{n-1} \frac{B_k}{k!} \left(\log(z)^k \Li_{n-k}(z) + (-1)^n \log(z^{-1})^k \Li_{n-k}(z^{-1})\right) & \text{(\Cref{def:Ln})} \\
        &= \sum_{k=0}^{n-1} \frac{B_k}{k!} \log(z)^k \left( \Li_{n-k}(z) + (-1)^{n-k} \Li_{n-k}(z^{-1})\right) \\
        &= \sum_{k=0}^{n-1} \frac{B_k}{k!} \log(z)^k \left(-\frac1{(n-k)!} \log(z)^{n-k}\right) & \text{(Coleman--Sinnott)} \\
        &= -\frac1{n!} \log(z)^n \sum_{k=0}^{n-1} \binom{n}{k} B_k \\
        &= -\frac1{n!} \log(z)^n 0^{n-1} \\[1mm]
        &= 0
    \end{align*}
    where the second-to-last equality follows from Faulhaber's formula
    \[ \sum_{k=0}^n \binom{n}{k} B_k m^{n-k} = \sum_{i=0}^{m-1} i^{n-1} \]
    applied with $m = 1$, and the assumption $n \geq 2$ is used in the last equality.
\end{proof}

\begin{lemma}[{\cite[Proposition~6.4(ii)]{coleman:dilogarithms}}]
\label{L2-reflection-formula}
    For every $z \in X(\cO_{\fp})$ we have
    \[ L_2(1-z) = -L_2(z). \]
\end{lemma}

What makes the polylogarithmic quotient particularly convenient to work with in Chabauty--Kim theory is the fact that it is semisimple as a motive.

\begin{lemma}
    \label{trivial-action-on-polylogarithmic-quotient}
    The action of $U_S^{\MT}$ on $\Pi_{\PL}^{\omega}$ is trivial.
\end{lemma}

\begin{proof}
Recall that $\pi_1^{\mot}(X,0)$ denotes the motivic fundamental group at the tangential base point $\vec{1}_0$ at~$0$; similarly, let $\pi_1^{\mot}(X,1)$ denote the motivic fundamental group at the tangential base point $-\vec{1}_1$.
We have local monodromy maps \cite[§5.4]{deligne-goncharov}
\begin{align*}
    \mu_0\colon \bQ(1) &\to \pi_1^{\mot}(X,0),\\
    \mu_1'\colon \bQ(1) &\to \pi_1^{\mot}(X,1).
\end{align*}
The target of $\mu_1'$ is the motivic fundamental group based at~$1$ but following \cite[§16.11]{deligne:droite-projective} we can construct another map $\mu_1\colon \bQ(1) \to N_1^{\ab} \subseteq \Pi_{\PL}$ from it. Namely, the path torsor $\pi_1^{\mot}(\bG_m; 0,1)$ is motivically trivialised by a path $\gamma_{01}$ \cite[§15.51]{deligne:droite-projective}. Let $Q$ be the $N_1$-torsor of paths in $\pi_1^{\mot}(X;0,1)$ lifting~$\gamma_{01}$. Then the map
\[ \bQ(1) \times Q \to N_1^{\ab}, \quad (a,q) \mapsto q^{-1} \mu_1'(a) q \]
does not depend on~$q$, hence factors through $\mu_1\colon \bQ(1) \to N_1^{\ab}$, giving the claimed map.

Denote the map $\bQ(1) \to \Lie(\Pi_{\PL})$ induced by $\mu_i$ also by $\mu_i$. Then $\mu_0$ and $\ad(\mu_1)^{n-1}(\mu_0) \colon \bQ(n) \to \Lie(\Pi_{\PL})$ induce a motivic Lie algebra homomorphism
\[ \bQ(1) \ltimes \prod_{n=1}^\infty \bQ(n) \to \Lie(\Pi_{\PL}). \]
By \cite[Proposition~16.13]{deligne:droite-projective}, this is an isomorphism, as one can check on de Rham realisations. Since $U_S^{\MT}$ acts trivially on the Tate motives~$\bQ(n)$, it acts trivially on $\Lie(\Pi_{\PL}^{\omega})$, thus on~$\Pi_{\PL}^{\omega}$ itself.
\end{proof}

Since $U_S^{\MT}$ acts trivially on $\Pi_{\PL,N}^{\omega}$, we are in the situation of §\ref{sec:lie-algebra-variant} and can identify the motivic polylogarithmic Selmer scheme with a space of graded Lie algebra homomorphisms: 
\[ \Sel_{S,\PL,N}^{\mot}(X) \cong \Hom_{\gr}(\Lie(U_S^{\MT}), \Lie(\Pi_{\PL,N}^{\omega})). \]

\subsection{Coordinates on the polylogarithmic Selmer scheme}
\label{sec:coordinates}

Recall that $U_S^{\MT}$ denotes the unipotent part of the Tannaka group of the category $\MT(\cO_{K,S},\bQ)$ of mixed Tate motives over~$\cO_{K,S}$. It has the structure of a free pro-unipotent group, so its Lie algebra $\Lie(U_S^{\MT})$ is free pro-nilpotent \cite[Proposition~2.3]{deligne-goncharov}. In order to define coordinates on the polylogarithmic Selmer scheme, we need a choice of free generators of $\Lie(U_S^{\MT})$. This can be achieved by lifting a basis from the abelianisation, for which there are canonical isomorphisms of pro-finite dimensional vector spaces \cite[(2.3.11)]{deligne-goncharov}
\[ \Lie(U_S^{\MT})^{\ab} \cong \prod_{n = 1}^{\infty} \Ext^1_{\MT(\cO_{K,S},\bQ)}(\bQ(0), \bQ(n))^{\vee} \cong \prod_{n = 1}^{\infty} K_{2n-1}(\cO_{K,S})_{\bQ}^\vee. \]
We denote by $\tau_1,\ldots,\tau_{d_1}$ a choice of generators in half-weight (also called ``degree'') $-1$ and by $\sigma_{n,1},\ldots,\sigma_{n,d_n}$ generators in half-weight $-n$ for $n \geq 2$, where
\begin{equation} \label{eq:borel-dimensions}
    d_n = \dim_{\bQ} K_{2n-1}(\cO_{K,S})_{\bQ} = \begin{cases}
        r_1 + r_2 + \#S - 1, & \text{if $n = 1$},\\
        r_2, & \text{if $n \geq 2$ even},\\
        r_1 + r_2, & \text{if $n \geq 3$ odd}.
    \end{cases}
\end{equation}
Here, $r_1$ and $r_2$ are the number of real and complex places of $K$, respectively. We set
\begin{align*}
    \Sigma_1 &\coloneqq \{\tau_1,\ldots,\tau_{d_1}\},\\
    \Sigma_n &\coloneqq \{\sigma_{n,1},\ldots,\sigma_{n,d_n}\} \quad \text{for $n \geq 2$},\\
    \Sigma &\coloneqq \bigcup_{n \geq 1} \Sigma_n.
\end{align*}

Now we can define coordinates $x_i$, $y_i$ ($1 \leq i \leq d_1$) and $z_{n,i}$ ($2 \leq n \leq N$, $1 \leq i \leq d_n$) on the depth-$N$ polylogarithmic Selmer scheme via
    \begin{align}
    \label{eq:tauimapto}
        \xi(\tau_i) &= x_i(\xi) e_0 + y_i(\xi) e_1, \\
        \label{eq:sigmanmapto}
        \xi(\sigma_{n,i}) &= z_{n,i}(\xi) \ad(e_0)^{n-1} e_1
    \end{align}
for $\xi\in \Hom_{\gr}(\Lie(U_S^{\MT}), \Lie(\Pi_{\PL,N}^{\omega}))$.
This gives a parametrisation of the Selmer scheme:

\begin{thm}[{\cite[Proposition~7.6 + Theorem~11.1\,(a)]{motivic-selmer-scheme}}]
    The depth-$N$ polylogarithmic Selmer scheme $\Sel_{S,\PL,N}^{\mot}(X)$ is an affine space of dimension
    \[ \Sel_{S,\PL,N}^{\mot}(X) = r_1 \left\lfloor \frac{N+3}{2} \right\rfloor + r_2(N+1) + 2\#S - 2,\]
    and the functions $x_i, y_i, z_{n,i}$ provide an isomorphism
    \begin{equation*}
        \Sel_{S,\PL,N}^{\mot}(X) \cong \Spec \bQ[\left\{x_i,y_i\right\}_{1\leq i\leq d_1}\cup\left\{z_{n,i}\right\}_{2 \leq n\leq N,1\leq i\leq d_n}].
    \end{equation*} 
\end{thm}

When we specialise to imaginary quadratic fields, we have $d_1=\#S$ and $d_n=1$ for $n\geq 2$. We simplify the notation by writing $\sigma_n\coloneqq \sigma_{n,1}$ in that case.

\subsection{The Goncharov quotient}
\label{sec:goncharov-quotient}

In order to describe the cocycle evaluation map for the polylogarithmic quotient, it is useful to observe that all graded homomorphisms $\Lie(U_S^{\MT}) \to \Lie(\Pi_{\PL}^{\omega})$ factor through a certain quotient, which was previously defined in \cite[§2.3.2]{CDC:polylog2}.

\begin{defn}
    \label{def:goncharov-quotient}
    The \emph{Goncharov quotient} $\Lie(U_S^{\MT}) \twoheadrightarrow \Lie(U_S^{\MT})_{\Gon}$ is the quotient by the \emph{Goncharov ideal} 
    \[ I_{\Gon} \coloneqq [\Lie(U_S^{\MT})_{\leq -2}, \Lie(U_S^{\MT})_{\leq -2}]. \]
    Here, $\Lie(U_S^{\MT})_{\leq -2}$ denotes the Lie ideal of elements of degree $\leq -2$ for the grading by half-weight.
\end{defn}

\begin{lemma}
    \label{cocycles-vanishing-on-goncharov-ideal}
    Any graded homomorphism $\Lie(U_S^{\MT}) \to \Lie(\Pi_{\PL}^{\omega})$ vanishes on $I_{\Gon}$.
\end{lemma}

\begin{proof}
    Any graded homomorphism maps $\Lie(U_S^{\MT})_{\leq -2}$ into $\Lie(\Pi_{\PL}^{\omega})_{\leq -2}$, which is spanned by elements of the form $\ad(e_0)^{n-1} e_1$ with $n \geq 2$. By definition of the polylogarithmic quotient, the commutator of any two such elements vanishes in $\Lie(\Pi_{\PL}^{\omega})$.
\end{proof}

Using our choice free generators $\tau_i$, $\sigma_{n,i}$ of $\Lie(U_S^{\MT})$ we construct a simple basis of a complementary subspace of the Goncharov ideal.

\begin{lemma}
\label{goncharov-basis}
    Let $\cB$ be the set of elements of $\Lie(U_S^{\MT})$ consisting of
    \begin{itemize}
        \item the Lie algebra generators $\tau_i$ and $\sigma_{n,i}$;
        \item all elements of the form
        \begin{equation}
            \label{eq:basis-element-taus}
            \ad(\tau_{i_1}) \cdots \ad(\tau_{i_{k-1}}) \tau_{i_k}, \quad i_1 \leq \ldots \leq i_{k-1} > i_k, \quad k \geq 2,
        \end{equation}
        where $1 \leq i_j \leq d_1$;
        \item all elements of the form
        \begin{equation}
            \label{eq:basis-elements-sigma}
            \ad(\tau_{i_1})\cdots \ad(\tau_{i_k}) \sigma_{n,i}, \quad i_1\leq \ldots \leq i_k, \quad k \geq 1,
        \end{equation}
        where $1 \leq i_j \leq d_1$, $n \geq 2$, $1 \leq i \leq d_n$.
    \end{itemize}
    Then $\cB$ maps to a basis of $\Lie(U_S^{\MT})_{\Gon}$.
\end{lemma}

\begin{proof}
    For a free Lie algebra there exists a construction of a certain Hall basis which (unlike the Lyndon basis) is compatible with the derived series \cite[§5.3]{reutenauer-book}. We apply the construction to $\Lie(U_S^{\MT})$. The basis has the form $H = H_0 \cup H_1 \cup \ldots $ with the property that the $n$-th derived subspace $D^n$ is spanned by $H_{\geq n} \coloneqq \bigcup_{m \geq n} H_m$ \cite[Theorem~5.7]{reutenauer-book}. The set $H_0$ is defined as set~$\Sigma$ of all generators $\tau_i$, $\sigma_{n,i}$. We fix a total ordering on~$\Sigma$ satisfying $\tau_1 < \ldots < \tau_{d_1}$ and $\sigma_{n,j} < \tau_i$ for all $n, j, i$. For example, order the generators by degree first (where $\deg(\tau_i) = -1$ and $\deg(\sigma_{n,i}) = -n$) and order them within each fixed degree by $\sigma_{n,1} < \sigma_{n,2} < \ldots$ and $\tau_1 < \tau_2 < \ldots$. Once an ordering on~$H_n$ is chosen, the set $H_{n+1}$ is inductively defined as the set of nested commutators of the form
    \begin{equation}
        \label{eq:H1-elements}
        \bigl[\cdots[[h_1, h_2], h_3],\ldots,h_k\bigr]
    \end{equation}
    where $h_i \in H_n$, $k \geq 2$, and
    \begin{equation}
        \label{eq:order-condition}
        h_1 < h_2 \geq \ldots \geq h_k.
    \end{equation}
    
    Let $\cC$ be the union of $H_{\geq 2}$ and the set of all elements in~$H_1$ of the form~\eqref{eq:H1-elements} whose last symbol~$h_k$ is a $\sigma$, i.e.\ an element of~$\Sigma_{\geq 2} \coloneqq \Sigma_2 \cup \Sigma_3 \cup \ldots$. We claim that the Goncharov ideal~$I_{\Gon}$ is spanned by~$\cC$. 
    Let us first check that the elements of $\cC$ are indeed contained in~$I_{\Gon}$. The set $H_{\geq 2}$ is a basis of second derived subspace $D^2$ of $\Lie(U_S^{\MT})$. By definition, $D^2$ is spanned by commutators of commutators. Since commutators have degree~$\leq -2$, we have $D^2 \subseteq I_{\Gon}$ and thus $H_{\geq 2} \subseteq I_{\Gon}$. Now consider an element~$h$ of~$H_1$ of the form \eqref{eq:H1-elements} whose last symbol is $\sigma \in \Sigma_{\geq 2}$. Write $h = [h',\sigma]$ with $h' = [\cdots[[a_1,a_2],a_3]\ldots,a_{k-1}]$. If $k > 2$, the degree of $h'$ is $\leq -2$, hence $h \in I_{\Gon}$. If $k = 2$, we must have $a_1 < \sigma$ by the condition~\eqref{eq:order-condition}. Since the ordering on the generators is chosen in such a way that each $\sigma$ is smaller than each~$\tau$, the element $a_1$ must also be a~$\sigma$. In particular, $a_1$ must also have degree~$\leq -2$, so again we find $h \in I_{\Gon}$. 

    We now show that each element of $I_{\Gon}$ can be written as a linear combination of elements in~$\cC$. Note that $\Lie(U_S^{\MT})_{\leq -2}$ is spanned by $D^1$ together with $\Sigma_{\geq 2}$, hence $I_{\Gon}$ is spanned by $D^2 = [D^1,D^1]$ together with $[\Sigma_{\geq 2}, \Sigma_{\geq 2}]$ and $[D^1,\Sigma_{\geq 2}]$. Elements of $D^2$ are linear combinations of $H_{\geq 2} \subseteq \cC$. Consider an element $[\sigma,\sigma']$ of $[\Sigma_{\geq 2}, \Sigma_{\geq 2}]$, where we can assume $\sigma \neq \sigma'$. If $\sigma < \sigma'$ then $[\sigma,\sigma']$ is contained in~$H_1$; otherwise $[\sigma',\sigma] = -[\sigma,\sigma']$ is an element of $H_1$ and in either case the last symbol is in~$\Sigma_{\geq 2}$, so the element belongs to~$\cC$. Finally, consider an element of $[D^1,\Sigma_{\geq 2}]$.
    Since $D^1$ is spanned by $H_1 \cup H_{\geq 2}$, such an element can be written as a linear combination of elements of the form $[h,\sigma]$ with $h \in H_1$ and $\sigma \in \Sigma_{\geq 2}$, modulo the subspace $D^2$ which we already showed to be contained in the span of~$\cC$. Write $h = [\cdots[[a_1,a_2],a_3],\ldots,a_k]$ with $a_i \in \Sigma$, $k \geq 2$. We distinguish two cases. If $a_2 < \sigma$ then, since $a_2$ is maximal among the $a_i$ by condition~\eqref{eq:order-condition}, we have in fact $a_i < \sigma$ for all~$i$, so each $a_i$ is an element of $\Sigma_{\geq 2}$. In this case, $[h,\sigma]$ lies in the free sub-Lie algebra generated by~$\Sigma_{\geq 2}$, so it can be written as a linear combination of elements of~$H_{\geq 1}$ using only generators in~$\Sigma_{\geq 2}$. In particular, the elements of $H_1$ used in the linear combination all have an element of $\Sigma_{\geq 2}$ as their last symbol, hence $[h,\sigma]$ lies in the span of~$\cC$. Consider the second case $a_2 \geq \sigma$. If $Y$ is any nested commutator of length at least two, and $a,b$ are two elements, we have $[[Y,a],b] \equiv [[Y,b],a] \bmod D^2$ by the Jacobi identity. Thus, in the nested commutator
    \[ [h, \sigma] = [\cdots[[[a_1,a_2],a_3],\ldots,a_k],\sigma], \]
    changing the order of $a_3,\ldots,a_k,\sigma$ does not change the element modulo~$D^2$. Denote by $a_3' \geq \ldots \geq a_{k+1}'$ the sequence $a_3,\ldots,a_k$ with $\sigma$ inserted in the correct position. Then we have
    \begin{equation}
        \label{eq:reordered-commutator}
        [h,\sigma] \equiv [\cdots[[a_1,a_2],a_3'],\ldots,a_{k+1}'] \bmod D^2
    \end{equation}
    and the right hand side is an element of~$H_1$ since \eqref{eq:order-condition} is satisfied. Note that $\sigma \geq a_{k+1}'$ since $\sigma$ is among the $a_i'$, which implies that $a_{k+1}' \in \Sigma_{\geq 2}$ by the choice of ordering on~$\Sigma$. In particular, the nested commutator on the right hand side of~\eqref{eq:reordered-commutator} is contained in~$\cC$. This completes the proof that the Goncharov ideal is spanned by~$\cC$.

    As a consequence, the elements of the basis~$H$ not contained in~$\cC$ span a complementary subspace to~$I_{\Gon}$. We show that the elements of $H \smallsetminus \cC$ form the claimed basis~$\cB$ up to sign. Indeed, $H \smallsetminus \cC$ consists of $H_0 = \Sigma$ as well as all elements of $H_1$ whose last symbol is in~$\Sigma_1$. Consider such an element $h = [\cdots[[a_1,a_2],a_3],\ldots,a_k] \in H_1$ with $a_k \in \Sigma_1$. Note that $a_2,\ldots,a_k$ must all be in $\Sigma_1$ by condition~\eqref{eq:order-condition}, so $a_1$ is the only symbol which is possibly in $\Sigma_{\geq 2}$. Assume first that $a_1 \in \Sigma_1$ as well, so that $h$ is of the form
    \[ h = [\cdots[[\tau_{i_1},\tau_{i_2}],\tau_{i_3}],\ldots,\tau_{i_k}], \quad i_1 < i_2 \geq \ldots \geq i_k, \quad k\geq 2. \]
    Changing from left-centered to right-centered brackets reverses the order of the entries and introduces a harmless sign, so $h$ is of the claimed form~\eqref{eq:basis-element-taus} up to sign. Assume now that $a_1 \in \Sigma_{\geq 2}$. In this case, $h$ is of the form
    \[ h = [\cdots[[[\sigma,\tau_{i_1}],\tau_{i_2}],\tau_{i_3}],\ldots,\tau_{i_k}], \quad i_1 \geq i_2 \geq \ldots \geq i_k, \quad k\geq 1. \]
    Changing from left-centered to right-centered brackets, we obtain an element of the form~\eqref{eq:basis-elements-sigma} up to sign, finishing the proof.
\end{proof}

\begin{prop}
    \label{lie-algebra-element-expansion}
    Each element $\eps$ of $\Lie(U_S^{\MT})(R)$ for a $\bQ$-algebra~$R$ can uniquely be written in the form 
    \begin{equation}
        \label{eq:eps-form}
        \eps = \sum_{n=1}^\infty \eps_n + \eps_{\Gon}
    \end{equation}
    where $\eps_{\Gon}$ is contained in the Goncharov ideal (\Cref{def:goncharov-quotient}), 
    \begin{equation}
        \label{eq:eps-1}
        \eps_1 = \sum_{i=1}^{d_1} c_{\tau_i} \tau_i,
    \end{equation}
    and
    \begin{align}
        \label{eq:eps-n}
        \begin{split}
            \eps_n &= \sum_{ i_1 \leq \ldots \leq i_{n-1} > i_n} c_{\tau_{i_1}\cdots \tau_{i_n}} \ad(\tau_{i_1})\cdots\ad(\tau_{i_{n-1}}) \tau_{i_n} \\
            & \qquad + \sum_{m=2}^n\sum_{i=1}^{d_m} \sum_{i_1\leq\ldots\leq i_{n-m}} c_{\tau_{i_1}\cdots\tau_{i_{n-m}} \sigma_{m,i}} \ad(\tau_{i_1})\cdots\ad(\tau_{i_{n-m}}) \sigma_{m,i}
        \end{split}
    \end{align}
    for $n \geq 2$, with coefficients $c_w \in R$. Here, all the indices $i_j$ run over $\{1,\ldots,d_1\}$.
\end{prop}

\begin{proof}
    This follows from \Cref{goncharov-basis}, grouping basis elements by weight and subsuming the generators $\sigma_{m,i}$ under~\eqref{eq:basis-elements-sigma} by allowing $k = 0$.
\end{proof}

\begin{defn}
\label{def:goncharov-representation}    
    We call the representation of a Lie algebra element $\eps \in \Lie(U_S^{\MT})(R)$ in the form~\eqref{eq:eps-form}--\eqref{eq:eps-n} its \emph{Goncharov representation}, and say that the $c_w \in R$ are its \emph{Goncharov coefficients}.
\end{defn}

Given an element $\eps$ of $\Lie(U_S^{\MT})(R)$ for a $\bQ$-algebra~$R$, we can also consider its expansion
\begin{equation}
    \label{eq:lie-like-series-expansion}
    \eps = \sum_w b_w w \in R\llangle \Sigma \rrangle
\end{equation}
as a primitive element of the non-commutative power series ring $R\llangle \Sigma \rrangle$, i.e.\ a formal sum of words~$w$ in the generators~$\Sigma$ with coefficients $b_w \in R$, satisfying a shuffle identity (see \cite[Lemma~9.3\,(b)]{motivic-selmer-scheme}). For example:
\[ [\tau_1,[\tau_2,\sigma_3]] = \tau_1 \tau_2 \sigma_3 - \tau_1 \sigma_3 \tau_2 - \tau_2\sigma_3 \tau_1 + \sigma_3 \tau_2 \tau_1. \]
The Goncharov coefficient $c_w$ might disagree with the coefficient $b_w$ of the corresponding word in the expansion~\eqref{eq:lie-like-series-expansion}. However, they do agree in some cases:

\begin{lemma}
\label{lie-coeffs}
    We have $b_w = c_w$ in the following cases:
    \begin{enumerate}[label=(\alph*)]
        \item $w = \tau_i$;
        \item $w = \sigma_{n,j}$;
        \item $w = \tau_{i_1}\tau_{i_2}$ for $i_1 > i_2$;
        \item $w = \tau_{i_1}\tau_{i_2} \tau_{i_3}$ for $i_1 \leq i_2 > i_3$;
        \item $w = \tau_i \sigma_{n,j}$.
    \end{enumerate}
    Here, the $i_k$ are in the range $1 \leq i_k \leq d_1$, and $n \geq 2$ and $1 \leq j \leq d_n$.
\end{lemma}

\begin{proof}
    The Lie ideal $\Lie(U_S^{\MT})_{\leq -2}$ is generated by elements of the form $[\tau_{i_1},\tau_{i_2}]$ and $\sigma_{n,j}$, so the Goncharov ideal is generated by pairwise Lie brackets of such elements. Hence, the expansion of $\eps_{\Gon}$ only involves words with at least two $\sigma$'s, words with one $\sigma$ and least two $\tau$'s, and words with at least four $\tau$'s.
\end{proof}

\subsection{The cocycle evaluation map}
\label{sec:cocycle-evaluation-map-formula}

We can now derive a formula for the cocycle evaluation map 
\[ \ev_{\eps}\colon \Hom_{\gr}(\Lie(U_S^{\MT}), \Lie(\Pi_{\PL,N}^{\omega}))_R \to \Lie(\Pi_{\PL,N}^{\omega})_R \]
defined by an element $\eps \in \Lie(U_S^{\MT})(R)$ for a $\bQ$-algebra~$R$, in terms of the coordinates on the polylogarithmic Selmer scheme constructed in §\ref{sec:coordinates}.

\begin{thm}
    \label{cocycle-evaluation-map}
    Let $R$ be a $\bQ$-algebra and let $\eps \in \Lie(U_S^{\MT})(R)$ be written as in \Cref{lie-algebra-element-expansion} above. Let $\xi \in \Hom_{\gr}(\Lie(U_S^{\MT}),\Lie(\Pi_{\PL,N}^{\omega}))_R$ be a graded homomorphism defined over~$R$ with coordinates $x_i = x_i(\xi)$, $y_i = y_i(\xi)$ and $z_{n,i} = z_{n,i}(\xi)$ defined by~\eqref{eq:tauimapto}, \eqref{eq:sigmanmapto}. Then the cocycle evaluation map $\ev_{\eps}$ is given by
    \[ \ev_{\eps}(\xi) = \sum_{i=1}^{d_1} c_{\tau_i} x_i e_0 + \sum_{i=1}^{d_1} c_{\tau_i} y_i e_1 + \sum_{n=2}^N k_n \ad(e_0)^{n-1} e_1\]
    with
    \begin{equation}
        \begin{split}
            k_n &= \sum_{i_1 \leq \ldots \leq i_{n-1}} \sum_{i_n} c_{\tau_{i_1}\cdots\tau_{i_n}}' x_{i_1}\cdots x_{i_{n-1}} y_{i_n} \\
            & \qquad+ \sum_{m=2}^n \sum_{i=1}^{d_m} \sum_{i_1 \leq \ldots \leq i_{n-m}} c_{\tau_{i_1}\cdots \tau_{i_{n-m}}\sigma_{m,i}} x_{i_1} \cdots x_{i_{n-m}} z_{m,i}
        \end{split}
    \end{equation}
    and
    \[ 
        c_{\tau_{i_1}\cdots\tau_{i_n}}' = \begin{cases}
            c_{\tau_{i_1}\cdots\tau_{i_n}} & \text{if $i_{n-1} > i_n$},\\[2mm]
            -\displaystyle\sum_{\substack{j \in \{i_1,\ldots,i_{n-1}\},\\ j < i_n}} c_{\tau_{i_1}\cdots\hat\tau_{j}\cdots \tau_{i_n} \tau_{j}} & \text{if $i_{n-1} \leq i_n$.}
        \end{cases}
    \]
\end{thm}

\begin{proof}
    Since $\xi$ vanishes on $\eps_{\Gon}$ (\Cref{cocycles-vanishing-on-goncharov-ideal}), we have $\ev_{\eps}(\xi) = \xi(\eps) = \sum_{n=1}^N \xi(\eps_n)$. For $n = 1$ we compute
    \[ \xi(\eps_1) \overset{\eqref{eq:eps-1}}{=} \sum_{i=1}^{d_1} c_{\tau_i} \xi(\tau_i) \overset{\eqref{eq:tauimapto}}{=} \sum_{i=1}^{d_1} c_{\tau_i} x_i e_0 + \sum_{i=1}^{d_1} c_{\tau_i} y_i e_1. \]
    For $n \geq 2$, we first compute the evaluation of~$\xi$ on a nested commutator of the kind appearing in~\eqref{eq:eps-n}. For any sequence of indices $i_1,\ldots,i_n$ in $\{1,\ldots,d_1\}$ we have
    \begin{align*}
        \xi(\ad(\tau_{i_1}) \cdots \ad(\tau_{i_{n-1}}) \tau_{i_n}) &= \ad(\xi(\tau_{i_1}))\cdots \ad(\xi(\tau_{i_{n-1}})) \xi(\tau_{i_n}) \\
        &\overset{\mathclap{\eqref{eq:tauimapto}}}{=} \ad(x_{i_1} e_0 + y_{i_1} e_1) \cdots \ad(x_{i_{n-1}} e_0 + y_{i_{n-1}} e_1) (x_{i_n} e_0 + y_{i_n} e_1)\\
        &= \ad(x_{i_1} e_0 + y_{i_1} e_1) \cdots (x_{i_{n-1}} y_{i_n} - x_{i_n} y_{i_{n-1}}) [e_0,e_1] \\
        &= x_{i_1}\cdots x_{i_{n-2}} (x_{i_{n-1}} y_{i_n} - x_{i_n} y_{i_{n-1}}) \ad(e_0)^{n-1} e_1,
    \end{align*}
    where in the last equality the first $n-2$ occurrences of $e_1$ could be ignored since they appear in a Lie bracket whose second argument already contains $e_1$. A similar computation using~\eqref{eq:sigmanmapto} shows
    \[ \xi(\ad(\tau_{i_1}) \cdots \ad(\tau_{i_{n-m}}) \sigma_{m,i}) = x_{i_1}\cdots x_{i_{n-m}} z_{m,i} \ad(e_0)^{n-1} e_1, \]
    so evaluating $\xi$ on~\eqref{eq:eps-n} results in $\xi(\eps_n) = k_n \ad(e_0)^{n-1} e_1$ with
    \begin{align*}
        \begin{split}
            k_n &= \sum_{ i_1 \leq \ldots \leq i_{n-1} > i_n} c_{\tau_{i_1}\cdots \tau_{i_n}} x_{i_1}\cdots x_{i_{n-2}} (x_{i_{n-1}} y_{i_n} - x_{i_n} y_{i_{n-1}}) \\
            & \qquad + \sum_{m=2}^n\sum_{i=1}^{d_m} \sum_{i_1\leq\ldots\leq i_{n-m}} c_{\tau_{i_1}\cdots\tau_{i_{n-m}} \sigma_{m,i}} x_{i_1}\cdots x_{i_{n-m}} z_{m,i}.
        \end{split}
    \end{align*}
    It remains to determine the coefficient $c_{\tau_{j_1}\cdots\tau_{j_n}}'$ of each monomial $x_{j_1}\cdots x_{j_{n-1}} y_{j_n}$, where we can assume $j_1 \leq \ldots \leq j_{n-1}$. Let $A = \{ (i_1,\ldots,i_n) : i_1 \leq \ldots \leq i_{n-1} > i_n \}$. From a sequence $(i_1,\ldots,i_n) \in A$ we get a positive contribution to $x_{i_1}\cdots x_{i_{n-1}} y_{i_n}$ (in which the largest index $i_{n-1}$ appears among the~$x_i$) and a negative contribution to $x_{i_1}\cdots x_{i_{n-2}} x_{i_n} y_{i_{n-1}}$ (in which the largest index appears in $y_{i_{n-1}}$). If $j_{n-1} > j_n$, the only contribution comes from $(j_1,\ldots,j_{n-1},j_n) \in A$, giving the coefficient $c_{\tau_{j_1}\cdots \tau_{j_n}}$. If $j_{n-1} \leq j_n$, we get a contribution from each sequence $(i_1,\ldots,i_n) \in A$ in which $i_{n-1}$ is $j_n$ and the elements $i_1,\ldots,i_{n-2},i_n$ are $j_1,\ldots,j_{n-1}$ after ordering. But these sequences are precisely those which can be obtained from $(j_1,\ldots,j_n)$ by moving an entry $< j_n$ to the end, i.e.\ $(j_1,\ldots,\hat{k},\ldots,j_n,k)$ with $k \in \{j_1,\ldots,j_{n-1}\}$ and $k < j_n$, where the hat means that one occurrence of~$k$ is omitted. This results in the claimed formula for $c_{\tau_{j_1}\cdots\tau_{j_n}}'$.
\end{proof}

\subsection{The motivic refined Selmer scheme}
\label{sec:motivic-refined-selmer-scheme}

Now we consider the \emph{refined} Chabauty--Kim method as recalled in \S\ref{sec:ck-loci}. This is based on replacing the global étale Selmer scheme $\rH^1_{f,S}(G_K, \Pi^{\et})$ with the refined subscheme $\Sel^{\min}_{S,\Pi^{\et}}(X)$. For the depth-$N$ polylogarithmic fundamental group $\Pi_{\PL,N}$, we define a motivic refined Selmer scheme $\Sel_{S,\PL,N}^{\mot,\min}(X)$ of $\Sel_{S,\PL,N}^{\mot}(X)$, whose base change to~$\bQ_p$ agrees with the étale refined Selmer scheme $\Sel^{\min}_{S,\PL,N}(X)$. It is not necessary to restrict to the polylogarithmic quotient but we limit the discussion to this case as we introduced the coordinates on the Selmer scheme only in this setting.

\begin{defn}
\label{def:refined-selmer-scheme}
    Assume \Cref{cond:F2}. Let $1 \leq N \leq \infty$. Suppose the set $\Sigma_1$ of chosen Lie algebra generators in degree~$-1$ includes elements $(\tau_{\fl})_{\fl \in S}$ corresponding to the valuations $v_{\fl} \in (\cO_{K,S}^\times \otimes \bQ)^{\vee}$ under the isomorphism $\Lie(U_S^{\MT})_{-1} \cong (\cO_{K,S}^\times \otimes \bQ)^{\vee}$. Define the \emph{motivic refined Selmer scheme} $\Sel_{S,\PL,N}^{\mot,\min}(X)$ as the closed subscheme of~$\Sel^{\mot}_{S,\PL,N}(X)$ defined by the equations
    \[ x_{\fl} y_{\fl} (x_{\fl} + y_{\fl}) = 0 \quad \text{for $\fl \in S$} \]
    in terms of the coordinates $x_{\fl}$ and $y_{\fl}$ defined in §\ref{sec:coordinates}. It can be written as a union
    \begin{equation}
    \label{eq:decompositionSelmin}
        \Sel_{S,\PL,N}^{\mot,\min}(X) = \bigcup_{\Sigma} \Sel_{S,\PL,N}^{\mot,\Sigma}(X)
    \end{equation}
    over the $3^{\#S}$ \emph{refinement conditions} $\Sigma = (\Sigma_{\fl})_{\fl \in S}\in  \{0,1,\infty\}^S$, where $\Sel_{S,\PL,N}^{\mot,\Sigma}(X)$ is defined by the equations
    \begin{align*}
        y_{\fl} &= 0 \quad \text{if $\Sigma_{\fl} = 0$},\\
        x_{\fl} &= 0 \quad \text{if $\Sigma_{\fl} = 1$},\\
        x_{\fl} + y_{\fl} &= 0 \quad \text{if $\Sigma_{\fl} = \infty$}
    \end{align*}
    for $\fl \in S$.
\end{defn}

\begin{prop}
\label{refined-selmer-scheme-comparison}
    Assume \Cref{cond:F2}. The comparison isomorphism of unrefined Selmer schemes $\Sel_{S,\PL,N}^{\mot}(X)_{\bQ_p} \cong \rH^1_{f,S}(G_K, \Pi^{\et}_{\PL,N})$ from~\eqref{eq:motivic-to-etale} induces an isomorphism
    \[ \Sel_{S,\PL,N}^{\mot,\min}(X)_{\bQ_p} \cong \Sel_{S,\PL,N}^{\min}(X). \]
\end{prop}

\begin{proof}
    This follows the description of the étale refined Selmer scheme given in \cite[Proposition~5.11]{BKL:chabauty-kim-sc}. It is stated over $K = \bQ$ in loc.\ cit.\ but the generalisation to general number fields is straightforward. 
\end{proof}

We also define the $\Sigma$-refined Chabauty--Kim loci
\begin{align*}
    X(\cO_K \otimes \bZ_p)_{S,\PL,N}^{\Sigma} &\coloneqq j_p^{-1}(\loc_p(\Sel_{S,\PL,N}^{\mot,\Sigma}(X))), \\ 
    X(\cO_{\fp})_{S,\PL,N}^{\Sigma} &\coloneqq \pr_{\fp}(X(\cO_K \otimes \bZ_p)_{S,\PL,N}^{\Sigma}).
\end{align*}
From \eqref{eq:decompositionSelmin} we get
$$X(\cO_K \otimes \bZ_p)_{S,\PL,N}^{\min}=\bigcup_{\Sigma} X(\cO_K \otimes \bZ_p)_{S,\PL,N}^{\Sigma},$$
corresponding to the ``decomposition'' $X(\cO_{K,S}) = \bigcup_{\Sigma} X(\cO_{K,S})_{\Sigma}$ where
\[ X(\cO_{K,S})_{\Sigma} = \left\{x \in X(\cO_{K,S}) : \red_{\fl}(x) \in (X \cup \{\Sigma_{\fl}\})(\bF_{\fl}) \text{ for $\ell \in S$} \right\}. \]
Here, the mod-$\fl$ reduction map is
\[ \red_{\fl}\colon X(K_{\fl}) \subseteq \bP^1(K_{\fl}) = \bP^1(\cO_{\fl}) \twoheadrightarrow \bP^1(\bF_{\fl}). \]

As noted in \cite[Lemma~4.12]{BKL:chabauty-kim-sc}, the refined Chabauty--Kim locus for the full depth-$N$ quotient of the fundamental group is stable under the action of $\Aut(X) \cong S_3$. This is no longer true for the  \emph{polylogarithmic} quotient, cf.\ Remark 4.13 of loc.\ cit. However, the polylogarithmic locus is still stable under the action of the inversion map $\iota \colon X \to X$, $z \mapsto 1/z$:

\begin{lemma}
    \label{inversion-action-on-refined-loci}
    The automorphism $\iota\colon z \mapsto 1/z$ of~$X$ induces an involution of $X(\cO_K \otimes \bZ_p)_{S,\PL,N}^{\min}$ permuting the subsets $X(\cO_K \otimes \bZ_p)_{S,\PL,N}^{\Sigma}$ according to the action of $\iota$ on $\{0,1,\infty\}$, in the sense that
    \[ \iota(X(\cO_K \otimes \bZ_p)_{S,\PL,N}^{\Sigma}) = X(\cO_K \otimes \bZ_p)_{S,\PL,N}^{\iota(\Sigma)} \]
    for all refinement conditions $\Sigma \in \{0,1,\infty\}^S$.
\end{lemma}

\begin{proof}
    Let $b = \vec{1}_0$ and let $N_1(b)$ be the kernel of the map $\pi_1^{\mot}(X,b)_N \twoheadrightarrow \pi_1^{\mot}(\bG_m)$ induced by $X \hookrightarrow \bG_m$, and set $\Pi_{\PL,N}(b) \coloneqq \Pi_{\PL,N} = \pi_1^{\mot}(X,b)_N/[N_1(b),N_1(b)]$. Let $\iota b$ be the tangential base point at $\infty$ obtained by applying $\iota$ to~$b = \vec{1}_0$ and define $N_1(\iota b)$ and $\Pi_{\PL,N}(\iota b)$ similarly. The involution $\iota$ of~$X$ extends to an involution of~$\bG_m$, so we have a commutative diagram
    \[ 
    \begin{tikzcd}
        \pi_1^{\mot}(X,b) \rar["\iota_*"] \dar[->>] & \pi_1(X,\iota  b) \dar[->>] \\
        \pi_1^{\mot}(\bG_m) \rar["\iota_*"] & \pi_1^{\mot}(\bG_m),
    \end{tikzcd}
    \]
    which implies that $\iota_*(N_1(b)) = N_1(\iota b)$. The isomorphism $\iota_*\colon \pi_1^{\mot}(X,b) \to \pi_1^{\mot}(X,\iota b)$ thus induces an isomorphism of the depth-$N$ polylogarithmic quotients:
    \[ \iota_* \colon \Pi_{\PL,N}(b) \to \Pi_{\PL,N}(\iota b). \]
    The operation of pushing $\Pi_{\PL,N}(b)$-torsors forward along this isomorphism defines an isomorphism of Selmer schemes
    \begin{equation}
        \label{eq:iota_star}
        \iota_*\colon \Sel_{S,\Pi_{\PL}(b)}^{\mot}(X) \to \Sel_{S,\Pi_{\PL}(\iota b)}^{\mot}(X). 
    \end{equation}
    Note that conjugation by any path from~$b$ to~$\iota b$ maps $N_1(b)$ into $N_1(\iota b)$, thus $\pi_1^{\mot}(X,\iota b) \twoheadrightarrow \Pi_{\PL,N}(\iota b)$ is the quotient \emph{corresponding to} $\pi_1^{\mot}(X,b) \twoheadrightarrow \Pi_{\PL,N}(b)$ in the sense of \cite[Definition~2.6]{motivic-selmer-scheme}. By Corollary~2.8 of loc.\ cit., we get a change-of-basepoint isomorphism
    \[ \Sel_{S,\Pi_{\PL}(\iota b)}^{\mot}(X) \cong \Sel_{S,\Pi_{\PL}(b)}^{\mot}(X). \]
    The composition of this isomorphism with \eqref{eq:iota_star} yields an automorphism of the polylogarithmic Selmer scheme. One has a similar automorphism on the local side, and a compatibility with the global and local Kummer maps, which implies that $\iota$ preserves the refined Chabauty--Kim locus $X(\cO_K \otimes \bZ_p)_{S,\PL,N}^{\min}$. The more precise statement about permuting the $\Sigma$-refined subsets follows from the fact that for each $\fq \in S$, the action of $\iota$ on the local Selmer scheme $\rH^1(G_{\fq}, \Pi_{\PL})$ interchanges the two components of $j_{\fq}(X(K_{\fq}))^{\mathrm{Zar}}$ corresponding to the cusps $0$ and~$\infty$. (See \cite[Lemma~2.9]{refined-selmer-equations} for a description of the three components in the case $K = \bQ$.)
\end{proof}

\section{Motivic and $\fp$-adic periods}
\label{sec:periods}

Let $K$ be a number field and let $S$ be a finite set of primes of~$K$. Recall that $U_S^{\MT}$ denotes the unipotent part of the Tannaka group of $\MT(\cO_{K,S}, \bQ)$. To simplify the notation, let $Z = \Spec(\cO_{K,S})$ and write
\begin{align*}
    \fn(Z) &\coloneqq \Lie(U_S^{\MT}),\\
    A(Z) &\coloneqq \cO(U_S^{\MT})
\end{align*}
for the Lie algebra of, and the ring of functions on~$U_S^{\MT}$, respectively. We also use the notation $A(\cO_{K,S})$ and $\fn(\cO_{K,S})$ when we wish to indicate the dependence on the set~$S$. Both $\fn(Z)$ and $A(Z)$ carry a grading by half-weight: 
\begin{equation*}
    \fn(Z) = \prod_{n=1}^{\infty} \fn(Z)_{-n}, \qquad A(Z) = \bigoplus_{n=0}^{\infty} A(Z)_n.
\end{equation*}

\begin{defn}
    \label{def:motivic-periods}
    A (unipotent mixed Tate) \emph{motivic period} of $\cO_{K,S}$ is an element of $A(Z)$, i.e.\ a function on $U_S^{\MT}$. If $\fp \not\in S$ is a prime of~$K$ and $\eta_{\fp} \in U_S^{\MT}(K_{\fp})$ is the $\fp$-adic period point from \Cref{def:period-element}, the value $f(\eta_{\fp}) \in K_{\fp}$ of a motivic period $f \in A(Z)$ at~$\eta_{\fp}$ is called a (unipotent mixed Tate) \emph{$\fp$-adic period}. The map
    \begin{equation}
        \label{eq:period-homomorphism}
        \eta_{\fp}^{\sharp} = \per_{\fp} \colon A(Z) \to K_{\fp}, \quad f \mapsto f(\eta_{\fp})
    \end{equation}
    sending a motivic period to its $\fp$-adic value is called the \emph{$\fp$-adic period homomorphism}.
\end{defn}

\subsection{Abstract motivic periods}
\label{sec:abstract-periods}

Suppose we have a choice of free generators $\Sigma = \coprod_{n=1}^{\infty} \Sigma_n$ of $\fn(Z)$ as in §\ref{sec:coordinates}. This allows us to construct abstract motivic periods as follows. For each $\bQ$-algebra~$R$, the group of $R$-valued points $U_S^{\MT}(R)$ equals the group of grouplike elements of the non-commutative power series algebra $R\llangle \Sigma \rrangle$ with respect to the ``deconcatenation coproduct'', which is characterised by the property $\Delta \sigma = 1 \otimes \sigma + \sigma \otimes 1$ for $\sigma \in \Sigma$. Concretely, such an element is given by a formal series
\[ \sum_{w} a_w w \in R\llangle\Sigma\rrangle \]
with $w$ running over non-commutative monomials (``words'') in the generators~$\Sigma$, and with coefficients $a_w \in R$ satisfying a shuffle product formula (see \cite[Lemma~9.3\,(a)]{motivic-selmer-scheme}).

\begin{defn}
    \label{def:abtract-motivic-period}
    Given a choice of free generators $\Sigma$ of $\Lie(U_S^{\MT})$ and a word~$w$ in those generators, define the motivic period $f_w \in A(Z)$ as the function $\sum_w a_w w \mapsto a_w$.
\end{defn}

In other words, $f_w$ is given by extracting the $w$-coefficient in the expansion of an element of $U_S^{\MT}$ as a grouplike series in the chosen generators. The ring of motivic periods $A(Z)$ is a shuffle algebra with basis given by the $f_w$. It is the linear dual of the pro-finite dimensional vector space $\bQ\llangle \Sigma \rrangle$. If $w = \sigma_1\cdots\sigma_r$ and we set $\deg(w) \coloneqq \sum_i \deg(\sigma_i)$, then $f_w$ is homogeneous of degree~$-\deg(w)$ for the grading of $A(Z)$ by half-weight.

Another way of viewing the abstract motivic periods $f_w$ is as coefficients of the \emph{universal point} of $U_S^{\MT}$. The latter is defined as the $A(Z)$-valued point $\eta_{\univ} \in U_S(A(Z))$ corresponding to the identity morphism $\Spec(A(Z)) \to U_S^{\MT}$. The expansion of $\eta_{\univ}$ in the chosen generators is
\begin{equation}
    \label{eq:eta-univ}
    \eta_{\univ} = \sum_w f_w w \in A(Z)\llangle \Sigma \rrangle
\end{equation}
by definition of the~$f_w$. When $\fp \not\in S$ is a prime of~$K$ then, tautologically, the $\fp$-adic period point $\eta_{\fp}$ is the image of the universal point $\eta_{\univ}$ under the map $U_S^{\MT}(A(Z)) \to U_S^{\MT}(K_{\fp})$ induced by the period homomorphism $\per_{\fp}\colon A(Z) \to K_{\fp}$. This means that the expansion
\begin{equation}
\label{eq:eta-p-expansion}
    \eta_{\fp} = \sum_w a_w^{\fp} w \in K_{\fp}\llangle \Sigma\rrangle.
\end{equation}
of $\eta_{\fp}$ is obtained from \eqref{eq:eta-univ} by applying $\per_{\fp}$ to the coefficients: 
\begin{equation}
\label{eq:p-adic-from-motivic-coeff}
    a_w^{\fp} = \per_{\fp}(f_w) \quad \text{for all $w$}.
\end{equation}
We also need the expansion
\begin{equation}
    \label{eq:eps-p-expansion}
    \eps_{\fp} = \sum_w b_w^{\fp} w \in K_{\fp}\llangle \Sigma\rrangle
\end{equation}
of the Lie algebra element $\eps_{\fp} = \log(\eta_{\fp})$ in the chosen generators. The coefficients~$b_w^{\fp}$ can be expressed in terms of the~$a_w^{\fp}$ via the following lemma.

\begin{lemma}
    \label{lie-coefficients-from-group-coefficients}
    Let $\eta \in U_S^{\MT}(R)$ for some $\bQ$-algebra~$R$ and let $\eps \coloneqq \log(\eta) \in \Lie(U_S^{\MT})(R)$ be the corresponding Lie algebra element. Write
    \begin{equation*}
        \eta = \sum_w a_w w, \qquad
        \eps = \sum_w b_w w
    \end{equation*}
    for the expansions of $\eta$ and $\eps$ as a grouplike resp.\ primitive element of $R\llangle \Sigma\rrangle$. Then we have
    \[ b_w = \sum_{n=1}^{|w|} \frac{(-1)^{n+1}}{n} \sum_{\substack{w = w_1\cdots w_n\\ |w_i| \geq 1}} a_{w_1}\cdots a_{w_n}, \]
    for all words~$w$. Here, the inner sum runs over all decompositions of~$w$ into~$n$ nonempty subwords. 
\end{lemma}

\begin{proof}
    This follows from the fact that the logarithm isomorphism $U_S^{\MT}(R) \cong \Lie(U_S^{\MT})(R)$ is given by sending the grouplike power series $\eta = \sum_w a_w w$ to the primitive power series
    \[ \log(\eta) = \sum_{n=1}^{\infty} \frac{(-1)^{n+1}}{n} \left( \sum_w a_w w -1 \right)^n, \]
    cf.\ \cite[Lemma~10.5]{motivic-selmer-scheme} for the analogous computation in the case $\Sigma = \{e_0,e_1\}$.
\end{proof}

In practice, it may be difficult to compute the $\fp$-adic values $\per_{\fp}(f_w) \in K_{\fp}$. This is related to the non-canonicity in choosing the generators~$\Sigma$. We discuss this issue in §\ref{sec:generators} below.

\subsection{Motivic logarithms and polylogarithms}
\label{sec:motivic-polylogarithms}

One source of motivic periods whose $\fp$-adic values can be computed are motivic iterated integrals between $S$-integral points on $\bP^1 \smallsetminus \{0,1,\infty\}$ or $\bG_m$. In this paper we only need motivic logarithms and polylogarithms. For $\alpha \in \bG_m(\cO_{K,S}) = \cO_{K,S}^\times$, the motivic logarithm $\log^{\fu}(\alpha) \in A(Z)$ is defined as follows. There is a motivic path space $\pi_1^{\mot}(\bG_m; 1, \alpha)$ in $\MT(\cO_{K,S},\bQ)$ whose canonical realisation contains a canonical path $\gamma_{1,\alpha} \in \pi_1^{\omega}(\bG_m; 1, \alpha)(\bQ)$. It can be characterised as the unique path which is $\bG_m$-invariant, or equivalently as the path corresponding to the algebra homomorphism
\[ \cO(\pi_1^{\omega}(\bG_m; 1, \alpha)) = \bigoplus_{n \geq 0} \omega_n(\cO(\pi_1^{\mot}(\bG_m; 1, \alpha)) \to \bQ \]
given by projection onto the 0th factor.

\begin{defn}
    \label{def:motivic-logarithm}
    The \emph{motivic logarithm} $\log^{\fu}(\alpha) \in A(Z)$ of $\alpha \in \cO_{K,S}^\times$ is defined as the function
    \[ u \mapsto \int_{u(\gamma_{1,\alpha})} \frac{\rd t}{t}. \]
\end{defn}

Here, $u$ is an $R$-valued point of $U_S^{\MT}$ for some $\bQ$-algebra~$R$, $u(\gamma_{1,\alpha}) \in \pi_1^{\omega}(\bG_m; 1, \alpha)(R)$ is the image of the canonical path under the action of $u$, and the integral $\int_{u(\gamma_{1,\alpha})} \rd t/t \in R$ is the coefficient of $e_0$ in the expansion of $u(\gamma_{1,\alpha})$ as a group-like element of $R\llangle e_0 \rrangle$. 
With the identification $\pi_1^{\omega}(\bG_m; 1, \alpha) \otimes_{\bQ} K \cong \pi_1^{\dR}(\bG_m; 1,\alpha)$, the element $u(\gamma_{1,\alpha})$ can also be viewed as an $R' \coloneqq R \otimes_{\bQ} K$-valued de Rham path and the integral can be described as a matrix coefficient arising from the unipotent vector bundle with connection $(\cE, \nabla)$ on~$\bG_m$ defined by
\begin{equation}
    \label{eq:kummer-bundle}
    \cE = \cO_{\bG_m} e_1 \oplus \cO_{\bG_m} e_0, \quad \nabla = \rd - \begin{pmatrix}
    0 & \rd t/t \\
    0 & 0
    \end{pmatrix}.
\end{equation}
In this formulation, the integral $\int_{u(\gamma_{1,\alpha})} \rd t/t$ is the value of 1 under the composition
\[ R' \overset{e_0}{\lto} R'^2 = \Fib_1(\cE) \otimes_K R' \overset{u(\gamma_{1,\alpha})}{\lto} \Fib_{\alpha}(\cE) \otimes_K R' = R'^2 \overset{e_1^{\vee}}{\lto} R'. \]

Since the motivic path torsor $\pi_1^{\mot}(\bG_m;\vec{1}_0,1)$ is trivial \cite[§5.4]{deligne-goncharov}, one can equivalently use the tangential base point $\vec{1}_0$ instead of $1 \in \bG_m$ in the definition of the motivic logarithm.
We have $\log^{\fu}(\alpha \beta) = \log^{\fu}(\alpha) + \log^{\fu}(\beta)$ by $\gamma_{1,\alpha\beta}=\gamma_{1,\alpha}\cdot \gamma_{\alpha,\alpha\beta}$ and the change of variables formula for multiplication by~$\alpha$. 
Therefore, the motivic logarithm map extends to a well-defined $\bQ$-linear map
\[ \log^{\fu}\colon \cO_{K,S}^\times \otimes_{\bZ} \mathbb{Q} \to A(Z).\]

\begin{lemma}
\label{lem:loguislogp}
    The image of $\log^{\fu}(\alpha)$ under the $\fp$-adic period homomorphism is the $\fp$-adic logarithm $\log^{\fp}(\alpha) \in K_{\fp}$.
\end{lemma}
\begin{proof}
    By \Cref{eta-of-hodge-equals-crystalline}, $\eta_{\fp}$ maps the canonical path $\gamma_{1,\alpha}$ to the Frobenius-invariant path $\gamma_{1,\alpha}^{\crs}$. The integral along $\gamma_{1,\alpha}^{\crs}$ is the $\fp$-adic Coleman integral, i.e.\ in this case the $\fp$-adic logarithm:
    \[ \per_{\fp}(\log^{\fu}(\alpha)) = \log^{\fu}(\alpha)(\eta_{\fp}) = \int_{\eta_{\fp}(\gamma_{1,\alpha})} \frac{\rd t}{t} = \int_{\gamma_{1,\alpha}^{\crs}} \frac{\rd t}{t} = \int_1^{\alpha} \frac{\rd t}{t} = \log^{\fp}(\alpha). \qedhere \]
\end{proof}

Motivic polylogarithms of elements of $X(\cO_{K,S})$ are defined as motivic iterated integrals on the thrice-punctured line as follows. 

\begin{defn}
    \label{def:motivic-polylog}
    Let $n \geq 1$ and $\alpha \in X(\cO_{K,S})$. The \emph{motivic polylogarithm} $\Li^{\fu}_n(\alpha) \in A(Z)$ is the function
    \begin{equation}
    \label{eq:motivic-polylog}
        u \mapsto \int_{u(\gamma_{0,\alpha})} \underbrace{\frac{{\rd}t}{t} \cdots \frac{{\rd}t}{t} \frac{{\rd}t}{1-t}}_n,
    \end{equation}
    where $\gamma_{0,\alpha} \in \pi_1^{\omega}(X;0,\alpha)(\bQ)$ is the canonical (i.e.\ $\bG_m$-invariant) path from $\vec{1}_0$ to~$\alpha$.
\end{defn}
Here, $u$ is an $R$-valued point of $U_S^{\MT}$ for some $\bQ$-algebra~$R$ and the iterated integral~\eqref{eq:motivic-polylog} is defined as the coefficient of $e_0^{n-1}e_1$ in the expansion of $u(\gamma_{0,\alpha})$ as a group-like element of $R\llangle e_0,e_1\rrangle$. As in the case of the motivic logarithm, this can also be viewed as a matrix coefficient arising from a unipotent vector bundle with connection on~$X$. Since $\gamma_{1,\alpha}$ is $\bG_m$-invariant and the word $e_0^{n-1}e_1$ lives in half-weight~$-n$, the motivic polylogarithm $\Li_n^{\fu}(\alpha)$ belongs to the subspace $A(Z)_n$ of half-weight~$n$. For $n=1$, the change of variables formula with respect to $z \mapsto 1-z$ shows $\Li_1^{\fu}(\alpha) = -\log^{\fu}(1-\alpha)$. 
Similar to \Cref{lem:loguislogp} we have:

\begin{lemma}
    \label{LinuisLinp}
    The image of $\Li_n^{\fu}(\alpha)$ under the $\fp$-adic period homomorphism is the $\fp$-adic logarithm~$\Li_n^{\fp}(\alpha) \in K_{\fp}$. \qed
\end{lemma}

The definition of $\Li_n^{\fu}(\alpha) \in A(Z)$ also makes sense when $\alpha$ is an $S$-integral tangent vector. Taking $\alpha$ to be the tangent vector $-\vec{1}_1$ at~$1$, one defines the \emph{motivic zeta values}
\begin{equation}
\label{eq:motivic-zeta-value}
    \zeta^{\fu}(n) \coloneqq \Li_n^{\fu}(-\vec{1}_1) \in A(Z)_n.
\end{equation}

\subsection{Choosing generators of $\fn(Z)$}
\label{sec:generators}

Recall that a choice of a free generating set $\Sigma = \coprod_{n=1}^{\infty} \Sigma_n$ of $\fn(Z)$ determines a basis of $A(Z)$ consisting of elements $f_w$ indexed by words $w$ in those generators. The $\fp$-adic periods $\per_{\fp}(f_w) \in K_{\fp}$ appear as coefficients of the cocycle evaluation map for the $\fp$-adic period point~$\eta_{\fp}$ and its corresponding Lie algebra element $\eps_{\fp}$. In order to be able to compute these, we want to choose the generators in such a way that the $f_w$ can be expressed in terms of motivic (poly)logarithms. This is a difficult problem in general but can sometimes be achieved by constructing elements in $A(Z)_n$ forming bases of certain subspaces. This is explained for general mixed Tate groups in \cite[§3.2]{ishaidancohen_2020_mixed} and carried out in an example in \cite[§4.3]{CDC:polylog1}. We recall briefly what we need in this paper.

For every $n \geq 1$ there is an exact sequence
\begin{equation}
    \label{eq:extension-sequence}
    0 \lto \Ext^1_{\MT(\cO_{K,S},\bQ)}(\bQ(0), \bQ(n)) \lto A(Z)_n \xrightarrow{\Delta'} \bigoplus_{\substack{i + j = n\\ i,j \geq 1}} A(Z)_i \otimes A(Z)_j.
\end{equation}
The image of the extension group inside $A(Z)_n$ is denoted by $E_n(Z)$. It is the kernel of the \emph{reduced coproduct} 
\[ \Delta'(f) \coloneqq \Delta(f) - f \otimes 1 - 1 \otimes f. \]

\begin{lemma}
\label{single-letter-periods-from-En-basis}
    Let $f_1,\ldots,f_{d_n}$ be a basis of $E_n(Z)$. If the Lie algebra generators $\sigma_{n,1},\ldots,\sigma_{n,d_n} \in \Sigma_n$ lift the dual basis $f_1^{\vee},\ldots,f_{d_n}^{\vee}$ under the surjection
    \[ \fn(Z)_{-n} \twoheadrightarrow \fn(Z)^{\ab}_{-n} \cong \Ext^1_{\MT(\cO_{K,S},\bQ)}(\bQ(0),\bQ(n))^{\vee} \cong E_n(Z)^{\vee} \]
    then $f_{\sigma_{n,i}} = f_i$ for $i=1,\ldots,d_n$.
\end{lemma}

\begin{proof}
    The single-letter periods $f_{\sigma_{n,i}}$ are primitive for the deconcatenation coproduct, so they are contained in $\ker(\Delta') = E_n(Z)$ and form a basis thereof. By definition, they are dual to the images $\overline{\sigma}_{n,i}$ of the $\sigma_{n,i}$ in $E_n(Z)^{\vee}$. But the $f_i$ are also dual to the $\overline{\sigma}_{n,i}$ by assumption, so we have  $f_{\sigma_{n,i}} = f_i$ for all~$i$.
\end{proof}

We apply the above for $n = 1$. In this case we have $E_1(Z) = A(Z)_1$ and we can easily construct a basis consisting of motivic logarithms. 

\begin{lemma}
    \label{sec:logu-iso}
    $\log^{\fu}$ defines a $\mathbb{Q}$-linear isomorphism $\mathcal{O}_{K,S}^\times \otimes \mathbb{Q}\to A(Z)_1$, $\alpha\mapsto \log^{\fu}(\alpha)$.
\end{lemma}

\begin{proof}
    By $\bG_m$-invariance of the canonical path $\gamma_{1,\alpha}$ and $\deg(e_0)=-1$, it is clear from the definition that $\log^{\fu}(\alpha)$ lives in degree~$1$. By our knowledge of Ext groups in $\MT(\cO_{K,S},\bQ)$, there are isomorphisms 
    \[ A(Z)_1 = E_1(Z) = \Ext^1_{\MT(\cO_{K,S},\bQ)}(\bQ(0),\bQ(1))\cong K_1(\cO_{K,S})_{\bQ} \cong \mathcal{O}_{K,S}^\times \otimes \mathbb{Q}, \]
    so the dimensions agree and it suffices to show that $\log^{\fu}$ is injective. (The composite isomorphism above is presumably the inverse of $\log^{\fu}$ but we will not need this fact.) Choose any prime $\fp \not\in S$ of~$K$. By Lemma \ref{lem:loguislogp}, the composition $\mathrm{per}_{\fp} \circ \log^{\fu}:\mathcal{O}_{K,S}^\times\otimes\mathbb{Q}\to K_{\fp}$ is just the $\fp$-adic logarithm $\log^{\fp}\colon \alpha\mapsto \log^{\fp}(\alpha)$, which is injective. This implies that $\log^{\fu}$ is injective.
\end{proof}

Now let $\alpha_1,\ldots,\alpha_{d_1}$ be a basis of the $\bQ$-vector space $\cO_{K,S}^\times \otimes \bQ$. Then, by \Cref{sec:logu-iso}, the elements $\log^{\fu}(\alpha_1),\ldots,\log^{\fu}(\alpha_{d_1})$ form a basis of $E_1(Z)$, and we can choose $\tau_1,\ldots,\tau_{d_1} \in \fn(Z)_{-1} = \fn(Z)_{-1}^{\ab}$ as the dual basis. Then \Cref{single-letter-periods-from-En-basis} immediately yields:

\begin{lemma}
    \label{f-tau-equals-log}
    Suppose that the $\tau_i$ are dual to the $\alpha_i$. Then the abstract motivic period $f_{\tau_i}$ equals the motivic logarithm $\log^{\fu}(\alpha_i)$:
    \[  f_{\tau_i} = \log^{\fu}(\alpha_i). \]
\end{lemma}

As a consequence, we obtain a formula for the coefficients of the $\tau_i$ in the expansions of the $\fp$-adic period elements $\eta_{\fp}$ and their logarithms $\eps_{\fp}$.

\begin{prop}
    \label{tau-coeffs}
    Suppose the $\tau_i$ are dual to $\alpha_i$. Let $\fp\not\in S$ be a prime of~$K$ and let $a_w^{\fp}$ and $b_w^{\fp}$ be the coefficients of $\eta_{\fp}$ and $\eps_{\fp} = \log(\eta_{\fp})$ in the expansions \eqref{eq:eta-p-expansion} and \eqref{eq:eps-p-expansion} as a grouplike resp.\ primitive element of~$K_{\fp}\llangle \Sigma \rrangle$. Let $c_w^{\fp}$ be the coefficients in the Goncharov expansion of~$\eps_{\fp}$ (\Cref{def:goncharov-representation}). Then we have 
    \begin{equation*}
        a_{\tau_i}^{\fp} = b_{\tau_i}^{\fp} = c_{\tau_i}^{\fp} = \log^{\fp}(\alpha_i)
    \end{equation*}
    for all $i$.
\end{prop}

\begin{proof}
    We have
    \[ a_{\tau_i}^{\fp} = \per_{\fp}(f_{\tau_i}) = \per_{\fp}(\log^{\fu}(\alpha_i)) = \log^{\fp}(\alpha_i) \]
    by \eqref{eq:p-adic-from-motivic-coeff}, \Cref{f-tau-equals-log}, and \Cref{lem:loguislogp}. We have $a_{\tau_i}^{\fp} = b_{\tau_i}^{\fp} = c_{\tau_i}^{\fp}$ by \Cref{lie-coefficients-from-group-coefficients} and \Cref{lie-coeffs}.
\end{proof}

\begin{rem}
    \label{rem:roots-of-unity-extensions}
    For $n > 1$ it is difficult to construct explicit motivic iterated integrals forming bases of the extension spaces $E_n(Z)$. See \cite[Conjectures~2.2.5 \& 2.2.7]{ishaidancohen_2020_mixed} for conjectures by Zagier and Goncharov in this direction.
\end{rem}

The only examples of elements of $E_n(Z)$ for $n > 1$ which are easy to construct are the motivic zeta values~\eqref{eq:motivic-zeta-value} and motivic polylogarithms at roots of unity:

\begin{lemma}
\label{zeta-and-cyclotomic-polylogs-in-En}
    For $n \geq 2$, the motivic zeta values $\zeta^{\fu}(n)$ and the values $\Li_n^{\fu}(\zeta)$ at roots of unity $\zeta \in \cO_K^{\times}$ are contained in~$E_n(Z)$.
\end{lemma}

\begin{proof}
    The extension spaces $E_n(Z)$ for $n \geq 2$ are independent of the set $S$, so we may assume that $S$ contains all primes dividing $1-\zeta$ for $\zeta$ a root of unity in~$K$. Then each such~$\zeta$ is an $S$-integral point of~$X$ and $\Li_n^{\fu}(\zeta)$ is a well-defined element of $A(Z)_n$. By definition, an element of $A(Z)_n$ belongs to $E_n(Z)$ if and only if its reduced coproduct is zero. By \cite[Corollary~3.2]{CDC:polylog1}, the reduced coproduct of a motivic polylogarithm $\Li_n^{\fu}(\alpha)$ is given by
    \[ \Delta' \Li_n^{\fu}(\alpha) = \sum_{i=1}^{n-1} \frac1{i!}\log^{\fu}(\alpha)^i \otimes \Li_{n-i}^{\fu}(\alpha). \]
    (The factors in the tensor product are switched compared to \cite{CDC:polylog1} as we use the functional convention for path composition, cf.\ the discussion in §2.1.4 of loc.\ cit.) When $\alpha$ is the tangent vector $-\vec{1}_1$ at~$1$ we have $\log^{\fu}(-\vec{1}_1) = \log^{\fu}(1) = 0$, and when $\alpha = \zeta$ is a root of unity we have $\log^{\fu}(\zeta) = 0$ since $\log^{\fu}\colon \cO_K^{\times} \to A(Z)$ is a homomorphism. In both cases the reduced coproduct $\Delta' \Li_n^{\fu}(\alpha)$ vanishes.
\end{proof}

For a cyclotomic field, it follows \cite[Théorème 6.8]{deligne-goncharov}
that the extension spaces $E_n(Z)$ are spanned by motivic polylogarithms at roots of unity, see also \cite[Theorem~19]{hirose}:
\begin{thm}
    \label{cyclotomic-En-basis}
    Let $K=\mathbb{Q}(\zeta_m)$ be a cyclotomic field and let $Z = \Spec(\cO_K)$. For $n \geq 2$, the extension space $E_n(Z)$ is spanned by $\Li_n^{\fu}(\zeta)$ with $\zeta$ ranging over $m$-th roots of unity in~$K$.
\end{thm}
\begin{proof}
Denote by $ \mu_m $ the set of $m$-th roots of unity. For $ n\geq 2 $, there is a natural identification $E_n(Z) = \gr^C_1(A(Z)) \cong \gr_1^C(\mathcal{A})_n$, where $ \mathcal{A}$ is the unipotent motivic period ring for the category $\MT(\bQ(\zeta_m), \Gamma_m)$ of \cite{hirose} and $C_{\bullet}$ denotes the coradical filtration. Note that the choice of $\Gamma_m \subseteq \bQ(\zeta_m)^{\times} \otimes \bQ$ there does not affect the extension spaces $E_n(Z)$ for $n \geq 2$. Taking
$ l=n-1 $ in \cite[Theorem~19]{hirose} (and taking into account differing conventions for the order of integration in iterated integrals), we obtain that $E_n(Z)$ is spanned by the motivic iterated integrals
\[ I^{\mathfrak{C}}(0;\zeta,\{0\}^{n-1};1)\colon u \mapsto \int_{u(\gamma_{0,1})} \underbrace{\frac{\rd t}{t}\cdots\frac{\rd t}{t} \frac{\rd t}{t-{\zeta}}}_n \]
with $\zeta \in \mu_m$.
The change of variables $ t=\zeta s $ gives
$I^{\mathfrak{C}}(0;\zeta,\{0\}^{n-1};1)
=
-\Li_n^{\fu}(\zeta^{-1}).$
Since inversion is an involution of $\mu_m $, the result follows.
\end{proof}

Using the distribution relation $\Li_n^{\fu}(z^d) = d^{n-1} \sum_{\zeta^d = 1} \Li_n^{\fu}(\zeta z)$, one can show that for $m\geq 3$, $$\left\{
\operatorname{Li}_n^{\mathfrak u}(\zeta_m^i)
\;\middle|\;
0\leq i<\frac m2,\ \gcd(i,m)=1
\right\}$$ forms a basis of $E_n(Z)$.

\subsection{The $p$-adic period conjecture}
\label{sec:period-conjecture}

Let $\fp \not\in S$ be a prime of~$K$. Consider the following $\fp$-adic analogue of Grothendieck's period conjecture, see \cite[Conjecture 2.2.11]{ishaidancohen_2020_mixed}.

\begin{conj}[$\fp$-adic period conjecture]
\label{conj:period-conjecture}
    The $\fp$-adic period homomorphism $\per_{\fp} \colon A(Z) \to K_{\fp}$ is injective.
\end{conj}

For a recent generalisation to the non-mixed Tate setting see also \cite[Conjecture~7.3.2]{CDC:periods}.

\begin{lemma}
\label{lem:periodconj-aij-nonvanishing}
   Assume the $\fp$-adic period conjecture. Then, in the expansion~\eqref{eq:eta-p-expansion}, all coefficients $a_w^{\fp}$ are nonzero.
\end{lemma}

\begin{proof}
    The ring of motivic periods $A(Z)$ is a shuffle algebra with $\bQ$-basis $f_w$, where $w$ runs over all words in the generators~$\Sigma$. In particular, each $f_w$ is nonzero, hence the injectivity of $\per_{\fp}$ implies that each $a_w^{\fp} = \per_{\fp}(f_w)$ is nonzero.
\end{proof}

\begin{lemma}
\label{period-conjecture-implies-coeff-nonzero}
    Assume the $\fp$-adic period conjecture. Then $f(\eps_{\fp}) \neq 0$ for any nonzero $\bQ$-linear map $f\colon \fn(Z) \to \bQ$. 
    In particular, in the Goncharov representation of $\eps_{\fp}$, the coefficients $c_w^{\fp} \in K_{\fp}$ are all nonzero.
\end{lemma}
\begin{proof}
    If $f$ is a nonzero $\bQ$-linear map $\fn(Z) \to \bQ$ then the precomposition with $\log\colon U_S^{\MT} \cong \Lie(U_S^{\MT}) = \fn(Z)$ is a nonzero regular function on~$U_S^{\MT}$, so the $\fp$-adic period conjecture implies that $(f \circ \log)(\eta_{\fp}) = f(\eps_{\fp})$ is nonzero. Taking $f$ to be the function that sends an element of the form \eqref{eq:eps-form}--\eqref{eq:eps-n} to its Goncharov coefficient $c_w$ for some~$w$ shows the claim.
\end{proof}

\subsection{The Galois action on motivic periods and period points}
\label{sec:galois-action-on-periods}

We show some general functoriality properties of motivic period rings and $p$-adic period points. They imply in particular that when $K/\bQ$ is Galois and~$S$ is Galois-stable, the Galois group $\Gal(K/\bQ)$ acts on the ring of motivic periods $A(\cO_{K,S})$, and the $\fp$-adic period points~$\eta_{\fp}$ for the various primes $\fp \mid p$ are permuted by the Galois action. 

Suppose that $\sigma\colon K \to K'$ is a homomorphism of number fields and $S$ is any set of primes of~$K$. Let $\sigma_*S$ be the set of primes~$\fp'$ of~$K'$ for which $\sigma^*\fp' \in S$. Recall from §\ref{sec:galois-action-on-selmer-scheme} that extension of scalars along~$\sigma$ induces a fully faithful exact $\otimes$-functor 
\begin{equation*}
    \sigma_*\colon \MT(\cO_{K,S},\bQ) \to \MT(\cO_{K',\sigma_*S}, \bQ) 
\end{equation*}
which is naturally compatible with the canonical fibre functors, and hence induces $\bG_m$-equivariant surjective homomorphisms
\begin{align}
    \label{eq:sigma-on-U-again}
    \sigma^*\colon U_{\sigma_*S}^{\MT} &\twoheadrightarrow U_S^{\MT},\\
    \label{eq:sigma-on-lie-algebras}
    \sigma^*\colon \Lie(U_{\sigma_*S}^{\MT}) &\twoheadrightarrow \Lie(U_S^{\MT}).
\end{align}
Taking the corresponding map between the rings of functions in~\eqref{eq:sigma-on-U-again} we obtain a graded homomorphism of Hopf algebras
\begin{equation}
    \label{eq:sigma-on-motivic-periods}
    \sigma_*\colon A(\cO_{K,S}) \hookrightarrow A(\cO_{K',\sigma_*S}).
\end{equation}

The effect of~\eqref{eq:sigma-on-motivic-periods} on motivic iterated integrals is easily described:
\begin{prop}
    \label{sigma-on-motivic-polylogarithms}
    For $\alpha \in \cO_{K,S}^{\times} \otimes \bQ$ we have
    \begin{equation}
    \label{eq:sigma-on-log}
        \sigma_*(\log^{\fu}(\alpha)) = \log^{\fu}(\sigma\alpha).
    \end{equation}
    For $\alpha \in X(\cO_{K,S})$ and $n \geq 1$ we have
    \begin{equation}
    \label{eq:sigma-on-Li}
        \sigma_*(\Li_n^{\fu}(\alpha)) = \Li_n^{\fu}(\sigma\alpha).
    \end{equation}
\end{prop}

\begin{proof}
    We show \eqref{eq:sigma-on-log}; the case of \eqref{eq:sigma-on-Li} is similar. We can assume $\alpha \in \bG_m(\cO_{K,S})$. Consider the motivic path space $\pi_1^{\mot}(\bG_m;1,\alpha)$, an affine group scheme in $\MT(\cO_{K,S},\bQ)$. Applying the extension of scalars functor~\eqref{eq:sigma-on-MT} we get the corresponding path space in $\MT(\cO_{K',\sigma_*S},\bQ)$:
    \[ \sigma_* \pi_1^{\mot}(\bG_m;1,\alpha) = \pi_1^{\mot}(\bG_m;1,\sigma \alpha). \]
    Applying the canonical fibre functor and using $\sigma_* \circ \omega \cong \omega$ we obtain a natural isomorphism
    \begin{equation}
    \label{eq:sigma-on-path-space}
        \sigma_*\colon \pi_1^{\omega}(\bG_m; 1,\alpha) \xrightarrow{\simeq} \pi_1^{\omega}(\bG_m;1,\sigma\alpha). 
    \end{equation}
    It is equivariant relative to the homomorphism $\sigma^*\colon U_{\sigma_*S}^{\MT} \to U_S^{\MT}$, i.e.\ we have $u'(\sigma_*\gamma) = \sigma_*((\sigma^*u')(\gamma))$ for all $\gamma \in \pi_1^{\omega}(\bG_m;1,\alpha)$ and $u' \in U_{\sigma_*S}^{\MT}$. It is also $\bG_m$-equivariant; in particular, it maps the $\bG_m$-invariant path to the $\bG_m$-invariant path: $\sigma_* \gamma_{1,\alpha} = \gamma_{1,\sigma\alpha}$. Moreover, \eqref{eq:sigma-on-path-space} preserves integrals: $\int_{\sigma_*\gamma} \rd t/t = \int_{\gamma} \rd t/t$. This follows for example from the compatibility of the isomorphism $\omega \otimes_{\bQ} K \cong \omega^{\dR}$ with extension of scalars. Putting everything together we get
    \begin{align*}
        \sigma_*(\log^{\fu}(\alpha))(u') &= \log^{\fu}(\alpha)(\sigma^*u') = \int_{(\sigma^*u')(\gamma_{1,\alpha})} \frac{\rd t}{t} = \int_{\sigma_*((\sigma^*u')(\gamma_{1,\alpha}))} \frac{\rd t}{t} \\[1mm]
        &= \int_{u'(\sigma_* \gamma_{1,\alpha})} \frac{\rd t}{t} = \int_{u'(\gamma_{1,\sigma\alpha})} \frac{\rd t}{t} = \log^{\fu}(\sigma\alpha)(u')
    \end{align*}
    for all $u' \in U_{\sigma_*S}^{\MT}$, hence $\sigma_*(\log^{\fu}(\alpha)) = \log^{\fu}(\sigma\alpha)$.
\end{proof}

We have the following functoriality property for the $\fp$-adic period points.
\begin{thm}
\label{sigma-on-period-point}
    Let $\fp' \not\in \sigma_*S$ be a prime of~$K'$ and let $\sigma^*\fp'$ be the corresponding prime of~$K$. Then the map~\eqref{eq:sigma-on-U-again} sends the $\fp'$-adic period point $\eta_{\fp'}$ to $\eta_{\sigma^* \fp'}$, i.e.\ the following diagram commutes:
    \[
    \begin{tikzcd}
        \Spec(K'_{\fp'}) \rar["\eta_{\fp'}"] \dar["\sigma^*"] & U_{\sigma_*S}^{\MT} \dar["\sigma^*"] \\
        \Spec(K_{\sigma^*\fp'}) \rar["\eta_{\sigma^*\fp'}"] & U_S^{\MT}.
    \end{tikzcd}
    \]
\end{thm}

\begin{proof}
    Recall from~\eqref{eq:eta-p-splittings} that $\eta_{\fp'}$ is defined as the $\otimes$-automorphism of $\omega \otimes_{\bQ} K'_{\fp'}$ converting between two natural splittings of the weight filtration, coming from the Hodge filtration and the Frobenius action, respectively. The construction factors through the $p$-adic étale realisation functor~$\rho_{\et,\fp'}$ to the category $\Rep_{\bQ_p}^{\MT,\cris}(G_{\fp'})$ of crystalline mixed Tate $G_{\fp'}$-representations \cite[Lemma~2.2.5]{chatzistamatiou-unver:p-adic_periods}, where $p$ is the rational prime lying under~$\fp'$. The $p$-adic étale realisation functor is compatible with field extensions, i.e.\ the following square is naturally 2-commutative:
    \[
    \begin{tikzcd}
        \MT(\cO_{K,S},\bQ) \rar["\sigma_*"] \dar["\rho_{\et,\sigma^*\fp'}"] & \MT(\cO_{K',\sigma_*S},\bQ)  \dar["\rho_{\et,\fp'}"] \\
        \Rep_{\bQ_p}^{\MT,\cris}(G_{\sigma^*\fp'}) \rar["\sigma_*"] & \Rep_{\bQ_p}^{\MT,\cris}(G_{\fp'}).
    \end{tikzcd}
    \]
    Here, the arrow at the bottom is the restriction functor along the homomorphism $G_{\fp'} \to G_{\sigma^*\fp'}$, which in turn is induced by $\sigma_*\colon K_{\sigma^*\fp'} \to K'_{\fp'}$ and a compatible choice of algebraic closures. It follows that for $M \in \MT(\cO_{K,S},\bQ)$, the natural isomorphism $\omega(M) \otimes K'_{\fp'} \cong \omega(\sigma_* M) \otimes K'_{\fp'}$ is compatible with the splittings $s_{\dR}$ and $s_{\cris}$ of the weight filtrations, which implies the claim.
\end{proof}

A notable special case of \Cref{sigma-on-period-point} occurs when $K = K'$ and $\sigma\colon K \to K$ is an automorphism. 

\begin{prop}
\label{eta-sigma-p-from-eta-p}
    Let $\sigma \in \Aut(K)$ be an automorphism. Let $\fp$ be a prime of~$K$ with $\sigma^*\fp \not \in S$ and $K_{\fp} = \bQ_p$. Then the composition
    \[ A(\cO_{K,S}) \xrightarrow{\sigma_*} A(\cO_{K,\sigma_* S}) \xrightarrow{\per_{\fp}} \bQ_p \]
    equals the $\sigma^*\fp$-adic period map:
    \[ \per_{\sigma^* \fp} = \per_{\fp} \circ \sigma_*. \]
\end{prop}

\begin{proof}
    This follows from \Cref{sigma-on-period-point}.
\end{proof}

One of the main problems in explicit Chabauty--Kim theory is the determination of the coefficients $b_w^{\fp}$ in the expansion of $\eps_{\fp} = \sum_w b_w^{\fp} w$ as a primitive element of $K_{\fp}\llangle \Sigma \rrangle$, or equivalently of the coefficients $c_w^{\fp}$ in the Goncharov representation of~$\eps_{\fp}$ (\Cref{def:goncharov-representation}). One usually needs a sufficiently large supply of $S$-integral points to achieve this. If one does not have enough points, one can try to enlarge~$S$; this is the strategy proposed in \cite[§4.3.2]{CDC:polylog1} and \cite{ishaidancohen_2020_mixed}. A useful consequence of \Cref{sigma-on-period-point} is that it opens up the possibility of using an alternative strategy, namely passing to a larger field where the $S$-unit equation acquires new solutions. Here is a precise statement of how the Goncharov coefficients of $\eps_{\fp'}$ over a larger field are related to those of~$\eps_{\fp}$ over the original field. 

\begin{prop}
\label{goncharov-coeffs-from-larger-field}
    Let $K \subseteq K'$ be an extension of number fields, let $S$ be a set of primes of~$K$ and let $S'$ be the set of primes of~$K'$ lying over a prime in~$S$. Let $\fp' \mid \fp$ be primes with $\fp \not \in S$. Let $\Sigma = \{\tau_i\}_i \cup \{\sigma_{n,i}\}_{n,i}$ and $\Sigma' = \{\tau_i'\}_i \cup \{\sigma_{n,i}'\}_{n,i}$ be free homogeneous Lie algebra generators of $\Lie(U_S^{\MT})$ and $\Lie(U_{S'}^{\MT})$ respectively (see §\ref{sec:coordinates}), chosen compatibly in the sense that the surjection
    \[ \Lie(U_{S'}^{\MT}) \twoheadrightarrow \Lie(U_S^{\MT}) \]
    maps $\tau_i' \mapsto \tau_i$ for $1 \leq i \leq d_1$ and $\tau_i' \mapsto 0$ for $d_1 < i \leq d_1'$, and similarly for the $\sigma_{n,i}$. Then the coefficients $c_w \in K_{\fp}$ and $c_w' \in K_{\fp'}$ in the Goncharov representations of $\eps_{\fp}$ and $\eps_{\fp'}$ satisfy
    \begin{align}
        c_{\tau_i} &= c'_{\tau_i'} \quad \text{for $1 \leq i \leq d_1$}, \\
        c_{\tau_{i_1}\cdots\tau_{i_n}} &= c_{\tau_{i_1}'\cdots\tau_{i_n}'}' \quad \text{for $i_1 \leq \ldots \leq i_{n-1} > i_n$ with $1 \leq i_j \leq d_1$},\\
        c_{\tau_{i_1}\cdots\tau_{i_k}\sigma_{n,i}} &= c_{\tau_{i_1}'\cdots\tau_{i_k}'\sigma_{n,i}'}' \quad \text{for $1 \leq i_1 \leq \ldots \leq i_k \leq d_1$, $n \geq 2$, $1 \leq i \leq d_n$}.
    \end{align}
\end{prop}

\begin{proof}
    By \Cref{sigma-on-period-point}, $\eps_{\fp'} \mapsto \eps_{\fp}$ under the surjection $\Lie(U_{S'}^{\MT})(K_{\fp'}) \twoheadrightarrow \Lie(U_S^{\MT})(K_{\fp'})$. The compatibility of $\Sigma$ and $\Sigma'$ ensures that the Goncharov representation of $\eps_{\fp}$ in the generators $\Sigma$ is obtained from that of $\eps_{\fp'}$ by replacing each $\tau_i'$ with $\tau_i$ for $i \leq d_1$, and each $\sigma_{n,i}'$ with $\sigma_{n,i}$ for $i \leq d_n$ in~\eqref{eq:eps-1}--\eqref{eq:eps-n}, and dropping all terms containing other generators.
\end{proof}

In \Cref{rem:enlarging-field} below, we discuss an example where this strategy can be applied: passing from $K = \bQ$, $S = \{3\}$, where the $S$-unit equation has no solutions, to $K' = \bQ(\zeta_3)$, $S' = \{(1-\zeta_3)\}$, where there are~$8$ solutions.

\section{Integral points over imaginary quadratic fields}
\label{sec:integral-points-imag}

We determine the functions defining the polylogarithmic Chabauty--Kim locus in the case that $K$ is an imaginary quadratic field and $S = \emptyset$. It turns out that they do not depend on~$K$. We can determine their exact set of solutions, resulting in a proof of Kim's Conjecture for $K = \bQ(\zeta_3)$, and a proof that the polylogarithmic quotient is insufficient to prove the conjecture for other fields, even when combined with $S_3$-symmetrisation.

\subsection{Lie algebra generators adapted to the Galois-action}
\label{sec:lie-algebra-generators-galois-action}

Let $K$ be an imaginary quadratic field. Recall that for any set of primes~$S$ of~$K$, the Lie algebra $\Lie(U_S^{\MT})$ is free pro-nilpotent, and a set of generators $\Sigma = \Sigma_1 \cup \Sigma_2 \cup \ldots $ is obtained by choosing $\Sigma_n \subseteq \Lie(U_S^{\MT})_{-n}$ as a lift of a basis along the abelianisation map
\[ \Lie(U_S^{\MT})_{-n} \twoheadrightarrow \Lie(U_S^{\MT})_{-n}^{\ab} \cong \Ext^1_{\MT(\cO_{K,S},\bQ)}(\bQ(0),\bQ(n))^{\vee}. \]
Since $K$ is imaginary quadratic, the spaces $\Ext^1_{\MT(\cO_{K,S},\bQ)}(\bQ(0),\bQ(n))$ are 1-dimensional for all $n \geq 2$ by~\eqref{eq:borel-dimensions}. Thus, there is just a single generator~$\sigma_n \in \Sigma_n$ for each $n \geq 2$. By §\ref{sec:galois-action-on-periods}, the Galois group $\Gal(K/\bQ)$ acts contravariantly on $\Lie(U_S^{\MT})$. The following lemma ensures that the generators $\sigma_n$ can be chosen in such a way that we understand how the Galois group acts on them.

\begin{lemma}
\label{lem:sigmasigman}
    Let $\sigma\in\Gal(K/\mathbb{Q})$ be complex conjugation and let~$S$ be a Galois-stable set of primes of $K$. For every $n\geq 2$, we can choose the Lie algebra generator $\sigma_n$ in $\Lie(U_S^{\MT})_{-n}$ such that the $\Gal(K/\mathbb{Q})$-action on it is given by $\sigma^*\sigma_n=(-1)^{n+1}\sigma_n$. 
\end{lemma}
\begin{proof}
    Let $n \geq 2$. The action of $\sigma$ on the 1-dimensional space~$\Lie(U_S^{\MT})_{-n}^{\ab}$ is an involution, hence $\sigma$ acts as either $+1$ or $-1$. By \cite[(2.16.2)]{deligne-goncharov}, the Galois-fixed subspace of an extension space in $\MT(K,\bQ)$ is the corresponding extension space in $\MT(\bQ,\bQ)$:
    \[ \Ext^1_{\MT(\bQ,\bQ)}(\bQ(0),\bQ(n)) \cong \Ext^1_{\MT(K,\bQ)}(\bQ(0),\bQ(n))^{\Gal(K/\bQ)}. \]
    By~\eqref{eq:borel-dimensions}, the extension space over~$\bQ$ is 1-dimensional when $n$ is odd, and 0-dimensional when~$n$ is even. Thus, $\sigma$ acts as $(-1)^{n+1}$ on $\Lie(U_S^{\MT})_{-n}^{\ab}$. By Maschke's Theorem, there exists a $\Gal(K/\bQ)$-equivariant splitting of the surjection
    \[ \Lie(U_S^{\MT})_{-n}\twoheadrightarrow \Lie(U_S^{\MT})_{-n}^{\ab}. \]
    Choose $\sigma_n$ as the image of any basis vector under such a splitting, then we have $\sigma^*\sigma_n = (-1)^{n+1} \sigma_n$ as claimed. 
\end{proof}

\subsection{Deriving equations for the Chabauty--Kim locus}
\label{sec:imag-quad-integral-equations}

Let $K$ be an imaginary quadratic field. For a rational prime $p$ which splits in $K$, let $\mathfrak{p}_1,\mathfrak{p}_2\mid p$ denote the two primes of $K$ lying over $p$ and let $\iota_1, \iota_2\colon K \hookrightarrow \bQ_p$ denote the two $p$-adic embeddings of~$K$ defined by $\mathfrak{p}_1,\mathfrak{p}_2$. We thus identify $X(\mathcal{O}_{K}\otimes\mathbb{Z}_p)=X(\mathcal{O}_{\mathfrak{p}_1})\times X(\mathcal{O}_{\mathfrak{p}_2})$ with $X(\mathbb{Z}_p)\times X(\mathbb{Z}_p)$. 

\begin{thm}
    \label{integral-points-equations}
    Let $K$ be an imaginary quadratic field and let $p$ be a prime which splits completely in~$K$. For $1 \leq n \leq \infty$, the depth-$n$ polylogarithmic Chabauty--Kim locus $X(\cO_K \otimes \bZ_p)_{\emptyset,\PL,n}$ is contained in the set of pairs $(z_1,z_2)$ satisfying the equations
    \begin{align}
        \log(z_1) = \log(z_2) &= 0, \label{eq:log-equations}\\
        \Li_1(z_1) = \Li_1(z_2) &= 0, \label{eq:Li1-equations}\\
        \Li_m(z_1) + (-1)^m \Li_m(z_2) &= 0 \qquad \text{for $2\leq m \leq n$}. \label{eq:Lin-equations}
    \end{align}
    The containment is an equality if the $p$-adic period conjecture holds for the primes dividing $p$.
\end{thm}

\begin{proof}

     We have to determine the image of the product of evaluation maps
\begin{equation*}
    \ev_{\eps_1} \times \ev_{\eps_2}\colon \Hom_{\gr}(\Lie(U_S^{\MT}),\Lie(\Pi_{\PL,n}^{\omega})) \to \Lie(\Pi^{\dR}_{\PL,n})_{\bQ_p} \times \Lie(\Pi_{\PL,n}^{\dR})_{\bQ_p}
\end{equation*}
The Lie algebra $\Lie(\Pi^{\dR}_{\PL,n})_{\bQ_p}$ has basis $e_0, e_1,\ldots,\ad(e_0)^{n-1} e_1$; let $L_0,\ldots,L_n$ be the dual basis. The global Selmer scheme is an affine space with coordinates $z_2,\cdots,z_n$ such that
$$\xi(\sigma_m) = z_m(\xi) \ad(e_0)^{m-1} e_1, \qquad 2\leq m \leq n$$
for $\xi\in \Hom_{\gr}(\Lie(U_S^{\MT}),\Lie(\Pi_{\PL,n}^{\omega}))$.

     Write the $\fp_1$-adic period element
     \[ \varepsilon_1 = c_2 \sigma_2 + c_3 \sigma_3 + c_4 \sigma_4 + \ldots \quad \bmod \text{commutators} \]
     with $c_m \in \bQ_p$. The $p$-adic period conjecture implies that $c_m \neq 0$ for all $m \geq 2$ (\Cref{period-conjecture-implies-coeff-nonzero}). We have
     \[ \ev_{\varepsilon_1}(\xi) = \xi(\varepsilon_1) = c_2 z_2(\xi) [e_0,e_1] + c_3 z_3(\xi) [e_0,[e_0,e_1]] + c_4 z_4(\xi) [e_0,[e_0,[e_0,e_1]]] + \ldots, \]
     so in the coordinates $(z_2,z_3,\ldots)$ and $(L_0,L_1,L_2,\ldots)$ the map~$\ev_{\varepsilon_1}$ is given by
     \[ (z_2,z_3,z_4,\ldots) \mapsto (0,0,c_2 z_2, c_3z_3, c_4 z_4,\ldots). \]
     
     The second period element $\varepsilon_2$ is related to the first by the Galois action: $\varepsilon_2 = \sigma^* \varepsilon_1$, where $\sigma \in \Gal(K/\bQ)$ is complex conjugation. By the proof of Lemma \ref{lem:sigmasigman}, $\sigma^*$ acts as $(-1)^{m+1}$ on $\Lie(U^{\MT})_{-m}^{\ab}$, and thus
     \[ \varepsilon_2 = \sigma^* \varepsilon_1 = - c_2 \sigma_2 + c_3 \sigma_3 - c_4 \sigma_4 \pm \ldots \quad \bmod \text{commutators}. \]
     Evaluation at $\varepsilon_2$ is therefore given by
     \[ (z_2,z_3,z_4,\ldots) \mapsto (0,0,-c_2 z_2, c_3 z_3, -c_4 z_4, \ldots) \]
     in coordinates. Since $c_m \neq 0$ for all $m$, the image of the pair of maps $(\ev_{\varepsilon_1},\ev_{\varepsilon_2})$ in the bottom row of the Chabauty--Kim diagram is precisely cut out by the equations
     \begin{align*}
        L_0^{(1)} = L_0^{(2)} &= 0,\\
        L_1^{(1)} = L_1^{(2)} &= 0,\\
        L_m^{(1)} + (-1)^n L_m^{(2)} &= 0 \qquad \text{for $2\leq m\leq n$},
     \end{align*}
     where $L_m^{(i)}$ is the composition of $L_m$ with the projection of $\Lie(\Pi_{\PL}) \times \Lie(\Pi_{\PL})$ onto the $i$-th factor for $i=1,2$. 
     Write $L_m$ also for the function $X(\bZ_p) \to \bQ_p$ obtained by pullback via
     \[ X(\bQ_p) \overset{j_p}{\lto} \Pi_{\PL}(\bQ_p) \overset{\log_{\Pi_{\PL}}}{\lto} \Lie(\Pi_{\PL})_{\bQ_p},\]
     then the depth-$n$ Chabauty--Kim locus $X(\cO_K \otimes \bZ_p)_{\emptyset,\PL,n}$ is defined inside $X(\bZ_p) \times X(\bZ_p)$ by the equations
     \begin{align*}
        \log(z_1) = \log(z_2) &= 0,\\
        L_1(z_1) = L_1(z_2) &= 0,\\
        L_m(z_1) + (-1)^m L_m(z_2) &= 0 \qquad \text{for $2\leq m\leq n$}.
     \end{align*}
     It follows from the definition of the functions $L_m$ (\Cref{def:Ln}) that $L_m(z) = \Li_m(z)$ when $\log(z) = 0$, so we can replace $L_m(z_i)$ with $\Li_m(z_i)$ in our Kim functions, resulting in the claimed equations. 
\end{proof}

\subsection{Verifying Kim's Conjecture for integral points}
\label{sec:kims-conjecture-for-integral-points}

We are able to determine the exact set of solutions to the equations obtained in \Cref{integral-points-equations}, for any prime~$p$, thus obtaining examples where Kim's Conjecture holds. 
	
As observed in \cite[§6.1]{BDCKW}, we have the following lemma.
\begin{lemma}
\label{lem:logLi1zeta6}
    Let $z\in\mathbb{Z}_p$ with $z \not\equiv 0,1 \bmod p$. If $\log(z)=0$ and $\mathrm{Li}_1(z)=0$, then $z$ is a primitive $6$-th root of unity.
\end{lemma}

\noindent 
\begin{minipage}[c]{0.7\textwidth}
\begin{proof}
     If $\log(z)=0$ and $\mathrm{Li}_1(z)= -\log(1-z) = 0$ then both $z$ and $1-z$ are roots of unity. Fix an embedding of $\bQ(z)$ into $\mathbb{C}$. Then $|z|=1$ and $|1-z|=1$ for the complex norm~$|\cdot|$. So $z$ must lie in the intersection of the unit circle centered at the origin and the unit circle centered at $1$ in the complex plane. The points of intersection are the two primitive sixth roots of unity $\zeta_6 = e^{2\pi i/6}$ and $\zeta_6^{-1} = e^{-2\pi i/6}$.
     \end{proof}
\end{minipage}
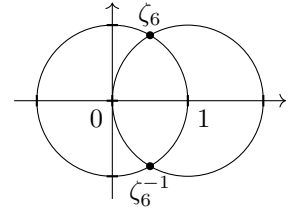
\hfill %
\begin{minipage}[c]{0.3\textwidth} %
    \centering
    \begin{tikzpicture}[scale=1.0] %
        \draw[->] (-1.3,0) -- (2.3,0);
        \draw[->] (0,-1.3) -- (0,1.3);
        
        \draw (0,0) circle (1);
        \draw (1,0) circle (1);
        
        \fill (60:1) circle (1.5pt) node[above] {$\zeta_6$};
        \fill (-60:1) circle (1.5pt) node[below] {$\zeta_6^{-1}$};
        
        \node[below left] at (0,0) {$0$};
        \node[below right] at (1,0) {$1$};

        \draw[thick] (-1,-0.075) -- (-1,0.075);
        \draw[thick] (0,-0.075) -- (0,0.075);
        \draw[thick] (1,-0.075) -- (1,0.075);
        \draw[thick] (2,-0.075) -- (2,0.075);
        \draw[thick] (-0.075,-1) -- (0.075,-1);
        \draw[thick] (-0.075,0) -- (0.075,0);
        \draw[thick] (-0.075,1) -- (0.075,1);
    \end{tikzpicture}
\end{minipage}

\begin{thm}
\label{integral-points-locus-depth1}
Let $K$ be an imaginary quadratic field and let~$p$ be a prime which splits completely in~$K$. The Chabauty--Kim locus $X(\cO_K \otimes \bZ_p)_{\emptyset,1}$ is
    \[ \begin{cases}
        \emptyset & \text{if $p \not\equiv 1 \bmod 3$},\\
        \{(\zeta_6, \zeta_6),\; (\zeta_6^{-1}, \zeta_6^{-1}),\; (\zeta_6, \zeta_6^{-1}), \; (\zeta_6^{-1}, \zeta_6)\} & \text{if $p \equiv 1 \bmod 3$}.
    \end{cases} \]
As a result, when $K = \bQ(\zeta_3)$, $S = \emptyset$, $p$ is a prime such that $p \equiv 1 \bmod 3$, and $\fp$ is a prime of~$K$ dividing~$p$, then Kim's Conjecture for the single-prime locus (\Cref{conj:kim-single}) holds in depth $1$.
\end{thm}
\begin{proof}
    By \Cref{lem:logLi1zeta6}, the solutions of the equations \eqref{eq:log-equations}, \eqref{eq:Li1-equations} are $$(\zeta_6, \zeta_6),\; (\zeta_6^{-1}, \zeta_6^{-1}),\; (\zeta_6, \zeta_6^{-1}), \; (\zeta_6^{-1}, \zeta_6)$$ if $\zeta_6$ exists in $\mathbb{Z}_p$ (which is equivalent to $p\equiv 1\bmod{3}$. Otherwise there are no solutions). As a result, $X(\mathcal{O}_{\fp})_{\emptyset,1} = \left\{\zeta_6,\zeta_6^{-1}\right\}$ for $\fp\mid p$ if $p\equiv 1\bmod{3}$, and hence Kim's Conjecture (\Cref{conj:kim-single}) holds in depth $1$ for $K=\mathbb{Q}(\zeta_3)$, $S=\emptyset$ since $X(\mathbb{Z}[\zeta_3])=\left\{\zeta_6,\zeta_6^{-1}\right\}$.
\end{proof}

For the full place version of Kim's Conjecture (\Cref{conj:kim-full}), we need to go to higher depth. The following theorem determines the Chabauty--Kim loci $X(\cO_K \otimes \bZ_p)_{\emptyset,\PL,\infty}$ by solving the equations in \Cref{integral-points-equations}.

\begin{thm}
    \label{integral-points-locus}
    Let $K$ be an imaginary quadratic field and let~$p$ be a prime which splits completely in~$K$. The polylogarithmic Chabauty--Kim locus $X(\cO_K \otimes \bZ_p)_{\emptyset,\PL,\infty}$ satisfies
    \[ X(\cO_K \otimes \bZ_p)_{\emptyset,\PL,\infty} \subseteq \begin{cases}
        \emptyset & \text{if $p \not\equiv 1 \bmod 3$},\\
        \{(\zeta_6, \zeta_6^{-1}), \; (\zeta_6^{-1}, \zeta_6)\} & \text{if $p \equiv 1 \bmod 3$}.
    \end{cases} \]
    The inclusion is an equality if the $p$-adic period conjecture holds for the primes dividing~$p$.
\end{thm}

\begin{proof}
     By \Cref{integral-points-locus-depth1} we have
    \[ X(\cO_K \otimes \bZ_p)_{\emptyset,\PL,\infty} \subseteq \{(\zeta_6, \zeta_6),\; (\zeta_6^{-1}, \zeta_6^{-1}),\; (\zeta_6, \zeta_6^{-1}), \; (\zeta_6^{-1}, \zeta_6)\}. \]
    If $(z,z)$ is one of the diagonal solutions $(\zeta_6,\zeta_6)$, $(\zeta_6^{-1},\zeta_6^{-1})$ then Eq.~\eqref{eq:Lin-equations} yields $\Li_n(z) = 0$ for all $n \geq 2$ even. Taking $n = p-3$ it was shown in \cite[Proposition~5.15]{BKL:chabauty-kim-sc} that the only root of unity in~$\bZ_p \setminus \{1\}$ satisfying $\Li_{p-3}(z) = 0$ is $z = -1$. Thus $(\zeta_6,\zeta_6)$ and $(\zeta_6^{-1},\zeta_6^{-1})$ are not contained in $X(\cO_K \otimes \bZ_p)_{\emptyset,\PL,\infty}$. The non-diagonal solutions $(\zeta_6, \zeta_6^{-1})$ and $(\zeta_6^{-1}, \zeta_6)$ on the other hand satisfy $\Li_n(z_1) + (-1)^n \Li_n(z_2) = 0$ for every $n \geq 2$ by the Coleman--Sinnott functional equation (\Cref{coleman-sinnott}), so they are contained in $X(\cO_K \otimes \bZ_p)_{\emptyset,\PL,\infty}$.
\end{proof}

\begin{cor}
    \label{kims-conjecture-for-Qzeta3}
    When $K = \bQ(\zeta_3)$, $S = \emptyset$, and $p$ is a prime such that $p \equiv 1 \bmod 3$ then we have $X(\bZ[\zeta_3] \otimes \bZ_p)_{\emptyset,\PL,\infty} = X(\bZ[\zeta_3])$, so Kim's Conjecture (\Cref{conj:kim-full}) holds for the polylogarithmic quotient.
\end{cor}

\begin{proof}
    Since $p \equiv 1 \bmod 3$, the prime~$p$ splits in $\bQ(\zeta_3)$. Let $\iota_1,\iota_2\colon \bQ(\zeta_3) \hookrightarrow \bQ_p$ be the two $p$-adic embeddings. By \Cref{integral-points-locus}, under the embedding $(\iota_1,\iota_2)\colon X(\bZ[\zeta_3]) \hookrightarrow X(\bZ_p) \times X(\bZ_p)$ we have 
    \begin{equation}
        \label{eq:Zzeta6-inclusions}
        X(\bZ[\zeta_3]) \subseteq X(\bZ[\zeta_3] \otimes \bZ_p)_{\emptyset,\PL,\infty} \subseteq \{(\zeta_6, \zeta_6^{-1}), \; (\zeta_6^{-1}, \zeta_6)\}
    \end{equation}
    with $\zeta_6 \in \bZ_p$ a primitive sixth root of unity. But $\bZ[\zeta_3]$ contains $\zeta_6 \coloneqq -\zeta_3$, and both $\zeta_6$ and $\zeta_6^{-1}$ are solutions to the unit equation since $\zeta_6 + \zeta_6^{-1} = 1$, so the inclusions in~\eqref{eq:Zzeta6-inclusions} are equalities.
\end{proof}

\begin{rem}
    \label{depth-p-minus-3}
    The proof of \Cref{integral-points-locus} shows that the functions $\log(z_i) = \Li_1(z_i) = 0$ for $i = 1,2$, and $\Li_{p-3}(z_1) + \Li_{p-3}(z_2) = 0$ suffice to cut out the locus $\{(\zeta_6, \zeta_6^{-1}), \; (\zeta_6^{-1}, \zeta_6)\}$, so \Cref{conj:kim-full} holds in depth $p-3$. If $\Li_2(\zeta_6) \neq 0$ in~$\bQ_p$ then depth~2 is in fact sufficient since it implies that the equation $\Li_2(z_1) + \Li_2(z_2) = 0$ is not satisfied for $(\zeta_6,\zeta_6)$ and $(\zeta_6^{-1},\zeta_6^{-1})$. The nonvanishing $\Li_2(\zeta_6) \neq 0$ is implied by the $p$-adic period conjecture. Indeed, the motivic dilogarithm $\Li_2^{\fu}(\zeta_6)$ is nonzero. This can be verified by computing its $p$-adic value for a single prime; with $p=7$ we have $\Li_2(\zeta_6) = \pm 7^2 + O(7^3) \neq 0$. Then the $p$-adic period conjecture implies that $\Li_2(\zeta_6) \neq 0$ in $\bQ_p$ for \emph{any} split prime~$p$. This has been verified for $p < 10^5$, see \cite[Remark~3.6]{refined-selmer-equations}.
\end{rem}

\begin{rem}
    \label{Kim-conjecture-full-vs-single-place}
    The proof of \Cref{integral-points-locus} also shows that for $K = \bQ(\zeta_3)$, $S = \emptyset$, and $p \equiv 1 \bmod 3$, the \emph{depth-1} locus equals
    \[ X(\cO_K \otimes \bZ_p)_{\emptyset,1} = \bigl\{ (\zeta_6,\zeta_6),\, (\zeta_6^{-1},\zeta_6^{-1}),\,(\zeta_6,\zeta_6^{-1})\,(\zeta_6^{-1},\zeta_6) \bigr\}. \]
    The projection to either component is $\{ \zeta_6, \zeta_6^{-1}\} = X(\bZ[\zeta_3])$, hence \Cref{conj:kim-single} holds in depth~1 for both primes $\fp_i$ dividing~$p$. However, the pairs $(\zeta_6,\zeta_6)$ and $(\zeta_6^{-1},\zeta_6^{-1})$ are not in the image of $(\iota_1,\iota_2)\colon X(\bZ[\zeta_3]) \to X(\cO_{\fp_1}) \times X(\cO_{\fp_2})$, so \Cref{conj:kim-full} fails in depth~1 and we have to go to higher depth in order to confirm the conjecture. 
\end{rem}

\begin{cor}
    \label{kims-conjecture-p-not-1-mod-3}
    When $K$ is an imaginary quadratic field different from $\bQ(\zeta_3)$, $S = \emptyset$, and $p$ is a prime which splits in~$K$ and satisfies $p \not\equiv 1 \bmod 3$ then we have $X(\cO_K \otimes \bZ_p)_{\emptyset,1} = X(\cO_K) = \emptyset$, so Kim's Conjecture~\ref{conj:kim-full} holds in depth~1.
\end{cor}

\begin{proof}
    The depth-1 locus $X(\cO_K \otimes \bZ_p)_{\emptyset,1}$ (which agrees with the \emph{polylogarithmic} depth-1 locus) is defined by the equations~\eqref{eq:log-equations}--\eqref{eq:Li1-equations}. For any pair $(z_1,z_2)$ in the locus, the $z_i$ must be primitive sixth roots of unity by \Cref{lem:logLi1zeta6}. But when $p \not \equiv 1 \bmod 3$, these are not contained in $\bZ_p$, so the locus is empty.
\end{proof}

\subsection{Insufficiency of the polylogarithmic quotient and $S_3$-symmetrisation}
\label{sec:insufficiency}

The equations obtained in \Cref{integral-points-equations} do not depend on the imaginary quadratic field~$K$, thus the Chabauty--Kim loci will contain the $\cO_{\bQ(\zeta_3)}$-integral points~$\zeta_6$ and $\zeta_6^{-1}$ even when $K \neq \bQ(\zeta_3)$.

\begin{thm}
    \label{integral-points-kims-conjecture-failure}
    Let $K$ be an imaginary quadratic field different from $\bQ(\zeta_3)$, $S = \emptyset$, and let $p$ be a prime which splits in~$K$ and satisfies $p \equiv 1 \bmod 3$. Assume the $p$-adic period conjecture. Then we have
    \[ X(\cO_K) = \emptyset \subsetneq \{\zeta_6, \zeta_6^{-1}\} = X(\cO_{\fp})_{\emptyset,\PL,\infty}, \]
    for $\fp \mid p$, so \Cref{conj:kim-single} fails for the polylogarithmic quotient in arbitrary depth.
\end{thm}

\begin{proof}
    By \Cref{integral-points-locus}, we have $X(\cO_K \otimes \bZ_p)_{\emptyset,\PL,\infty} = \{ (\zeta_6, \zeta_6^{-1}), \; (\zeta_6^{-1}, \zeta_6)\}$, hence, projecting to the $\fp$-component for one of the two primes $\fp \mid p$, we get
    \[ X(\cO_K) \subseteq X(\cO_{\fp})_{\emptyset,\PL,\infty} = \{\zeta_6, \zeta_6^{-1}\}. \]
    But $K \neq \bQ(\zeta_3)$ by assumption, so the sixth roots of unity are not contained in~$K$. Thus, $X(\cO_K) = \emptyset$ and the inclusion is strict.
\end{proof}

The insufficiency of the polylogarithmic quotient to verify Kim's Conjecture was also observed in an example over~$\bQ$ by Corwin and Dan-Cohen \cite[§5]{CDC:polylog1}. They showed that, assuming the $p$-adic period conjecture, the element $-1$ is contained in the polylogarithmic Chabauty--Kim locus $X(\bZ_p)_{\{l\},\PL,\infty}$ even though this is not an $\{l\}$-integral point when the prime~$l$ is odd. As a remedy, they proposed to consider the following strengthening. It is based on the observation that there is a natural $S_3$-action on $X = \bP^1 \smallsetminus \{0,1,\infty\}$ generated by the involutions $z \mapsto 1/z$ and $z \mapsto 1-z$. 

\begin{defn}
\label{def:S3-symmetrisation}
    In the setting of \Cref{def:ck-locus}, the \emph{$S_3$-symmetrisation} of the Chabauty--Kim locus $X(\cO_K \otimes \bZ_p)_{S,\Pi^{\et}}$ resp.\ $X(\cO_{\fp})_{S,\Pi^{\et}}$ is defined by
    \begin{align*}
        X(\cO_K \otimes \bZ_p)_{S,\Pi^{\et}}^{S_3} &\coloneqq \bigcap_{\sigma \in S_3} \sigma(X(\cO_K \otimes \bZ_p)_{S,\Pi^{\et}}),\\
        X(\cO_{\fp})_{S,\Pi^{\et}}^{S_3} &\coloneqq \bigcap_{\sigma \in S_3} \sigma(X(\cO_{\fp})_{S,\Pi^{\et}}),
    \end{align*}
    in other words, as the maximal $S_3$-stable subset. It contains all elements whose $S_3$-orbit is completely contained in the original Chabauty--Kim locus.
\end{defn}

Note that the $S_3$-symmetrised locus still contains $X(\cO_{K,S})$. The Chabauty--Kim locus for the full depth-$n$ quotient is automatically stable under the $S_3$-action \cite[Lemma~4.12]{BKL:chabauty-kim-sc} but this is not necessarily true for the polylogarithmic quotient. Thus there are inclusions
\begin{align*}
    X(\cO_{K,S}) &\subseteq X(\cO_K \otimes \bZ_p)_{S,n} \subseteq X(\cO_K \otimes \bZ_p)_{S,\PL,n}^{S_3} \subseteq X(\cO_K \otimes \bZ_p)_{S,\PL,n},\\
    X(\cO_{K,S}) &\subseteq X(\cO_{\fp})_{S,n} \subseteq X(\cO_{\fp})_{S,\PL,n}^{S_3} \subseteq X(\cO_{\fp})_{S,\PL,n}
\end{align*}
for all $1 \leq n \leq \infty$. Corwin and Dan-Cohen conjectured that $X(\cO_{K,S}) = X(\cO_K \otimes \bZ_p)_{S,\PL,n}^{S_3}$ for $n \gg 0$ (\cite[Conjecture~5.4]{CDC:polylog1} in the case $K= \bQ$, \cite[Conjecture~8.2.1]{ishaidancohen_2020_mixed} in general). This would also imply the single-prime variant $X(\cO_{K,S}) = X(\cO_{\fp})_{S,\PL,n}^{S_3}$ for $n \gg 0$. However, this turns out to be false, with infinitely many counterexamples provided by \Cref{integral-points-kims-conjecture-failure} (assuming the $p$-adic period conjecture).

\begin{thm}
    \label{S3-symmetrisation-failure}
    Let $K$ and $p$ be as in \Cref{integral-points-kims-conjecture-failure}. Then we have
    \[ X(\cO_K) = \emptyset \subsetneq \{\zeta_6, \zeta_6^{-1}\} = X(\cO_{\fp})_{\emptyset,\PL,\infty}^{S_3} \]
    for $\fp \mid p$, hence even the $S_3$-symmetrised version of \Cref{conj:kim-single} fails for the polylogarithmic quotient.
\end{thm}

\begin{proof}
    Since $1-\zeta_6 = \zeta_6^{-1}$, the locus $X(\cO_{\fp})_{\emptyset,\PL,\infty} = \{\zeta_6, \zeta_6^{-1}\}$ is already stable under the $S_3$-action, hence nothing is gained by $S_3$-symmetrisation.
\end{proof}

This shows the necessity to also study non-polylogarithmic quotients of the fundamental group in order to verify Kim's Conjecture.

\section{Roots of unity in Chabauty--Kim loci}
\label{sec:roots-of-unity}

As discussed in §\ref{sec:insufficiency}, Corwin and Dan-Cohen demonstrated the insufficiency of the polylogarithmic quotient in the case $K = \bQ$, $S = \{l\}$ by observing that the polylogarithmic Chabauty--Kim locus always contains~$-1$, even though this is not an $\{l\}$-integral point when $l \neq 2$. Still inverting one rational prime but taking $K$ to be imaginary quadratic, we find a more drastic example of this phenomenon: all nontrivial roots of unity in~$\bZ_p$ are contained in the polylogarithmic Chabauty--Kim locus. This leads to further such examples for all fields containing an imaginary quadratic field (\Cref{cor:CMfails}). The precise statement is the following.

\begin{thm}
\label{thm:roots-of-unity}
    Let $K$ be an imaginary quadratic field and let $S$ be the set of primes of~$K$ lying over a fixed rational prime $l$. Let $p\ne l$ be a prime which splits completely in $K$. Assume the $p$-adic period conjecture. Then the polylogarithmic Chabauty--Kim locus $X(\mathcal{O}_K \otimes \mathbb{Z}_p)_{S,\PL,\infty}$ contains all pairs $(\zeta,\zeta^{-1})$ where $\zeta \neq 1$ is a root of unity in $\mathbb{Z}_p$.
\end{thm}

\begin{proof}
Note that $\cO_{K,S}^{\times}$ contains the rank~1 subgroup $\cO_{\bQ,\{l\}}^{\times} = \bZ[1/l]^{\times}$ which is fixed by the Galois action, thus we can choose our basis $\Sigma_1 = \{\tau_1,\ldots,\tau_d\}$ of $\fn(\cO_{K,S})_{-1} = (\cO_{K,S}^{\times} \otimes_\bZ \bQ)^{\vee}$ (where $d \in \{1,2\})$ in such a way that $\tau_1$ is fixed by $\Gal(K/\bQ)$. (Concretely, we can take $\tau_1\coloneqq v_l \circ \mathrm{Nm}_{K/\bQ}$, the composition of the norm homomorphism $\cO_{K,S}^\times \otimes \bQ \to \bZ[1/l]^\times \otimes \bQ$ and the $l$-adic valuation.) For $n > 1$ we choose our Lie algebra generator $\Sigma_n = \{\sigma_n\}$ in degree~$-n$ in such a way that complex conjugation $\sigma \in \Gal(K/\bQ)$ acts by $\sigma^*\sigma_n = (-1)^{n+1}\sigma_n$, which is possible by \Cref{lem:sigmasigman}.
    In order to show that $(\zeta,\zeta^{-1})\in X(\mathcal{O}_K\otimes \mathbb{Z}_p)\simeq X(\mathbb{Z}_p)\times X(\mathbb{Z}_p)$ lies in $X(\mathcal{O}_K\otimes \mathbb{Z}_p)_{S,\PL,\infty}=j_p^{-1}(\loc_p(\mathrm{Sel}))$ for a root of unity $\zeta \ne 1$, it is enough to find some $\xi\in \mathrm{Sel} \coloneqq \Sel_{S,\PL,\infty}^{\mot}(X)$ such that $$\loc_p(\xi)=j_p((\zeta,\zeta^{-1})).$$

    We show that there exists such $\xi$ in the Selmer scheme with $x_i(\xi)=0$ for all $1\leq i\leq d$ and $y_i(\xi)$, $z_n(\xi)$ determined as follows. 
    Write the first period element
		 \[ \varepsilon_1 = \left(\sum_{i=1}^{d}c_{1,i}\tau_i\right) + c_2 \sigma_2 + c_3 \sigma_3 + c_4 \sigma_4 + \ldots \quad \bmod \text{commutators} \]
    Note that $\xi$ maps commutators to zero if $x_i(\xi)=0$ for all $1\leq i\leq d$. Indeed, both $\xi(\tau_{1,i}) = y_i(\xi) e_1$ and $\xi(\sigma_n) = z_n(\xi) \ad(e_0)^{n-1}e_1$ have $e_1$-degree $\geq 1$, so any commutator is mapped to an element of $e_1$-degree $\geq 2$, which is zero in the polylogarithmic quotient. So ~$\ev_{\varepsilon_1}(\xi)$ is given by
     \[ \ev_{\varepsilon_1}(\xi) = \left(0, \sum_{i=1}^{d}c_{1,i}y_i(\xi), c_2 z_2(\xi), c_3z_3(\xi), c_4 z_4(\xi),\ldots\right) \]
    in the coordinates $L_0,L_1,L_2,\ldots$ dual to $e_0,e_1,[e_0,e_1],\ldots$
     By \Cref{sigma-on-period-point}, the second period element $\varepsilon_2$ is related to the first by the Galois action: $\varepsilon_2 = \sigma^* \varepsilon_1$, where $\sigma \in \Gal(K/\bQ)$ is complex conjugation. Since $\sigma^* \sigma_n = (-1)^{n+1} \sigma_n$ for $n\geq 2$, we get
    \[ \varepsilon_2 = \sigma^* \varepsilon_1 = \left(\sum_{i=1}^{d}c_{1,i}\sigma^*\tau_i\right) - c_2 \sigma_2 + c_3 \sigma_3 - c_4 \sigma_4 \pm \ldots \quad \bmod \text{commutators}. \]
    Note that the coefficient of $e_1$ in $\xi(\sigma^*\tau_i)$ is $(-1)^{i+1}y_i(\xi)$. Therefore $\ev_{\varepsilon_2}(\xi)$ is given by
    \[ \ev_{\varepsilon_2}(\xi) = \left(0, \sum_{i=1}^{d}c_{1,i}(-1)^{i+1}y_i(\xi), -c_2 z_2(\xi), c_3z_3(\xi), -c_4 z_4(\xi),\ldots\right). \]
    Recall that we want to construct an element $\xi$ such that $\ev_{\eps_1}(\xi) = (L_n(\zeta))_{n \geq 0}$ and $\ev_{\eps_2}(\xi) = (L_n(\zeta^{-1})_{n \geq 0}$. Thus, given that $x_i(\xi) = 0$ for $1 \leq i \leq d$, the remaining coordinates $y_i(\xi)$, $z_n(\xi) \in \bQ_p$ must be chosen so as to satisfy the following equations:
    \begin{equation}
    \label{eq:c}
        \begin{cases}
            0=\log(\zeta), & \\
            0=\log(\zeta^{-1}), &\\
            \sum_{i=1}^{d}c_{1,i}y_i(\xi) = \Li_1(\zeta), & \\
            \sum_{i=1}^{d}c_{1,i}(-1)^{i+1}y_i(\xi) = \Li_1(\zeta^{-1}), &\\
            c_n z_n(\xi)=L_n(\zeta), & n\geq 2,\\
        	 	(-1)^{n+1}c_n z_n(\xi)=L_n(\zeta^{-1}), & n\geq 2.
        \end{cases}
    \end{equation}
First note that the first two lines in \eqref{eq:c} hold trivially since $\zeta$ is a root of unity. Also by Coleman--Sinnott (Lemma \ref{coleman-sinnott}), one has $L_n(\zeta)+(-1)^{n}L_n(\zeta^{-1})=0$ since $\zeta$ is a root of unity. Thus, the last line in \eqref{eq:c} is implied by the second-to-last line. So the only equations are 
    \begin{equation*}
\begin{cases}
\sum_{i=1}^{d}c_{1,i}y_i(\xi) = \Li_1(\zeta), & \\
\sum_{i=1}^{d}c_{1,i}(-1)^{i+1}y_i(\xi) = \Li_1(\zeta^{-1}), &\\
c_n z_n(\xi)=L_n(\zeta), & n\geq 2
\end{cases}
\end{equation*}
By the $p$-adic period conjecture, $c_{1,1}\ne 0 $ and $c_n\ne 0$ for $n\geq 2$ (\Cref{period-conjecture-implies-coeff-nonzero}). Also note $\Li_1(\zeta^{-1})=\Li_1(\zeta)$ since $\zeta$ is a root of unity. So we solve 
\begin{align*}
    y_i(\xi)&\coloneqq\begin{cases}
        \Li_1(\zeta)/c_{1,1}, & i=1,\\
        0, & \text{otherwise,}
    \end{cases}\\
    z_n(\xi)&\coloneqq L_n(\zeta)/c_n \quad \text{ for }n\geq 2. \qedhere
\end{align*}
\end{proof}

\begin{cor}
\label{cor:CMfails}
    Let $K$ be a number field that contains an imaginary quadratic field and suppose that $S$ contains all primes of $K$ lying over some rational prime $l$. Assume the $p$-adic period conjecture. For any totally split prime~$p \neq l$ and any $\fp \mid p$, the Chabauty--Kim locus $X(\cO_{\fp})_{S,\PL,\infty}$ contains all roots of unity $\neq 1$ in~$\bZ_p$. In particular, Kim's Conjecture \ref{conj:kim-single} fails for the polylogarithmic quotient when $p$ is large. 
\end{cor}
\begin{proof}
    Let $S_l(K)$ be the set of primes of~$K$ dividing~$l$. Let $K_0$ be an imaginary quadratic subfield of~$K$, which exists by assumption, and define $S_l(K_0)$ similarly. Let $\fp_0$ be the prime of~$K_0$ under~$\fp$. Then we have inclusions
    \[ X(\cO_{\fp_0})_{S_l(K_0),\PL,\infty} \subseteq X(\cO_{\fp})_{S_l(K),\PL,\infty} \subseteq X(\cO_{\fp})_{S,\PL,\infty} \]
    by \Cref{field-extension-functoriality} and by the assumption $S_l(K) \subseteq S$. But the first locus contains all non-trivial roots of unity in~$\bZ_p$ by \Cref{thm:roots-of-unity}. Since $K$ contains only finitely many roots of unity, these do not belong to $X(\cO_{K,S})$ when $p$ is large, so Kim's Conjecture fails.
\end{proof}

\section{Sets $S$ of size~1 over imaginary quadratic fields}
\label{sec:S-of-size-one}

In this section we consider an imaginary quadratic field~$K$ and a set~$S = \{\fl\}$ containing a single prime ideal of~$K$. The following proposition completely describes the solutions to the $S$-unit equation in this case.

\begin{prop}
\label{S=1-stable-solutions}
    For $K$ and $S = \{\fl\}$ as above, the solutions of the $S$-unit equation are given as follows:
    \begin{enumerate}[label=(\arabic*)]
        \item \label{item:Qzeta3-1-zeta3}
        if $K = \bQ(\zeta_3)$ and $\fl = (1-\zeta_3)$:
        \[ X(\cO_{K,S}) = \left\{ \zeta_6,\, \zeta_6^{-1}, \, \zeta_3,\, \zeta_3^{-1},\, 1-\zeta_3,\, 1-\zeta_3^{-1},\,\frac1{1-\zeta_3},\, \frac1{1-\zeta_3^{-1}} \right\}; \]
        \item \label{item:Qzeta3-2}
        if $K = \bQ(\zeta_3)$ and $\fl = (2)$:
        \[ X(\cO_{K,S}) = \bigl\{2,\, -1,\, \tfrac12,\, \zeta_6,\, \zeta_6^{-1} \bigr\};\]
        \item \label{item:Qzeta3-general}
        if $K = \bQ(\zeta_3)$ and $\fl \neq (1-\zeta_3), (2)$:
        \[ X(\cO_{K,S}) = X(\cO_K) = \bigl\{\zeta_6,\, \zeta_6^{-1} \bigr\};\]
        \item \label{item:Qi}
        if $K = \bQ(i)$ and $\fl = (1-i)$:
        \[ X(\cO_{K,S}) = \left\{ 2,\, -1,\, \tfrac12,\, \pm i, 1\pm i, \frac{1 \pm i}{2}\right\}; \]
        \item \label{item:l-divides-2}
        if $K \neq \bQ(i), \bQ(\zeta_3)$, $\fl \mid 2$, and $S$ is Galois-stable, then
        \[ X(\cO_{K,S}) = X(\bZ[1/2]) = \{2,-1,\tfrac12\}; \]
        \item in all other cases, $X(\cO_{K,S}) = \emptyset$.
    \end{enumerate}
\end{prop}

\begin{proof}
    One checks easily that the listed points are $S$-integral points in each case. For the other direction, let $z \in X(\cO_{K,S})$. Note that $X(\cO_{K,S})$ is a union of $S_3$-orbits for the $S_3$-action generated by $z \mapsto 1/z$ and $z \mapsto 1-z$. Replacing $z$ with $1/z$ if necessary we may assume $v_{\fl}(z) \geq 0$. Then, replacing $z$ with $1-z$ if necessary we may assume $v_{\fl}(z) = 0$, i.e.\ $z$ is a unit in $\cO_K$. Since $K$ is imaginary quadratic, $z$ must be a root of unity of order $\in \{2,3,4,6\}$. If $z = -1$, we have $1-z = 2$, and $\fl$ must divide~2 which cannot split completely in $K$ for this to be an $S$-unit, so we are in case~\ref{item:Qzeta3-2}, \ref{item:Qi}, or~\ref{item:l-divides-2}. If $z = \zeta_3$ has order~$3$, we must have $K = \bQ(\zeta_3)$ and $\fl$ must be $(1-\zeta_3)$ for $1-z$ to be an $S$-unit, so we are in case~\ref{item:Qzeta3-1-zeta3}.
    If $z = i$ has order~$4$, we must have $K = \bQ(i)$ and $\fl$ must be $(1-i)$ for $1-z$ to be an $S$-unit, so we are in case~\ref{item:Qi}. If $z = \zeta_6$ we have $K = \bQ(\zeta_3)$ and we are in case~\ref{item:Qzeta3-1-zeta3}, \ref{item:Qzeta3-2}, or \ref{item:Qzeta3-general}.
\end{proof}

Let $p$ be a prime not divisible by~$\fl$ which splits in~$K$. Let $\fp_1$ and $\fp_2$ be the two prime ideals of~$K$ dividing~$p$. Let $\pi_{\fl}$ be a generator of the 1-dimensional space $\cO_{K,S}^{\times} \otimes \bQ$, which determines a dual Lie algebra element $\tau_{\fl} \in \Lie(U_S^{\MT})_{-1}$. Let $\eps_i \coloneqq \eps_{\fp_i}$ be the $\fp_i$-adic period element ($i = 1,2$). We distinguish two cases, depending on whether $S$ is Galois-stable or not.
 
\subsection{Galois-stable case}
\label{sec:Galois-stable}
We first treat the case where $S = \{\fl\}$ is Galois-stable. Hence, $\fl$ is the unique prime lying over a rational prime which is either ramified or inert in~$K$. 

We derive the equations defining the polylogarithmic Chabauty--Kim locus up to depth~4. By \Cref{lem:sigmasigman}, we can choose the Lie algebra generators $\sigma_n \in \Lie(U_S^{\MT})_{-n}$ such that $\sigma^*\sigma_n = (-1)^{n+1} \sigma_n$ for $n \geq 2$, where $\sigma \in \Gal(K/\bQ)$ is the complex conjugation. We adopt this choice in the following. 
According to \Cref{lie-algebra-element-expansion}, we can write
\begin{equation}
    \label{eq:eps-expansion}
    \eps_1 = c_{\tau_{\fl}} \tau_{\fl} + c_{\sigma_2} \sigma_2 + c_{\sigma_3} \sigma_3 + c_{\tau_{\fl} \sigma_2} [\tau_{\fl}, \sigma_2] + c_{\sigma_4} \sigma_4 + c_{\tau_{\fl}\sigma_3} [\tau_{\fl},\sigma_3] + c_{\tau_{\fl} \tau_{\fl} \sigma_2} [\tau_{\fl}, [\tau_{\fl}, \sigma_2]] + \ldots 
\end{equation}
where the omitted elements belong to the Goncharov ideal or have degree $< -4$.

\begin{thm}
\label{S=1-stable-equations}
    Let $K$ be an imaginary quadratic field and assume that $S = \{\fl\}$ is Galois-stable. Let $p$ be a rational prime which splits in~$K$. Then, for $n=1,2,3,4$, the depth-$n$ polylogarithmic Chabauty--Kim locus $X(\cO_K \otimes \bZ_p)_{S,\PL,n}$ is contained in the set of pairs $(z_1,z_2) \in X(\bZ_p) \times X(\bZ_p)$ satisfying the following incremental list of equations:
    \begin{itemize}
        \item in depth~1:
        \begin{equation}
            \label{eq:S=1-stable-log}
            \log(z_1) = \log(z_2) \quad \text{and} \quad \log(1-z_1) = \log(1-z_2);
        \end{equation}
        \item in depth~2:
        \begin{equation}
            \label{eq:S=1-depth-2}
            L_2(z_1) + L_2(z_2) = 0;
        \end{equation}
        \item in depth 3:
         \begin{equation}
            \label{eq:S=1-depth-3}
            c_{\tau_{\fl}} c_{\sigma_2}(L_3(z_1) - L_3(z_2)) = 2c_{\tau_{\fl}\sigma_2} \log(z_1) L_2(z_1);
        \end{equation}
        \item in depth 4:
         \begin{equation}
            \label{eq:S=1-depth-4}
            c_{\tau_{\fl}} c_{\sigma_3}(L_4(z_1) + L_4(z_2)) = c_{\tau_{\fl}\sigma_3}\log(z_1)(L_3(z_1) + L_3(z_2)).
        \end{equation}
    \end{itemize}
    The containment is an equality if the $p$-adic period conjecture holds for the primes dividing~$p$.
\end{thm}

\begin{proof}
We have to determine the image of the product of evaluation maps
\begin{equation*}
    \ev_{\eps_1} \times \ev_{\eps_2}\colon \Hom_{\gr}(\Lie(U_S^{\MT}),\Lie(\Pi_{\PL,4}^{\omega})) \to \Lie(\Pi^{\dR}_{\PL,4})_{\bQ_p} \times \Lie(\Pi_{\PL,4}^{\dR})_{\bQ_p}
\end{equation*}
The Lie algebra $\Lie(\Pi^{\dR}_{\PL,4})_{\bQ_p}$ is 5-dimensional with basis $e_0, e_1,\ldots,\ad(e_0)^3 e_1$; let $L_0,\ldots,L_4$ be the dual basis. 
The global Selmer scheme is an affine space with coordinates $x_{\fl},y_{\fl},z_2,z_3,z_4$. By \Cref{cocycle-evaluation-map}, the cocycle evaluation map $\ev_{\eps_1}$ is given in these coordinates by
\begin{equation}
\label{eq:S=1-stable-ev-eps1}
\begin{split}
    &\ev_{\eps_1}\colon (x_{\fl},y_{\fl},z_2,z_3,z_4) \mapsto \\
    &\qquad (c_{\tau_{\fl}}x_{\fl}, \;\; c_{\tau_{\fl}}y_{\fl}, \;\; c_{\sigma_2}z_2, \;\; c_{\sigma_3}z_3+c_{\tau_{\fl} \sigma_2}x_{\fl}z_2, \;\; c_{\sigma_4}z_4+c_{\tau_{\fl}\sigma_3}x_{\fl}z_3+c_{\tau_{\fl} \tau_{\fl} \sigma_2} x_{\fl}^2 z_2).
\end{split}
\end{equation}
We exploit the Galois action to relate the period elements $\eps_1$ and $\eps_2$ to each other. By \Cref{sigma-on-period-point}, we have $\eps_2 = \sigma^* \eps_1$. Since $\fl$ is stable under the Galois action, $\sigma^*\tau_{\fl}=\tau_{\fl}$. Applying $\sigma^*$ to \eqref{eq:eps-expansion} and using $\sigma^*\sigma_n=(-1)^{n+1}\sigma_n$, the second period element $\eps_2$ is given by
\begin{equation*}
    \varepsilon_2 = c_{\tau_{\fl}} \tau_{\fl} - c_{\sigma_2} \sigma_2 + c_{\sigma_3} \sigma_3 - c_{\tau_{\fl} \sigma_2} [\tau_{\fl}, \sigma_2] - c_{\sigma_4} \sigma_4 + c_{\tau_{\fl}\sigma_3} [\tau_{\fl},\sigma_3] - c_{\tau_{\fl} \tau_{\fl} \sigma_2} [\tau_{\fl}, [\tau_{\fl}, \sigma_2]] + \ldots
\end{equation*}
Hence, the cocycle evaluation map $\ev_{\eps_2}$ for this element is given by
\begin{equation}
\label{eq:S=1-stable-ev-eps2}
\begin{split}
    &\ev_{\eps_2}\colon (x_{\fl},y_{\fl},z_2,z_3,z_4) \mapsto\\
    &\qquad (c_{\tau_{\fl}}x_{\fl}, \; c_{\tau_{\fl}}y_{\fl}, \; -c_{\sigma_2}z_2, \; c_{\sigma_3}z_3-c_{\tau_{\fl} \sigma_2}x_{\fl}z_2, \; -c_{\sigma_4}z_4+c_{\tau_{\fl}\sigma_3}x_{\fl}z_3-c_{\tau_{\fl} \tau_{\fl} \sigma_2} x_{\fl}^2 z_2).
\end{split}
\end{equation}
Comparing \eqref{eq:S=1-stable-ev-eps1} and~\eqref{eq:S=1-stable-ev-eps2}, the following equations hold on the image of $\ev_{\eps_1} \times \ev_{\eps_2}$:
\begin{gather*}
    L_0^{(1)} = L_0^{(2)}, \qquad L_1^{(1)} = L_1^{(2)}, \\ L_2^{(1)} + L_2^{(2)} = 0,\\
    c_{\tau_{\fl}} c_{\sigma_2}(L_3^{(1)} - L_3^{(2)}) = 2c_{\tau_{\fl}\sigma_2} L_0^{(1)} L_2^{(1)},\\
    c_{\tau_{\fl}} c_{\sigma_3}(L_4^{(1)} + L_4^{(2)}) = c_{\tau_{\fl}\sigma_3}(L_3^{(1)} + L_3^{(2)}) L_0^{(1)}.
\end{gather*}
Here, $L_n^{(i)}$ denotes $L_n$ composed with the $i$-th projection of $\Lie(\Pi_{\PL,4}^{\dR}) \times \Lie(\Pi_{\PL,4}^{\dR})$ for $i=1,2$. Pulling back these equations to $X(\bZ_p) \times X(\bZ_p)$ results in the claimed equations. If the $p$-adic period conjecture holds then the periods $c_{\sigma_2}$, $c_{\sigma_3}$, $c_{\sigma_4}$ are $\neq 0$ in $\bQ_p$ (\Cref{period-conjecture-implies-coeff-nonzero}), and in this case the image of $\ev_{\eps_1} \times \ev_{\eps_2}$ is precisely cut out by the given equations.
\end{proof}

\begin{rem}
\label{checking-period-conjecture}
    In \Cref{S=1-stable-equations}, only a small fragment of the $p$-adic period conjecture, namely the non-vanishing of the three $p$-adic periods $c_{\sigma_2}, c_{\sigma_3}, c_{\sigma_4}$, is needed to conclude that the loci $X(\cO_K \otimes \bZ_p)_{S,\PL,n}$ are precisely cut out by the given equations. In the case of $c_{\sigma_3}$, we can choose $\zeta^{\fu}(3)$ as a basis of the 1-dimensional extension space $E_3(\bQ) = E_3(K)$, so that $c_{\sigma_3} = \zeta(3)$ when $\sigma_3$ is chosen to be dual to $\zeta^{\fu}(3)$ (see \Cref{single-letter-periods-from-En-basis}). The non-vanishing of $\zeta(3)$ in $\bQ_p$ is known when $p$ is a regular prime, see \cite[Remark~2.20(i)]{furusho:p-adic}, and in any case it can quickly be checked on a computer for any given~$p$. In SageMath \cite{sagemath}, the $p$-adic zeta value $\zeta(n)$ can be computed as \texttt{Qp(p)(1).polylog(n)}. It is more difficult to check the non-vanishing of $c_{\sigma_2}$ and $c_{\sigma_4}$ since we do not have a general construction of a (conjectured) basis of $E_2(K)$ and $E_4(K)$ for arbitrary~$K$. In the case $K = \bQ(i)$ or $K = \bQ(\zeta_3)$, however, we have $\Li_n^{\fu}(i) \in E_n(\bQ(i))$ and $\Li_n^{\fu}(\zeta_3) \in E_n(\bQ(\zeta_3))$ by \Cref{zeta-and-cyclotomic-polylogs-in-En}, and we can check for any given~$p$ with $i \in \bQ_p$ resp.\ $\zeta_3 \in \bQ_p$ that the $p$-adic polylogarithms $\Li_n(i)$ and $\Li_n(\zeta_3)$ are non-zero. For $\zeta \in \{i,\zeta_3\}$, checking the non-vanishing of $\Li_n(\zeta) \in \bQ_p$ for a single prime~$p$ is enough to conclude that the motivic polylogarithm $\Li_n^{\fu}(\zeta)$ is nonzero. Alternatively, one can deduce this from the fact that $E_n(Z)$ is spanned by $\Li_n^{\fu}$ at roots of unity (see \Cref{cyclotomic-En-basis}). Hence we can take $\Li_n^{\fu}(\zeta)$ as a basis of the one-dimensional space $E_n(\bQ(\zeta))$, choose $\sigma_n \in \Lie(U_S^{\MT})_{-n}$ to be dual to it, and have $c_{\sigma_n} = \Li_n(\zeta)$ in~$\bQ_p$ for all~$p$. We checked that $\Li_n(\zeta) \neq 0$ in~$\bQ_p$ for $\zeta \in \{i,\zeta_3\}$, $n = 2,4$, and all primes $p < 10{,}000$ with $\zeta \in \bQ_p$. 
\end{rem}

Note that the equations in depth~1 and~2 in \Cref{S=1-stable-equations} do not depend on the field~$K$ and set~$S$. Since $\Pi_{\PL,2} = \Pi_2$, the polylogarithmic Chabauty--Kim locus in depth $n \leq 2$ agrees with the full depth-$n$ locus:
\[ X(\cO_K \otimes \bZ_p)_{S,\PL,n} = X(\cO_K \otimes \bZ_p)_{S,n} \quad \text{for $n = 1,2$.} \]
We analyse the loci $X(\cO_K \otimes \bZ_p)_{S,\PL,n}$ in small depth, starting with the equations~\eqref{eq:S=1-stable-log} for $X(\cO_K \otimes \bZ_p)_{S,1}$.

\begin{lemma}
\label{lem:logLi1eq}
    The solutions of the equations $\log(z_1)=\log(z_2)$, $\mathrm{Li}_1(z_1)=\mathrm{Li}_1(z_2)$ for $z_1,z_2\in\mathbb{Z}_p$ with $z_1,z_2 \not\equiv 0,1 \bmod p$ are those $(z_1,z_2)$ with $z_1=z_2$ or 
    \begin{equation}
    \label{eq:root-of-unity-parametrisation}
        (z_1,z_2)=\left(\frac{1-\eta}{\zeta-\eta},\zeta\frac{1-\eta}{\zeta-\eta}\right)
    \end{equation}
    for roots of unity $\zeta,\eta \in \bZ_p$ with $\zeta,\eta\ne 1$ and $\zeta\ne \eta$. 
\end{lemma}

\begin{proof}
    It is clear that the diagonal pairs $(z,z)$ satisfy the equations. Suppose that $(z_1,z_2) \in X(\bZ_p) \times X(\bZ_p)$ is a solution with $z_1 \neq z_2$.
    The equations imply 
    $$\log\left(\frac{z_1}{z_2}\right)=0,\quad \log\left(\frac{1-z_1}{1-z_2}\right)=0,$$
    hence both $z_2/z_1$ and $(1-z_2)/(1-z_1)$ are roots of unity. Denote them by $\zeta$ and $\eta$, respectively. The assumption $z_1 \neq z_2$ implies $\zeta, \eta \neq 1$ and $\zeta\ne \eta$. The equations 
    \[ z_2/z_1 = \zeta, \quad (1-z_2)/(1-z_1) = \eta \]
    can be solved for $z_1,z_2$ with linear algebra, resulting in the claimed solutions in terms of $\zeta,\eta$. Conversely, each pair $(z_1,z_2)$ of this form does satisfy the original equations.
\end{proof}

We refer to the two types of solutions to the depth-1 equations as the \emph{diagonal solutions} and \emph{off-diagonal solutions} respectively. 
Due to the diagonal solutions, the depth-1 locus in \Cref{S=1-stable-equations} is infinite. However, it becomes finite in depth~2.

\begin{thm}
    \label{S=1-stable-finiteness}
    Let $K$ be an imaginary quadratic field and assume that $S = \{\fl\}$ is Galois-stable. Let $p$ be a rational prime which splits in~$K$.
    Then $\#X(\cO_K \otimes \bZ_p)_{S,2} < \infty$.
\end{thm}

\begin{proof}
    By \Cref{lem:logLi1eq}, the depth-1 locus $X(\cO_K \otimes \bZ_p)_{S,1}$ contains diagonal solutions $(z,z)$ and finitely many solutions of the form~\eqref{eq:root-of-unity-parametrisation} parametrised by pairs of roots of unity. Among the diagonal solutions, only those $(z,z)$ with $L_2(z) = 0$ satisfy the depth-2 equation~\eqref{eq:S=1-depth-2}. But $L_2(z)$ has only finitely many roots.
\end{proof}

There are a number of solutions coming from roots of unity which can be proved to satisfy the equations~\eqref{eq:S=1-stable-log}--\eqref{eq:S=1-depth-2} for the depth-2 locus.

\begin{thm}
    \label{S=1-stable-depth2}
    For any prime~$p$, the set
    \begin{equation}
    \label{eq:depth2-off-diag-subset}
        \bigcup_{1 \neq \zeta \in \mu(\bZ_p)} \left\{ (\zeta,\zeta^{-1}), (1-\zeta,1-\zeta^{-1}), \,  \Bigl(\frac1{1-\zeta},\frac1{1-\zeta^{-1}}\Bigr)\right\}
    \end{equation}
    with $\zeta$ running over the non-trivial roots of unity in~$\bZ_p$, is contained in the locus cut out in $X(\bZ_p) \times X(\bZ_p)$ by the equations
    \[ \log(z_1) = \log(z_2),\quad \log(1-z_1) = \log(1-z_2), \quad L_2(z_1) + L_2(z_2) = 0. \]
\end{thm}

\begin{proof}
    The listed pairs clearly satisfy the depth-$1$ equations in terms of $\log$. From the functional equations $L_2(1/z) = -L_2(z)$ (\Cref{coleman-sinnott-for-Ln}) and $L_2(1-z) = -L_2(z)$ (\Cref{L2-reflection-formula}) we get
    \[ L_2(z) + L_2(z^{-1}) = L_2(1-z) + L_2(1-z^{-1}) = L_2\left(\frac1{1-z}\right) + L_2\left(\frac1{1-z^{-1}}\right) = 0 \]
    for all $z \in X(\bZ_p)$, hence the third equation also holds for the listed pairs.
\end{proof}

\begin{rem}
    \label{S=1-stable-depth2-explanation}
    Modulo the $p$-adic period conjecture, \Cref{S=1-stable-depth2} can be interpreted as saying that the depth-2 locus $X(\cO_K \otimes \bZ_p)_{S,2}$ in \Cref{S=1-stable-equations} contains the $S_3$-orbits of the elements $(\zeta,\zeta^{-1})$ with $\zeta \neq 1$ a root of unity in~$\bZ_p$. This is consistent with the fact that the depth-2 locus is stable under the $S_3$-action (see e.\,g.\ \cite[§4.6]{BKL:chabauty-kim-sc}) and contains the pairs $(\zeta,\zeta^{-1})$ by \Cref{thm:roots-of-unity}.
\end{rem}

\Cref{S=1-stable-depth2} gives a lower bound on the number of off-diagonal pairs in the depth-2 locus $X(\cO_K \otimes \bZ_p)_{S,2}$, which we found to be sharp in all cases we computed.

\begin{thm}
    \label{S=1-stable-depth2-bound}
    For any prime~$p > 2$, the number of pairs $(z_1,z_2) \in X(\bZ_p) \times X(\bZ_p)$ with $z_1 \neq z_2$ satisfying the equations \eqref{eq:S=1-stable-log}--\eqref{eq:S=1-depth-2} is
    \[  \geq \begin{cases}
        3p-9, & \text{if $p \not\equiv 1 \bmod 3$},\\
        3p-13, & \text{if $p \equiv 1 \bmod 3$.}
    \end{cases}\]
    The bound is sharp for all $p < 2{,}000$.
\end{thm}

\begin{proof}
    By \Cref{S=1-stable-depth2}, the solutions to \eqref{eq:S=1-stable-log}--\eqref{eq:S=1-depth-2} include the $S_3$-orbit of $(\zeta,\zeta^{-1})$ for every non-trivial root of unity~$\zeta \in \bZ_p$. In the union~\eqref{eq:depth2-off-diag-subset}, the contributions from $\zeta$ and $\zeta^{-1}$ together form the $S_3$-orbit of~$(\zeta,\zeta^{-1})$. Hence, \eqref{eq:depth2-off-diag-subset} is a union of $(p-1)/2$ disjoint orbits. Each has size~6, except when $\zeta$ has a non-trivial $S_3$-stabiliser. This happens when $\zeta = -1$ or when $\zeta$ is a primitive sixth root of unity. For $\zeta = -1$, the $S_3$-orbit is $\{(-1,-1),(2,2),(\tfrac12,\tfrac12)\}$ but these lie on the diagonal. When $p \not\equiv 1 \bmod 3$ then the sixth roots of unity are not contained in~$\bZ_p$, hence we get a lower bound of $((p-1)/2 - 1)\cdot 6 = 3p-9$ for the number of off-diagonal solutions. If $p \equiv 1 \bmod 3$, the exceptional orbit $\{(\zeta_6,\zeta_6^{-1}), (\zeta_6^{-1},\zeta_6)\}$ only has length~2 instead of~6, resulting in a lower bound of $(3p-9)-4 = 3p-13$.

    For any given~$p$ the solutions to \eqref{eq:S=1-stable-log}--\eqref{eq:S=1-depth-2} can easily be computed: generate all off-diagonal solutions to the depth-1 equations from pairs of roots of unity via \Cref{lem:logLi1eq} (which gives $(p-2)(p-3)$ pairs), then check which of them also satisfy $L_2(z_1) + L_2(z_2) = 0$. We checked this for all $p < 2{,}000$ and found that in each case, only the off-diagonal solutions coming from \Cref{S=1-stable-depth2} survive in depth~2. The Sage code for these computations is available at \url{https://github.com/martinluedtke/PolylogNF}. 
\end{proof}

We now turn attention to the diagonal solutions in the depth-2 locus $X(\cO_K \otimes \bZ_p)_{S,2}$ of \Cref{S=1-stable-equations}. These are the pairs $(z,z)$ with $z \in X(\bZ_p)$ satisfying $L_2(z) = 0$. Recall from \Cref{def:Ln} that 
\[ L_2(z) = \Li_2(z) - \tfrac12 \log(z) \Li_1(z). \]
The zeros of functions of the form $\Li(z) - a \log(z)\Li_1(z)$ were previously analysed in \cite[§5]{luedtke_2025_refined}.

\begin{lemma}[{\cite[§5.4]{luedtke_2025_refined}}]
    \label{L2-roots-general}
    Let $\zeta \neq 1$ be a nontrivial root of unity in~$\bZ_p$. If $v_p(\Li_2(\zeta)) = 2$ and $v_p(\Li_1(\zeta)) = 1$ then the function $L_2(z)$ has exactly one root in the residue disc of~$\zeta$.
\end{lemma}

Since $v_p(\Li_2(\zeta)) = 2$ and $v_p(\Li_1(\zeta)) = 1$ holds for most $\zeta$'s, the function $L_2(z)$ typically has $\approx p-2$ roots in $X(\bZ_p)$, which we confirmed numerically for many choices of~$p$. 

For small primes~$p$ we can calculate the depth-2 locus by hand and confirm Kim's Conjecture in a few cases. As noted in \cite[§8.2]{BDCKW}, the roots of $L_2(z)$ for $p \in \{3,5,7\}$ are given as follows.

\begin{lemma}
\label{lem:L2-roots}
    Let $z\in\mathbb{Z}_p$ with $z \not\equiv 0,1 \bmod p$. For $p \in \{3,5,7\}$, $L_2(z)=0$ if and only if $z=2$ or $-1$ or $1/2$.
\end{lemma}

\begin{proof}
The fact that $L_2$ vanishes on $2$, $-1$, $\tfrac12$ follows from the functional equations $L_2(z^{-1}) = - L_2(z)$ (\Cref{coleman-sinnott-for-Ln}) and $L_2(1-z) = -L_2(z)$ (\Cref{L2-reflection-formula}), or by taking $\zeta = -1$ in \Cref{S=1-stable-depth2}. For $p = 3,5,7$ one can check on a computer that $L_2$ has no other roots. Alternatively, one can analyse the Newton polygon using \cite[Lemma~5.1]{luedtke_2025_refined}.

\end{proof}

\begin{thm}	
\label{thm:locusforimquadSsingle}
	Let $K$ be an imaginary quadratic field and assume that $S=\left\{\mathfrak{l}\right\}$ is Galois-stable. 
    \begin{enumerate}[label=(\alph*)]
        \item \label{item:p=3}
        If $p=3$ splits in~$K$ then $X(\cO_K \otimes \bZ_3)_{S,2}\subseteq \left\{(2,2),(-1,-1),\left(\frac{1}{2},\frac{1}{2}\right)\right\}$.
        \item \label{item:p=5}
        If $p=5$ splits in~$K$ then 
        $$X(\cO_K \otimes \bZ_5)_{S,2}\subseteq \left\{(2,2),(-1,-1),\left(\tfrac{1}{2},\tfrac{1}{2}\right),(\pm i,\mp i),(1\pm i,1\mp i),\left(\frac{1\pm i}{2},\frac{1\mp i}{2}\right)\right\}$$
        where $i$ is a square root of $-1$ in $\mathbb{Z}_5$.
        \item \label{item:p=7}
        If $p=7$ splits in~$K$ then
        \[ X(\cO_K \otimes \bZ_7)_{S,2} \subseteq \left\{
        \begin{array}{l}
        (2,2),(-1,-1), \left(\tfrac12,\tfrac12\right), (\zeta_3, \zeta_3^{-1}), (\zeta_3^{-1},\zeta_3), (\zeta_6,\zeta_6^{-1}), (\zeta_6^{-1},\zeta_6), \\ (1-\zeta_3,1-\zeta_3^{-1}), (1-\zeta_3^{-1},1-\zeta_3), \Bigl(\frac1{1-\zeta_3},\frac1{1-\zeta_3^{-1}}\Bigr), \Bigl(\frac1{1-\zeta_3^{-1}},\frac1{1-\zeta_3}\Bigr)
        \end{array}
        \right\},\]
        where $\zeta_3$ is a primitive third root of unity in $\bZ_7$ and $\zeta_6 = -\zeta_3$.
    \end{enumerate}
\end{thm}
\begin{proof}
    If $(z,z)$ is a diagonal pair contained in $X(\cO_K \otimes \bZ_p)_{S,2}$ then $L_2(z) = 0$ by~\eqref{eq:S=1-depth-2}, but for $p \in \{3,5,7\}$ the only solutions are $2, \tfrac12, -1$ by \Cref{lem:L2-roots}. So it is enough to consider off-diagonal solutions, which according to \Cref{lem:logLi1eq} must be of the form~\eqref{eq:root-of-unity-parametrisation}, parametrised by pairs of roots of unity $\zeta, \eta \in \bZ_p$ with $\zeta,\eta \neq 1$ and $\zeta \neq \eta$. For $p = 3$ there are no two distinct and nontrivial roots of unity in~$\bZ_3$. For $p=5$ there are $3\cdot 2 = 6$ such pairs, resulting in precisely the solutions $(z_2,z_2) = (\pm i,\mp i),(1\pm i,1\mp i),\left(\frac{1\pm i}{2},\frac{1\mp i}{2}\right)$.
    (One can check that these also satisfy the depth-2 equation $L_2(z_1) + L_2(z_2) = 0$, either using functional equations of the dilogarithm, or using the fact that these constitute actual $S$-integral points for $K = \bQ(i)$ and $S = \{(1-i)\}$ and the depth-2 equation does not depend on $K$ and $S$.)
    For $p = 7$ there are $5\cdot 4 = 20$ such pairs but one checks computationally that the equation $L_2(z_1) + L_2(z_2) = 0$ holds at most for the 8 off-diagonal pairs in the set listed above. (The equation does indeed hold for these pairs since they are $S$-integral points for $K = \bQ(\zeta_3)$ and $S = \{(1-\zeta_3)\}$.)
\end{proof}

\begin{cor}
   Kim's Conjecture~\ref{conj:kim-full} holds in depth $2$ in the following cases:
    \begin{itemize}
        \item $K=\mathbb{Q}(\sqrt{-2})$, $S=\left\{(\sqrt{-2})\right\}$, $p=3$;
        \item $K=\mathbb{Q}(i)$, $S=\left\{(1-i)\right\}$, $p=5$.
    \end{itemize}
\end{cor}

\begin{proof}
    All the points in $X(\cO_K \otimes \bZ_p)_{S,2}$ given by \Cref{thm:locusforimquadSsingle} belong to $X(\mathcal{O}_{K,S})$, see \Cref{S=1-stable-solutions}.
\end{proof}

We now consider the polylogarithmic depth~3 and~4 loci in \Cref{S=1-stable-equations}, which unlike the depth-2 locus depend on~$K$ and~$S$. The difficulty arising here is that the equations involve certain constants $c_{\sigma_2}$, $c_{\sigma_3}$, $c_{\tau_{\fl} \sigma_2}$, $c_{\tau_{\fl} \sigma_3}$ which are hard to determine. However, we can compute them in some cases where we have enough $S$-integral points available. 

We know from \Cref{tau-coeffs} that $c_{\tau_{\fl}} \neq 0$, and assuming the $p$-adic period conjecture, we also have $c_{\sigma_2}, c_{\sigma_3} \neq 0$. In this case we set 
\begin{equation}
\label{eq:c3-c4}
    c_3 \coloneqq \frac{2c_{\tau_{\fl}\sigma_2}}{c_{\tau_{\fl}} c_{\sigma_2}}, \qquad c_4 \coloneqq \frac{c_{\tau_{\fl}\sigma_3}}{c_{\tau_{\fl}} c_{\sigma_3}},
\end{equation}
so that the equations~\eqref{eq:S=1-depth-3} and \eqref{eq:S=1-depth-4} can be written as
\begin{align}
    \label{eq:S=1-depth-3-simple}
    L_3(z_1) - L_3(z_2) &= c_3 \log(z_1) L_2(z_1), \\
    \label{eq:S=1-depth-4-simple}
    L_4(z_1) + L_4(z_2) &= c_4 \log(z_1) (L_3(z_1) + L_3(z_2)).
\end{align}

\begin{lemma}
\label{S=1-stable-coeffs-from-known-solution}
    In the situation of \Cref{S=1-stable-equations}, assume that we have an $S$-integral point $z \in X(\cO_{K,S})$ such that $\log(z) \neq 0$ and $L_2(z) \neq 0$ and $L_3(z) + L_3(\sigma z) \neq 0$. Then $c_{\sigma_2}, c_{\sigma_3} \neq 0$ and the constants $c_3$, $c_4$ in~\eqref{eq:S=1-depth-3-simple} and \eqref{eq:S=1-depth-4-simple} are given by
    \begin{equation}
    \label{eq:S=1-stable-coeffs-from-known-solutions}
    c_3 = \frac{L_3(z) - L_3(\sigma z)}{\log(z) L_2(z)}, \qquad 
    c_4 = \frac{L_4(z) + L_4(\sigma z)}{\log(z)(L_3(z) + L_3(\sigma z))}.
    \end{equation}
\end{lemma}
Here and in the following, $z \in X(\cO_{K,S})$ is viewed as an element of $X(\bZ_p)$ via the embedding $K \hookrightarrow K_{\fp_1} = \bQ_p$ defined by $\fp_1$. The image of $z$ in $X(\cO_K \otimes \bZ_p) \cong X(\bZ_p) \times X(\bZ_p)$ is thus the pair $(z_1,z_2) = (z,\sigma z)$.

\begin{proof}
    If $c_{\sigma_2}$ were zero then, by the formulas~\eqref{eq:S=1-stable-ev-eps1} and \eqref{eq:S=1-stable-ev-eps2} for the cocycle evaluation maps $\ev_{\eps_i}$ ($i=1,2$), the equations $L_2(z_1) = L_2(z_2) = 0$ would hold on $X(\cO_K \otimes \bZ_p)_{S,\PL,2}$. Similarly, if $c_{\sigma_3}$ were zero then $L_3(z_1) + L_3(z_2) = 0$ would hold on $X(\cO_K \otimes \bZ_p)_{S,\PL,3}$. But these equations do not hold for $(z_1,z_2) = (z,\sigma z)$ by assumption, hence we must have $c_{\sigma_2}, c_{\sigma_3} \neq 0$. Now the formulas for $c_3$ and $c_4$ follow from the equations~\eqref{eq:S=1-depth-3-simple} and~\eqref{eq:S=1-depth-4-simple} for~$(z_1,z_2) = (z,\sigma z)$.
\end{proof}

The two cases where the assumptions of \Cref{S=1-stable-coeffs-from-known-solution} can be verified are $K = \bQ(i)$, $S =\{(1-i)\}$ and $K = \bQ(\zeta_3)$, $S = \{(1-\zeta_3)\}$. In the first case we can take $z = 1-i$, in the second case $z = 1-\zeta_3$. In either case, we can verify that $L_2(z) \neq 0$ and $L_3(z) + L_3(\sigma z) \neq 0$ in~$\bQ_p$ for a single prime~$p$, and then the period conjecture predicts that this must be true for \emph{any}~$p$. For any given prime~$p$ this can be verified numerically. Then the constants $c_3$ and $c_4$ can be computed via~\eqref{eq:S=1-stable-coeffs-from-known-solutions} and one obtains the loci $X(\cO_K \otimes \bZ_p)_{S,\PL,3}$ and $X(\cO_K \otimes \bZ_p)_{S,\PL,4}$ as follows, distinguishing between diagonal and off-diagonal solutions:
\begin{itemize}
    \item The diagonal solutions in depth~2 are $(z,z)$ with $z$ a root of $L_2(z)$. These automatically satisfy the depth-3 equation~\eqref{eq:S=1-depth-3-simple}, so these are also the diagonal solutions in $X(\cO_K \otimes \bZ_p)_{S,\PL,3}$. By~\eqref{eq:S=1-depth-4-simple}, those which in addition satisfy
    \begin{equation*}
    \label{eq:S=1-stable-depth4-diagonal}
        L_4(z) = c_4 \log(z)L_3(z)
    \end{equation*}
    are the diagonal solutions in $X(\cO_K \otimes \bZ_p)_{S,\PL,4}$.
    \item The off-diagonal solutions in depth~1 are parametrised by pairs of roots of unity via \Cref{lem:logLi1eq}, and we simply check for each one whether the equations~\eqref{eq:S=1-depth-2}, \eqref{eq:S=1-depth-3-simple}, \eqref{eq:S=1-depth-4-simple} in depth~2, 3, 4 are satisfied to decide whether they belong to $X(\cO_K \otimes \bZ_p)_{S,\PL,n}$ for $n=2,3,4$.
\end{itemize}

We showed in \Cref{thm:roots-of-unity} that the pairs $(\zeta,\zeta^{-1})$ for $\zeta \neq 1$ a root of unity in~$\bZ_p$ will be contained in $X(\cO_K \otimes \bZ_p)_{S,\PL,n}$ for any depth~$n$, assuming the $p$-adic period conjecture. Thus, Kim's Conjecture~\ref{conj:kim-full} will generally not be satisfied for $X(\cO_K \otimes \bZ_p)_{S,\PL,4}$. The best we can hope for is that the inclusion
\begin{equation}
\label{eq:inclusion-S-integral-and-roots-of-unity}
    X(\mathcal{O}_{K,S})\cup \left\{(\zeta,\zeta^{-1}):1\ne \zeta\in\mu(\mathbb{Z}_p)\right\} \subseteq X(\cO_K \otimes \bZ_p)_{S,\PL,4}
\end{equation}
is an equality. Indeed, we have verified:

\begin{thm}
\label{S=1-stable-depth4}
    The inclusion~\eqref{eq:inclusion-S-integral-and-roots-of-unity} is an equality for all primes $p < 2{,}000$ which split in~$K$, in each of the following cases:
    \begin{enumerate}[label=(\alph*)]
        \item $K = \bQ(i)$ and $S = \{(1-i)\}$;
        \item $K = \bQ(\zeta_3)$ and $ S = \{(1-\zeta_3)\}$.
    \end{enumerate}
\end{thm}

\begin{rem}
\label{S=1-stable-off-diag-depth3}
    In the cases covered by \Cref{S=1-stable-depth4}, the inclusion~\eqref{eq:inclusion-S-integral-and-roots-of-unity} is already an equality in depth~3 away from the diagonal. That is, for $p < 2{,}000$, the intersection of the locus $X(\cO_K \otimes \bZ_p)_{S,\PL,3}$ with the complement of the diagonal $\Delta = \{z_1 = z_2\}$ contains precisely the pairs $(\zeta,\zeta^{-1})$ for $\zeta \in \bZ_p$ a root of unity with $\zeta \neq \pm 1$, and four additional pairs coming from $S$-integral points:
    \begin{enumerate}[label=(\alph*)]
        \item $(1+i,1-i)$, $(1-i,1+i)$, $\bigl(\frac{1+i}2, \frac{1-i}2\bigr)$, $\bigl(\frac{1-i}2, \frac{1+i}2\bigr)$; resp.
        \item $(1-\zeta_3,1-\zeta_3^{-1})$, $(1-\zeta_3^{-1}, 1-\zeta_3)$, $\bigl(\frac1{1-\zeta_3}, \frac1{1-\zeta_3^{-1}}\bigr)$, $\bigl(\frac1{1-\zeta_3^{-1}}, \frac1{1-\zeta_3}\bigr)$.
    \end{enumerate}
    Thus, in light of \Cref{lem:logLi1eq} and \Cref{S=1-stable-depth2-bound}, the size of $X(\cO_K \otimes \bZ_p)_{S,\PL,n} \smallsetminus \Delta$ is equal to $(p-2)(p-3)$ for $n = 1$, drops to $3p-9$ or $3p-13$ for $n = 2$, and finally to $p+1$ for $n \geq 3$.
\end{rem}

\begin{cor}
\label{Qzeta3-kim-conjecture}
    Let $K = \bQ(\zeta_3)$ and $S = \{(1-\zeta_3)\}$. Then Kim's Conjecture~\ref{conj:kim-full} holds for the full depth-4 quotient of the fundamental group for all primes $p < 2{,}000$ which split in~$K$.
\end{cor}

\begin{proof}
    When the inclusion~\eqref{eq:inclusion-S-integral-and-roots-of-unity} is an equality then $X(\cO_K \otimes \bZ_p)^{S_3}_{S,\PL,4} = X(\cO_{K,S})$; in other words, $S_3$-symmetrisation (\Cref{def:S3-symmetrisation}) eliminates all the roots of unity solutions $(\zeta,\zeta^{-1})$, except those where $\zeta$ has order $3$ or $6$, in which case $(\zeta,\zeta^{-1})$ comes from an $S$-integral point. Now the claim follows from $X(\cO_K \otimes \bZ_p)_{S,4} \subseteq X(\cO_K \otimes \bZ_p)_{S,\PL,4}^{S_3}$ and \Cref{S=1-stable-depth4}.
\end{proof}

\begin{rem}
\label{rem:enlarging-field}
    When $K$ is imaginary quadratic, $S = \{\fl\}$ is Galois-stable, and $l$ is the rational prime below~$\fl$, we have an inclusion $X(\bZ[1/l]) \subseteq X(\cO_{K,\{\fl\}})$, which upgrades to an inclusion of Chabauty--Kim loci $X(\bZ_p)_{\{l\},\PL,n} \subseteq X(\cO_K \otimes \bZ_p)_{\{\fl\},\PL,n}$ for all $n$ by \Cref{field-extension-functoriality}. Therefore, whenever~$p$ splits in~$K$, every point $z$ in $X(\bZ_p)_{\{l\},\PL,n}$ gives rise to a diagonal solution $(z,z)$ to the equations from \Cref{S=1-stable-equations}. Thus, writing $\tau_l \coloneqq \tau_{\fl}$, the equations
    \[  L_2(z) = 0, \qquad c_{\tau_l} c_{\sigma_3} L_4(z) = c_{\tau_l \sigma_3} \log(z) L_3(z) \]
    hold on $X(\bZ_p)_{\{l\},\PL,4}$. These are equivalent to the equations derived in \cite[Proposition~5.4]{CDC:polylog1}. In the case $l = 3$ the authors of loc.\ cit.\ managed to determine the coefficient $c_{\tau_3 \sigma_3}$ (which is $f_{\sigma\tau} -\frac12 f_{\sigma} f_{\tau}$ in their notation, for the choices of $\tau_3$ and $\sigma_3$ satisfying $c_{\tau_3} = \log(3)$ and $c_{\sigma_3} = \zeta(3)$) by enlarging $\{3\}$ to $\{2,3\}$ and exploiting the known $\{2,3\}$-integral points $3$ and~$9$. We now have a new way of determining $c_{\tau_3\sigma_3}$ by enlarging the field instead of the set of primes: from the fact that $1-\zeta_3$ is a $\{(1-\zeta_3)\}$-integral point of~$X$, we get
    \[ c_{\tau_3\sigma_3} = 2\zeta(3) \frac{L_4(1-\zeta_3) + L_4(1-\zeta_3^{-1})}{L_3(1-\zeta_3) + L_3(1-\zeta_3^{-1})},\]
    using~\eqref{eq:S=1-depth-4} and $\log(1-\zeta_3) = \frac12 \log(3) = \frac12 c_{\tau_3}$ and $c_{\sigma_3} = \zeta(3)$. This approach generalises to arbitrarily high depth: every element $z \in X(\bZ_p)_{\{3\},\PL,\infty}$ satisfies
    \[ L_2(z) = 0, \qquad c_{\tau_3} c_{\sigma_{n-1}} L_n(z) = c_{\tau_3\sigma_{n-1}} \log(z) L_{n-1}(z) \; \text{ for $n \geq 4$ even} \]
    with $c_{\tau_3} = \log(3)$, $c_{\sigma_{n-1}} = \zeta(n-1)$, and
    \[
        c_{\tau_3\sigma_{n-1}} = 2\zeta(n-1)\frac{L_n(1-\zeta_3) + L_n(1-\zeta_3^{-1})}{L_{n-1}(1-\zeta_3) + L_{n-1}(1-\zeta_3^{-1})}.
    \]
    This assumes that $p$ splits in~$\bQ(\zeta_3)$, so that $\zeta_3 \in \bQ_p$; however, seeing that the expression for $c_{\tau_3\sigma_{n-1}}$ is contained in~$\bQ_p$ by Galois invariance even when $\zeta_3 \not\in \bQ_p$, the same formula presumably works for any $p \neq 3$.
\end{rem}

\begin{rem}
\label{Qzeta-neg2-coeffs}
    In joint work with Kim \cite{KLL:cyclotomic}, we determine completely explicit equations for the depth~4 polylogarithmic Chabauty--Kim locus in the case $K' = \bQ(\zeta_8)$, $S' = \{(1-\zeta_8)\}$. Using the inclusion $\bQ(\sqrt{-2}) \subseteq \bQ(\zeta_8)$, similar to \Cref{rem:enlarging-field}, this makes it possible to determine the coefficients $c_3$ and $c_4$ in the case $K = \bQ(\sqrt{-2})$, $S = \{(\sqrt{-2})\}$, and compute the corresponding loci $X(\cO_K \otimes \bZ_p)_{S,\PL,n}$ for $n \leq 4$ whenever $p$ splits in~$\bQ(\zeta_8)$. However, we have not carried this out for this paper.
\end{rem}

\subsection{Galois-unstable case}
\label{sec:Galois-unstable}

We now treat the case where $S = \{\fl\}$ is not Galois-stable. In this case, $\fl$ is one of the two primes lying over a rational prime~$l$ which splits in~$K$.
Let $\pi_{\fl}$ be a generator of the 1-dimensional space $\cO_{K,\{\fl\}}^{\times} \otimes \bQ$, which dually defines a generator $\tau_{\fl}$ of $\Lie(U_S^{\MT})$ in degree~$-1$. Choose the generators $\sigma_n$ in degree~$-n$ arbitrarily and write the element $\eps_{\univ} \coloneqq \log(\eta_{\univ}) \in \Lie(U_S^{\MT})(A(\cO_{K,S}))$ as
\[ \eps_{\univ} = c_{\tau_{\fl}}^{\fu} \tau_{\fl} + c_{\sigma_2}^{\fu} \sigma_2 + c_{\sigma_3}^{\fu} \sigma_3 + c_{\tau_{\fl}\sigma_2}^{\fu} [\tau_{\fl},\sigma_2] + c_{\sigma_4}^{\fu} \sigma_4 + c_{\tau_{\fl}\sigma_3}^{\fu} [\tau_{\fl},\sigma_3] + c_{\tau_{\fl} \tau_{\fl} \sigma_2}^{\fu} [\tau_{\fl}, [\tau_{\fl}, \sigma_2]] + \ldots \]
where the coefficients $c_w^{\fu}$ are motivic periods in $A(\cO_{K,S})$. Specialising at $\fp_1$ with $c_w = \per_{\fp_1}(c_w^{\fu})$ gives the expansion of $\eps_1\coloneq \eps_{\fp_1}$ as in \eqref{eq:eps-expansion}, which we write again below:
\begin{equation*}
    \eps_1=\eps_{\fp_1} = c_{\tau_{\fl}} \tau_{\fl} + c_{\sigma_2}\sigma_2 + c_{\sigma_3} \sigma_3 + c_{\tau_{\fl}\sigma_2} [\tau_{\fl},\sigma_2] + c_{\sigma_4} \sigma_4 + c_{\tau_{\fl}\sigma_3} [\tau_{\fl},\sigma_3] + c_{\tau_{\fl} \tau_{\fl} \sigma_2} [\tau_{\fl}, [\tau_{\fl}, \sigma_2]] + \ldots
\end{equation*}

Similarly, we write the expansion of $\eps_2\coloneq \eps_{\fp_2}=\eps_{\sigma^*\fp_1}$ as the following
\begin{equation}
    \label{eq:Galois-unstable-eps2}
    \eps_2=\eps_{\sigma^*\fp_1} = c_{\tau_{\fl}}' \tau_{\fl} + c_{\sigma_2}'\sigma_2 + c_{\sigma_3}' \sigma_3 + c_{\tau_{\fl}\sigma_2}' [\tau_{\fl},\sigma_2] + c_{\sigma_4}' \sigma_4 + c_{\tau_{\fl}\sigma_3}' [\tau_{\fl},\sigma_3] + c_{\tau_{\fl} \tau_{\fl} \sigma_2}' [\tau_{\fl}, [\tau_{\fl}, \sigma_2]] +\ldots
\end{equation}
with $c_w' = \per_{\sigma^*\fp_1}(c_w^{\fu}) = \per_{\fp_1}(\sigma_* c_w^{\fu})$ by \Cref{eta-sigma-p-from-eta-p}. In particular,
\[ c_{\tau_{\fl}}' = \per_{\fp_1}(\sigma_* \log^{\fu}(\pi_{\fl})) = \per_{\fp_1}(\log^{\fu}(\sigma \pi_{\fl})) = \log(\sigma \pi_{\fl}) \]
and
\[ c_{\sigma_n}' = \per_{\fp_1}(\sigma_*c_{\sigma_n}^{\fu}) = \per_{\fp_1}((-1)^{n+1} c_{\sigma_n}^{\fu}) = (-1)^{n+1} c_{\sigma_n},\]
where we used that $c_{\sigma_n}^{\fu}$ belongs to the 1-dimensional extension space $E_n \cong \Ext^1_{\MT(K,\bQ)}(\bQ(0),\bQ(n))$ where $\sigma_*$ acts as $(-1)^{n+1}$. (Note that $c_{\sigma_n}^{\fu} = b_{\sigma_n}^{\fu} = a_{\sigma_n}^{\fu} = f_{\sigma_n} \in E_n$ by \Cref{lie-coeffs} and \Cref{lie-coefficients-from-group-coefficients}.)

\begin{thm}
\label{S=1-unstable-equations}
    Assume that $S = \{\fl\}$ is not Galois-stable. Then, for $n=1,2,3,4$, the depth-$n$ polylogarithmic Chabauty--Kim locus $X(\cO_{K,S} \otimes \bZ_p)_{S,\PL,n}$ is contained in the set of pairs $(z_1,z_2) \in X(\bZ_p) \times X(\bZ_p)$ satisfying the incremental list of equations:
    \begin{itemize}
        \item in depth~1:
        \begin{align}
            \label{eq:S=1-unstable-log}
            \log(\sigma\pi_{\fl}) \log(z_1) &= \log(\pi_{\fl}) \log(z_2),\\
            \log(\sigma\pi_{\fl}) \log(1-z_1) &= \log(\pi_{\fl}) \log(1-z_2);
        \end{align}
        \item in depth~2:
        \begin{equation}
            \label{eq:S=1-unstable-depth-2}
            L_2(z_1) + L_2(z_2) = 0;
        \end{equation}
        \item in depth~3:
        \begin{equation}
            c_{\tau_{\fl}} c_{\sigma_2}(L_3(z_1) - L_3(z_2)) = (c_{\tau_{\fl}\sigma_2}+c_{\tau_{\fl}\sigma_2}')\log(z_1) L_2(z_1);
        \end{equation}
        \item in depth~4:
        \begin{equation}
        \begin{aligned}
            & c_{\tau_{\fl}}c_{\sigma_3}(c_{\tau_{\fl}\sigma_2}+c_{\tau_{\fl}\sigma_2}')(L_4(z_1)+L_4(z_2))\\
            =&(c_{\sigma_3}(c_{\tau_{\fl}\tau_{\fl}\sigma_2}-c_{\tau_{\fl}\tau_{\fl}\sigma_2}')(L_3(z_1)-L_3(z_2))+(c_{\tau_{\fl}\sigma_3}+c_{\tau_{\fl}\sigma_3}')(c_{\tau_{\fl}\sigma_2}'L_3(z_1)+c_{\tau_{\fl}\sigma_2}L_3(z_2)))\log(z_1)
        \end{aligned}
        \end{equation}
    \end{itemize}
where $\log(\pi_{\fl}), \log(\sigma \pi_{\fl}) \in \bQ_p$ are defined via the $\fp_1$-adic embedding $K^\times \hookrightarrow K_{\fp_1}^{\times} \cong \bQ_p^{\times}$, and $\sigma\in\Gal(K/\mathbb{Q})$ is complex conjugation.
The containment is an equality if the $p$-adic period conjecture holds for the primes dividing~$p$.
\end{thm}

\begin{proof}
We need to determine the image of $\ev_{\eps_1}\times \ev_{\eps_2}$ and use the same coordinates $x_{\fl},y_{\fl},z_2,z_3,z_4$ of the depth $4$ global Selmer scheme as in the proof of \Cref{S=1-stable-equations}. Using the expression of $\eps_1,\eps_2$ above and by \Cref{cocycle-evaluation-map}, the cocycle evaluation map $\ev_{\eps_1}$ is given by
\begin{equation}
\begin{split}
    &\ev_{\eps_1}\colon (x_{\fl},y_{\fl},z_2,z_3,z_4) \mapsto \\
    &\qquad (c_{\tau_{\fl}}x_{\fl}, \;\; c_{\tau_{\fl}}y_{\fl}, \;\; c_{\sigma_2}z_2, \;\; c_{\sigma_3}z_3+c_{\tau_{\fl} \sigma_2}x_{\fl}z_2, \;\; c_{\sigma_4}z_4+c_{\tau_{\fl}\sigma_3}x_{\fl}z_3+c_{\tau_{\fl} \tau_{\fl} \sigma_2} x_{\fl}^2 z_2),
\end{split}
\end{equation}
and the cocycle evaluation map $\ev_{\eps_2}$ for $\eps_2$ is given by
\begin{equation}
\begin{split}
    &\ev_{\eps_2}\colon (x_{\fl},y_{\fl},z_2,z_3,z_4) \mapsto\\
    &\qquad (c_{\tau_{\fl}}'x_{\fl}, \; c_{\tau_{\fl}}'y_{\fl}, \; -c_{\sigma_2}z_2, \; c_{\sigma_3}z_3-c_{\tau_{\fl} \sigma_2}'x_{\fl}z_2, \; -c_{\sigma_4}z_4+c_{\tau_{\fl}\sigma_3}'x_{\fl}z_3-c_{\tau_{\fl} \tau_{\fl} \sigma_2}' x_{\fl}^2 z_2).
\end{split}
\end{equation}

Comparing the expressions of $\ev_{\eps_1},\ev_{\eps_2}$ and pulling back these equations to $X(\bZ_p) \times X(\bZ_p)$ results in the claimed equations.

\end{proof}

\begin{thm}
\label{S=1-finiteness}
    In the setting of \Cref{S=1-unstable-equations}, the depth-$1$ polylogarithmic Chabauty--Kim locus $X(\cO_K \otimes \bZ_p)_{S,\PL,1}$ is finite. \qed
\end{thm}

\Cref{S=1-finiteness} immediately follows from the following:

\begin{lemma}
    \label{depth-1-finiteness}
    Let $a, b \in \bQ_p \smallsetminus \{0\}$ with $a \neq b$. Then there are only finitely many pairs $(z_1,z_2) \in X(\bZ_p) \times X(\bZ_p)$ satisfying the equations
    \begin{align*}
        a \log(z_1) &= b \log(z_2), \\
        a \log(1-z_1) &= b \log(1-z_2).
    \end{align*}
\end{lemma}

We prove \Cref{depth-1-finiteness}  using the Ax--Schanuel theorem, following a suggestion by Netan Dogra. Since the Ax--Schanuel theorem is usually stated over the complex numbers, we need a lemma to pass from $p$-adic to complex logarithms. The following is the torus variant of \cite[Lemma~2.1]{dogra_2023_unlikely}, proved in the same way by comparing the underlying formal groups:

\begin{lemma}[Formal comparison of logarithm and exponential]
    \label{lem:formal-log-exp}
    Let $T=\mathbb{G}_m^n$ over $\bQ_p$, and let $\Delta_{\log}\subset T^{\an}\times\Lie(T)^{\an}$ be the graph of a branch of the $p$-adic logarithm. After choosing an embedding $\bQ_p\hookrightarrow\bC$, the formal completion of $\Delta_{\log}$ at any point of $T(\bQ_p)$ is, after translation, the base change to $\bC$ of the formal graph of the complex exponential. \qed
\end{lemma}

We need the following geometric formulation of the Ax--Schanuel theorem for tori:

\begin{thm}[{\cite[Theorem~1.3]{Tsimerman_2015}}]
\label{ax-schanuel}
    Let $V \subseteq \bC^n \times (\bC^\times)^n$ be an irreducible subvariety. Let $\Delta_{\exp} \subset \bC^n \times (\bC^\times)^n$ be the graph of the exponential and let $U$ be an irreducible component of the complex analytic space $V \cap \Delta_{\exp}$. If the projection of $U$ to $(\bC^\times)^n$ is not contained in a coset of a proper subtorus then
    \[ \dim_{\bC} V \geq \dim_{\bC} U + n. \]
\end{thm}

\begin{proof}[Proof of \Cref{depth-1-finiteness}]
    Let $V\subseteq X_{\bQ_p}\times X_{\bQ_p}$ be the Zariski closure of the set of pairs $(z_1,z_2)\in X(\bZ_p)\times X(\bZ_p)$ satisfying $a\log(z_1)=b\log(z_2)$ and $a\log(1-z_1)=b\log(1-z_2)$. We first show that $V$ is a union of finitely many points and finitely many curves.
    Let $Z$ be the algebraic subvariety of $\mathbb{A}_{\mathbb{Q}_p}^4\times\mathbb{G}_{m,\mathbb{Q}_p}^4$ defined by
    \[
    z_i+z_i'=1\quad (i=1,2),
    \qquad
    a u_1=b u_2,
    \qquad
    a u_1'=b u_2'.
    \]
    Here $(u_1,u_2,u_1',u_2')$ are the additive coordinates and
    $(z_1,z_2,z_1',z_2')$ the multiplicative coordinates. The intersection $Z \cap \Delta_{\log}$ with the graph of the $p$-adic logarithm is given by
    \[
    u_1=\log(z_1),\quad
    u_2=\log(z_2),\quad
    u_1'=\log(1-z_1),\quad
    u_2'=\log(1-z_2).
    \]
    The intersection $Z\cap \Delta_{\log}$ is a union of finitely many irreducible analytic components. Let $W$ be such a component. Take a point $P$ in $W$ and let $\hat{W}$ be the formal completion of $W$ at $P$. By \Cref{lem:formal-log-exp}, after choosing an embedding $\mathbb{Q}_p\hookrightarrow \mathbb{C}$, $\hat{W}_{\mathbb{C}}$ is an irreducible analytic component of the formal completion of $Z_{\mathbb{C}}\cap \Delta_{\exp}$ at $P$, where $\Delta_{\exp}\subseteq \mathbb{C}^4\times (\mathbb{C}^\times)^4$ denotes the graph of the complex exponential. If $W$ is not finite, then $\hat{W}_{\mathbb{C}}$ is the formal completion at $P$ of some irreducible analytic component $\tilde{W}$ of $Z_{\mathbb{C}}\cap \Delta_{\exp}$ satisfying $\dim_{\mathbb{C}} Z_{\mathbb{C}}=4<\dim_{\mathbb{C}} \tilde{W}+4$. 
    By \Cref{ax-schanuel}, $\mathrm{pr}_{(\mathbb{C}^\times)^4}(\tilde{W})$ is contained in a translate of proper subtorus of $\mathbb{G}_{m,\mathbb{C}}^4$ (where $\mathrm{pr}_{(\mathbb{C}^\times)^4}$ denotes the projection to the $(\mathbb{C}^\times)^4$ components), hence the same holds for $\mathrm{pr}_{(\mathbb{C}^\times)^4}(\hat{W}_{\mathbb{C}})$, and hence for $\mathrm{pr}_{(\mathbb{Q}_p^\times)^4}(\hat{W})$, and hence for $\mathrm{pr}_{(\mathbb{Q}_p^\times)^4}(W)$. Then there exist $(a_1,a_2,a_1',a_2')\in\mathbb{Z}^4\backslash \left\{0\right\}$ and $c\in\mathbb{Q}_p^\times$ such that $z_1^{a_1}z_2^{a_2}z_1'^{a_1'}z_2'^{a_2'}=c$ for any $(z_1,z_2,z_1',z_2')\in \mathrm{pr}_{(\mathbb{Q}_p^\times)^4}(W)$, or equivalently $z_1^{a_1}z_2^{a_2}(1-z_1)^{a_1'}(1-z_2)^{a_2'}=c$ for any $(z_1,z_2)\in \mathrm{pr}_{(z_1,z_2)}(W)$, which is a curve in $(\mathbb{Q}_p^\times)^2$. Therefore, we proved that for each irreducible analytic component $W$ of $Z\cap \Delta_{\log}$, either $W$ is finite or $\mathrm{pr}_{(z_1,z_2)}(W)$ is contained in a curve. Since there are only finitely many such components, we conclude that $V\subseteq \mathrm{pr}_{(z_1,z_2)}(Z\cap \Delta_{\log})^{\mathrm{Zar}}$ is included in a union of finitely many points and finitely many curves.

    Next we show that $V$ cannot contain any curve, and hence $V$ is finite. Suppose for contradiction that $C$ is such a curve. By the first equation, the Coleman function $z \mapsto a \log(z_1) - b \log(z_2)$ vanishes on infinitely many points of~$C(\bQ_p)$, hence it is constant zero. Its derivative must therefore be zero, which implies that the differential $a \frac{\rd z_1}{z_1} - b \frac{\rd z_2}{z_2}$ on $X_{\bQ_p} \times X_{\bQ_p}$ pulls back to zero on~$C$. Similarly, by the second equation, the differential $a \frac{\rd z_1}{1-z_1} - b \frac{\rd z_2}{1-z_2}$ also pulls back to zero on~$C$. Thus, $C$ is contained in the vanishing locus of the 2-form
    \[ \left(a \frac{\rd z_1}{z_1} - b \frac{\rd z_2}{z_2}\right) \wedge \left(a \frac{\rd z_1}{1-z_1} - b \frac{\rd z_2}{1-z_2}\right) = ab(z_1 - z_2) \frac{\rd z_1}{z_1(1-z_1)} \wedge \frac{\rd z_2}{z_2(1-z_2)}, \]
    namely the diagonal $\{z_1 = z_2\} \subseteq X_{\bQ_p} \times X_{\bQ_p}$. But $a \neq b$ by assumption, so the first equation holds for only finitely many points on the diagonal: the pairs $(\zeta,\zeta)$ with $\zeta \in X(\bZ_p)$ a root of unity. This is a contradiction, so $V$ does not contain a curve.
\end{proof}

\subsection{Refined Chabauty--Kim loci}
\label{sec:refined}

We now study the refined Chabauty--Kim locus $X(\cO_K \otimes \bZ_p)_{S,\PL,\infty}^{\min}$ in the case that $K$ is imaginary quadratic and $S=\left\{\mathfrak{l}\right\}$ contains a single prime of~$K$. In this case, the non-degeneracy condition \ref{cond:F2} requires that either $2$ is inert in~$K$, or $2$ is ramified and $\fl$ is the unique prime of~$K$ dividing~$2$. If this is not satisfied then $X(\cO_{K,\{\fl\}})$ as well as all refined Chabauty--Kim loci are automatically empty.

By~\eqref{eq:decompositionSelmin}, the refined Chabauty--Kim locus is a union of three subsets, corresponding to the three refinement conditions $\Sigma=(0),(1),(\infty)$:
\[ X(\cO_K \otimes \bZ_p)_{\{\fl\},\PL,n}^{\min} = X(\cO_K \otimes \bZ_p)_{\{\fl\},\PL,n}^{(0)} \cup X(\cO_K \otimes \bZ_p)_{\{\fl\},\PL,n}^{(1)} \cup X(\cO_K \otimes \bZ_p)_{\{\fl\},\PL,n}^{(\infty)}. \]
for $1 \leq n \leq \infty$. We first treat the case $\Sigma=(1)$, where the equations become particularly simple and can be determined in infinite depth.

\begin{thm}
\label{thm:refinedkim-functionsforS=1}
	Let $K$ be an imaginary quadratic field and $S=\left\{\mathfrak{l}\right\}$. Assume \Cref{cond:F2}. Let $p$ be a rational prime which splits completely in~$K$ and is not divisible by~$\fl$. Then the refined polylogarithmic Chabauty--Kim locus $X(\cO_K \otimes \bZ_p)_{S,\PL,\infty}^{(1)}$ is contained in the set of pairs $(z_1,z_2) \in X(\bZ_p) \times X(\bZ_p)$ satisfying the equations
		\begin{align}
			\log(z_1) =\log(z_2) &= 0, \label{eq:relog-equationsforS=1}\\
			\log(\sigma\pi_{\fl})\Li_1(z_1) - \log(\pi_{\fl})\Li_1(z_2)&=0 , \label{eq:reLi1-equationsforS=1}\\
			\Li_n(z_1) + (-1)^n \Li_n(z_2) &= 0 \qquad \text{for $n \geq 2$},\label{eq:reLin-equationsforS=1}
		\end{align}
    where $\log(\pi_{\fl}), \log(\sigma \pi_{\fl}) \in \bQ_p$ are defined via the $\fp_1$-adic embedding $K^\times \hookrightarrow K_{\fp_1}^{\times} \cong \bQ_p^{\times}$, and $\sigma\in\Gal(K/\mathbb{Q})$ is complex conjugation.
    The containment is an equality if the $p$-adic period conjecture holds for the primes dividing~$p$.
\end{thm}

\begin{proof}
    Let $\tau_{\mathfrak{l}}\in (\mathcal{O}_{K,S}^\times\otimes\mathbb{Q}_p)^\vee$ be the dual of $\alpha$. 
Write the element $\eps_{\univ} \coloneqq \log(\eta_{\univ}) \in \Lie(U_S)(A(\cO_{K,S}))$ as
\[ \eps_{\univ} = c_{\tau_{\fl}}^{\fu} \tau_{\fl} + c_{\sigma_2}^{\fu} \sigma_2 + c_{\sigma_3}^{\fu} \sigma_3 + c_{\sigma_4}^{\fu} \sigma_4 + \ldots \bmod \text{commutators}. \]

Specialising at $\fp_1$ with $c_w = \per_{\fp_1}(c_w^{\fu})$ gives the expansion of $\eps_1\coloneq \eps_{\fp_1}$
\[ \varepsilon_1 = c_{\tau_{\fl}}\tau_{\mathfrak{l}} + c_{\sigma_2} \sigma_2 + c_{\sigma_3} \sigma_3 + c_{\sigma_4} \sigma_4 + \ldots \quad \bmod \text{commutators} \]
where $c_{\tau_{\fl}}=\log(\pi_{\fl})$ by \Cref{tau-coeffs}.

The refined condition $(1)$, which is $v_{\mathfrak{l}}(z)=0$, requires $x_{\mathfrak{l}}(\xi)=0$ for $\xi\in \mathrm{Sel}_{S,\PL}^{(1)}$. As in the proof of Theorem \ref{thm:roots-of-unity}, $c\in \mathrm{Sel}_{S,\PL}^{(1)}$ maps commutators to zero since $x_{\mathfrak{l}}(\xi)=0$. So the localisation map at $\mathfrak{p}_1$ is given by
\begin{equation}
\label{eq:ev1}
    \ev_{\varepsilon_1}(\xi) = \left(0, \log(\pi_{\fl}) y_{\mathfrak{l}}(\xi), c_{\sigma_2} z_2(\xi), c_{\sigma_3}z_3(\xi), c_{\sigma_4} z_4(\xi),\ldots\right).
\end{equation}

Similarly, we write the expansion of $\eps_2\coloneq \eps_{\fp_2}=\eps_{\sigma^*\fp_1}$ as the following
\begin{equation}
    \label{eq:eps2-refined}
    \eps_2=\eps_{\sigma^*\fp_1} = c_{\tau_{\fl}}' \tau_{\fl} + c_{\sigma_2}'\sigma_2 + c_{\sigma_3}' \sigma_3 +  c_{\sigma_4}' \sigma_4  +\ldots \bmod \text{commutators}.
\end{equation}
with $c_w' = \per_{\sigma^*\fp_1}(c_w^{\fu}) = \per_{\fp_1}(\sigma_* c_w^{\fu})$ by \Cref{eta-sigma-p-from-eta-p}. In particular,
\[ c_{\tau_{\fl}}' = \log(\sigma \pi_{\fl}) ,\quad c_{\sigma_n}' =  (-1)^{n+1} c_{\sigma_n}\]
as explained in \S\ref{sec:Galois-unstable}. Therefore, the evaluation map at $\fp_2$ is given by

\begin{equation}
    \label{eq:ev2nonstable}
\ev_{\varepsilon_2}(\xi) = \left(0, \log(\sigma \pi_{\fl})y_{\mathfrak{l}}(\xi),-c_{\sigma_2} z_2(\xi), c_{\sigma_3}z_3(\xi), -c_{\sigma_4} z_4(\xi),\ldots\right).
\end{equation}

Comparing \eqref{eq:ev1} with \eqref{eq:ev2nonstable}, the desired equations are exactly \eqref{eq:relog-equationsforS=1}, \eqref{eq:reLi1-equationsforS=1}, \eqref{eq:reLin-equationsforS=1}.
\end{proof}

\begin{rem}
\label{rem:refined-galois-stable}
    If $S = \{\fl\}$ is Galois-stable in \Cref{thm:refinedkim-functionsforS=1} then $\log(\sigma \pi_{\fl}) = \log(\pi_{\fl})$ since $\sigma \pi_{\fl} = \pi_{\fl}$ holds in $\cO_{K,S}^{\times} \otimes \bQ$, so the coefficients in the equation~\eqref{eq:reLi1-equationsforS=1} can be omitted, which recovers the equation $\log(1-z_1) = \log(1-z_2)$ for the unrefined locus $X(\cO_K \otimes \bZ_p)_{S,\PL,1}$ found in \Cref{S=1-stable-equations}. Thus, in the Galois-stable case, the equations for $X(\cO_K \otimes \bZ_p)_{S,\PL,\infty}^{(1)}$ do not depend on~$K$ and~$S$.
\end{rem}

We can completely solve the equations in \Cref{thm:refinedkim-functionsforS=1} for arbitrary~$p$:

\begin{thm}	
\label{S=1-refined1-equations}
	Let $K$ be an imaginary quadratic field and suppose $S=\left\{\mathfrak{l}\right\}$ contains one prime of~$K$. Let $p$ be a rational prime not divisible by~$\fl$ which splits in~$K$. Assume \Cref{cond:F2}. Then the refined polylogarithmic Chabauty--Kim locus $X(\cO_K \otimes \bZ_p)_{S,\PL,\infty}^{(1)}$ in infinite depth is contained in
\begin{itemize}
    \item the set of all pairs $(\zeta,\zeta^{-1})$ where $\zeta\ne 1$ is a root of unity in $\mathbb{Z}_p$, if $S$ is stable under the $\mathrm{Gal}(K/\mathbb{Q})$-action;
    \item $\left\{(\zeta_6,\zeta_6^{-1}),(\zeta_6^{-1},\zeta_6)\right\}$, if $S$ is not stable under the $\mathrm{Gal}(K/\mathbb{Q})$-action and there exists a primitive $6$-th root of unity $\zeta_6$ in $\mathbb{Z}_p$;
    \item empty otherwise.
\end{itemize}
The containment is an equality if the $p$-adic period conjecture holds for the primes dividing~$p$.
\end{thm}
\begin{proof}
    We determine the solutions of the equations \eqref{eq:relog-equationsforS=1}, \eqref{eq:reLi1-equationsforS=1}, \eqref{eq:reLin-equationsforS=1} for $X(\cO_K \otimes \bZ_p)_{S,\PL,\infty}^{(1)}$ from Theorem \ref{thm:refinedkim-functionsforS=1}.
    First note that $(z_1,z_2)$ satisfies \eqref{eq:relog-equationsforS=1} if and only if $z_1,z_2\ne 1$ are roots of unity in $\mathbb{Z}_p$.
    We show that $(z_1,z_2)$ in addition satisfies \eqref{eq:reLin-equationsforS=1} if and only if $z_2 = z_1^{-1}$, i.e.\ the pair is of the form $(\zeta,\zeta^{-1})$ with $\zeta\ne 1$ a root of unity in $\mathbb{Z}_p$. All such pairs satisfy \eqref{eq:reLin-equationsforS=1} by Coleman--Sinnott (Lemma \ref{coleman-sinnott}). Conversely, assume that $z_1, z_2$ are nontrivial roots of unity in~$\bZ_p$ satisfying \eqref{eq:reLin-equationsforS=1}. If $p = 2$ or $p=3$ then $\zeta = -1$ is the only nontrivial root of unity in $\mathbb{Z}_p$, so $(z_1,z_2)=(-1,-1)$ necessarily, which is of the claimed form.
    Assume $p\geq 5$. We use an argument involving finite polylogarithms \cite[Proposition~2.1]{besser:finite-polylogs} similar to the proof of \cite[Proposition~5.15]{BKL:chabauty-kim-sc}. We have a $\mathbb{Z}_p$-valued modified polylogarithm function
    \begin{equation}
    \label{eq:modifiedLin}
        \mathrm{Li}_n^{(p)}(z) \coloneqq \Li_n(z) - \frac1{p^n}\Li_n(z^p)
    \end{equation}
    whose mod-$p$ reduction is given by
    $$\mathrm{Li}_n^{(p)}(z)\equiv \frac{1}{1-z^p}\mathrm{li}_n(\overline{z})\bmod{p}, \quad \mathrm{li}_n(x)\coloneq \sum_{k=1}^{p-1}\frac{x^k}{k^n} \text{ for }x\in\mathbb{F}_p.$$
    Since $z_1$, $z_2$ are $(p-1)$st roots of unity satisfying~\eqref{eq:reLin-equationsforS=1} by assumption, we have 
    $$ \mathrm{Li}_n^{(p)}(z_1)+(-1)^n\mathrm{Li}_n^{(p)}(z_2)=0 \quad \text{for $n \geq 2$}, $$
    which implies that the mod-$p$ reductions $\overline{z_i} \in \mathbb{F}_p$ satisfy
    \begin{equation}
    \label{eq:lin}
        \frac{1}{1-\overline{z_1}}\mathrm{li}_n(\overline{z_1})+(-1)^n\frac{1}{1-\overline{z_2}}\mathrm{li}_n(\overline{z_2})=0 \quad \text{for $n \geq 2$}.
    \end{equation}
    We will use this with $n=p-2$ and $n=p-3$, where $p\geq 5$ guarantees $n\geq 2$. For any $x\in \mathbb{F}_p\smallsetminus\left\{0,1\right\}$, we have
    \begin{align}
    \label{eq:lip-2}
        \mathrm{li}_{p-2}(x)&=\sum_{k=1}^{p-1}\frac{x^k}{k^{p-2}}=\sum_{k=1}^{p-1}kx^k=x\frac{1-x^p}{(1-x)^2}= \frac{x(1-x)}{(1-x)^2}=\frac{x}{1-x},\\
    \label{eq:lip-3}
        \mathrm{li}_{p-3}(x)&=\sum_{k=1}^{p-1}\frac{x^k}{k^{p-3}}=\sum_{k=1}^{p-1}k^2x^k=x\frac{(1+x)(1-x^p)}{(1-x)^3}= x\frac{(1+x)(1-x)}{(1-x)^3}=\frac{x(1+x)}{(1-x)^2}.
    \end{align}
    Taking $n=p-2$ in \eqref{eq:lin} and using \eqref{eq:lip-2} we have
    $$\frac{\overline{z_1}}{(1-\overline{z_1})^2}-\frac{\overline{z_2}}{(1-\overline{z_2})^2}=0.$$
    This simplifies to $(\overline{z_1}-\overline{z_2})(1 - \overline{z_1}\overline{z_2}) = 0$, hence $\overline{z_1}=\overline{z_2}$ or $\overline{z_1}=\overline{z_2}^{-1}$. Since $z_1,\ z_2$ are roots of unity, $z_1=z_2$ or $z_1=z_2^{-1}$ holds in~$\bZ_p$. In fact we always have $z_1=z_2^{-1}$. Indeed, if $z_1=z_2$, then taking $n=p-3$ in~\eqref{eq:lin} and yields $\li_{p-3}(\overline{z_1}) = 0$, which implies $\overline{z_1} = -1$ by~\eqref{eq:lip-3}, so we have $z_1 = z_2 = -1$, which still satisfies $z_1=z_2^{-1}$. 
    
    In summary, we proved that the set of $(z_1,z_2) \in X(\bZ_p) \times X(\bZ_p)$ satisfying \eqref{eq:relog-equationsforS=1} and \eqref{eq:reLin-equationsforS=1} is precisely the set of all pairs $(\zeta,\zeta^{-1})$ where $\zeta\ne 1$ is a root of unity in~$\bZ_p$. It remains to determine which of these pairs also satisfy~\eqref{eq:reLi1-equationsforS=1}. Note that $\mathrm{Li}_1(\zeta)=\mathrm{Li}_1(\zeta^{-1})$. If $S$ is stable under the $\mathrm{Gal}(K/\mathbb{Q})$-action, then $\log(\pi_{\fl})= \log(\sigma\pi_{\fl})$ and \eqref{eq:reLi1-equationsforS=1} is automatically satisfied for $(\zeta,\zeta^{-1})$. Otherwise $\log(\pi_{\fl})\ne \log(\sigma\pi_{\fl})$, so \eqref{eq:reLi1-equationsforS=1} gives $\mathrm{Li}_1(\zeta)=0$. Then $\zeta$ must be a primitive sixth root of unity by \Cref{lem:logLi1zeta6}.
\end{proof}

\begin{rem}
\label{roots-of-unity-in-refined-locus}
    When $K$ is imaginary quadratic and $S$ is Galois-stable of size~$1$, we already saw in \Cref{thm:roots-of-unity} that the pairs $(\zeta,\zeta^{-1})$ for $\zeta \neq 1$ a root of unity in~$\bZ_p$ are contained in $X(\cO_K \otimes \bZ_p)_{S,\PL,\infty}$, assuming the period conjecture.
    The Galois-stable part of \Cref{S=1-refined1-equations} says that these pairs are even present in the \emph{refined} Chabauty--Kim locus $X(\cO_K \otimes \bZ_p)_{S,\PL,\infty}^{\min}$. In fact this can already be seen by inspecting the proof of \Cref{thm:roots-of-unity}, where an element $\xi \in \Sel \coloneqq \Sel_{S,\PL,\infty}^{\mot}(X)$ with $\loc_p(\xi) = j_p((\zeta,\zeta^{-1}))$ is constructed with the additional property that $x_1(\xi) = 0$, which in the case of $S$ being Galois-stable of size~1 expresses that $\xi$ is contained in the refined Selmer scheme $\Sel^{(1)} \subseteq \Sel$.
\end{rem}

\begin{thm}
\label{selmer-section-conjecture}
    The $S$-Selmer Section Conjecture \cite[Conj.~1.6]{BKL:chabauty-kim-sc} holds for $K$ an imaginary quadratic field, $S$ which is empty or contains precisely one prime which is not stable under the $\Gal(K/\bQ)$-action, and the curve $X=\mathbb{P}^1\smallsetminus\left\{0,1,\infty\right\}$.
\end{thm}

\begin{proof}
     Note that the set $\fP$ of primes $\mathfrak{p} \notin S$ in $K$ whose underlying rational prime~$p$ splits completely in $K$ has Dirichlet density $1$. By \Cref{integral-points-locus} (for $S = \emptyset$) and \Cref{S=1-refined1-equations} (for $S = \{\fl\}$ not Galois-stable), the refined Chabauty--Kim locus $X(\mathcal{O}_{\mathfrak{p}})_{S,\PL,\infty}^{(1)}$ is a subset of $Z(K_{\fp})$ for all $\fp \in \fP$, where $Z\subseteq X$ is the reduced finite subscheme consisting of $\zeta_6,\zeta_6^{-1}$. Passing from the polylogarithmic quotient to the full fundamental group, we in particular get an inclusion $X(\cO_{\fp})_{S,\infty}^{(1)} \subseteq Z(K_{\fp})$. Taking $S_3$-orbits we get the full refined locus $X(\cO_{\fp})_{S,\infty}^{\min}$. But $Z(K_{\fp})$ is $S_3$-stable since $1 - \zeta_6 = \zeta_6^{-1} = 1/\zeta_6$, hence we find
     \[ X(\mathcal{O}_{\mathfrak{p}})_{S,\infty}^{\min}\subseteq Z(K_{\mathfrak{p}}) \]
     for all $\fp \in \fP$. As explained in \cite[§2.3]{BKL:chabauty-kim-sc}, this implies that every element $x = (x_{\fp})_{\fp}$ of the finite descent locus $X(\bA_{K,S})^{\fcov}_{\bullet}$ satisfies $x_{\fp} \in Z(K_{\fp})$ for $\fp \in \fP$, and this implies $x \in Z(K)$ by Theorem~2.9 of loc.\ cit. The sufficiency of the finite descent obstruction implies the $S$-Selmer Section Conjecture by Theorem~F of loc.\ cit.
\end{proof}

\begin{rem}
    \label{rem:selmer-section-conjecture}
    The $S$-Selmer Section Conjecture is known to hold for $X = \bP^1 \smallsetminus \{0,1,\infty\}$ over imaginary quadratic fields for any set~$S$ \cite[Corollary~6]{stix:birationalSC}, so the instances provided by \Cref{selmer-section-conjecture} are not new. However, \Cref{selmer-section-conjecture} shows that the strategy of using Chabauty--Kim computations to obtain results on the Section Conjecture is viable beyond the case $K = \bQ$.
\end{rem}

We now consider the refinement conditions $\Sigma = (0)$ and $\Sigma = (\infty)$ and focus on the case that $S = \{\fl\}$ is Galois-stable. In fact it is enough to consider only $\Sigma = (0)$, since $X(\cO_K \otimes \bZ_p)_{S,\PL,n}^{(0)}$ and $X(\cO_K \otimes \bZ_p)_{S,\PL,n}^{(\infty)}$ are related to one another by the inversion map $z \mapsto 1/z$ thanks to \Cref{inversion-action-on-refined-loci}. 

\begin{thm}
\label{S=1-refined0-equations}
    Let $K$ be an imaginary quadratic field and suppose $S = \{\fl\}$ is Galois-stable. Let $p$ be a rational prime not divisible by~$\fl$ which splits in~$K$. Assume \Cref{cond:F2}. Then, for all $1 \leq n \leq \infty$, the depth-$n$ refined polylogarithmic Chabauty--Kim locus $X(\cO_K \otimes \bZ_p)_{S,\PL,n}^{(0)}$ is the subset of the unrefined locus $X(\cO_K \otimes \bZ_p)_{S,\PL,n}$ consisting of pairs $(z_1,z_2)$ satisfying the additional equations
    \begin{equation}
    \label{eq:S=1-refined0-depth1}
       \log(1-z_1) = 0 \quad \text{ and } \quad \log(1-z_2) = 0. 
    \end{equation}
\end{thm}

\begin{proof}
    Recall in the proof of \Cref{S=1-stable-equations}, the evaluation maps up to depth $1$ are given by
    $$\ev_{\eps_i}:(x_{\fl},y_{\fl})\mapsto (c_{\tau_\fl}x_{\fl},c_{\tau_\fl}y_{\fl})\quad \text{ for }i=1,2.$$

    So $y_{\fl}=L_1^{(1)}/c_{\tau_\fl}=L_1^{(2)}/c_{\tau_\fl}$. The refined condition at $\Sigma=(0)$ is $y_{\fl}=0$, which is $L_1^{(1)}=L_1^{(2)}=0$. Pulling back the equations to $X(\mathbb{Z}_p)\times X(\mathbb{Z}_p)$ gives $\log(1-z_1)=0$ and $\log(1-z_2)=0$.
\end{proof}

Note that the loci $X(\cO_K \otimes \bZ_p)_{S,\PL,n}^{(0)}$ in \Cref{S=1-refined0-equations} do not depend on~$K$ and~$S$ for $n = 1,2$. They also agree with the loci $X(\cO_K \otimes \bZ_p)_{S,n}^{(0)}$ for the full depth-$n$ quotient of the fundamental group since $\Pi_{\PL,n} = \Pi_n$ for $n=1,2$. 

\begin{thm}
\label{S=1-refined0-depth1and2}
    In the setting of \Cref{S=1-refined0-equations}, the loci $X(\cO_K \otimes \bZ_p)_{S,n}^{(0)}$ for $n=1,2$ are given as follows:
    \begin{itemize}
        \item in depth~$n=1$:
        \[ X(\cO_K \otimes \bZ_p)_{S,1}^{(0)} = \bigcup_{1 \neq \zeta \in \mu(\bZ_p)} \{ (1-\zeta,1-\zeta), (1-\zeta,1-\zeta^{-1}) \}, \]
        where the union is over all nontrivial roots of unity in~$\bZ_p$;
        \item in depth~$n=2$:
        \begin{align*}
            X(\cO_K \otimes \bZ_p)_{S,2}^{(0)} &\subseteq \{ (1-\zeta,1-\zeta^{-1}) : 1 \neq \zeta \in \mu(\bZ_p) \}\\
            & \qquad \cup \{(1-\zeta,1-\zeta) : \pm 1 \neq \zeta \in \mu(\bZ_p) \textnormal{ s.\,t.\ $\Li_2(\zeta) = 0$}\}, 
        \end{align*}
        with equality assuming the $p$-adic period conjecture. The second set in the union is empty for all primes $p  < 2{,}000$.
    \end{itemize}
\end{thm}

\begin{proof}
   In depth $n=1$, both $1-z_1$ and $1-z_2$ are roots of unity by \eqref{eq:S=1-refined0-depth1}. So we write $z_1=1-\zeta_1$ and $z_2=1-\zeta_2$ for roots of unity $\zeta_1,\zeta_2\ne 1$ in $\mathbb{Z}_p$. By \eqref{eq:S=1-stable-log}, $z_1/z_2$ is a root of unity, denoted by $\eta$. Then $1-\zeta_1=\eta(1-\zeta_2)$, and hence $|1-\zeta_1|=|1-\zeta_2|$ with respect to the complex norm after embedding into $\mathbb{C}$. Geometrically, $\zeta_1$ and $\zeta_2$ which lie on the unit circle have the same distance from $1$, which implies that either $\zeta_2=\zeta_1$ or $\zeta_2=\zeta_1^{-1}$. So we get the desired result.

   In depth $n=2$, for a root of unity $\zeta\ne 1$, $(z_1,z_2)=(1-\zeta,1-\zeta^{-1})$ always satisfies \eqref{eq:S=1-depth-2} by the functional equation $L_2(1-z)+L_2(1-z^{-1})=0$ as we discussed in the proof of \Cref{S=1-stable-depth2}, and thus belongs to $X(\cO_K \otimes \bZ_p)_{S,2}^{(0)}$ assuming the $p$-adic period conjecture. For a root of unity $\zeta\ne \pm 1$, $(z_1,z_2)=(1-\zeta,1-\zeta)$ satisfies \eqref{eq:S=1-depth-2} if and only if $\Li_2(\zeta)=0$ by the functional equation $L_2(z)+L_2(1-z)=0$, which is checked to be impossible for $p<2{,}000$.
\end{proof}

To determine the loci $X(\cO_K \otimes \bZ_p)_{S,\PL,n}^{(0)}$ for $n=3,4$ we need to know the value of the constants $c_3$ and $c_4$ from~\eqref{eq:c3-c4}, which we achieve in the cases $K = \bQ(i)$, $S = \{(1-i)\}$, and $K = \bQ(\zeta_3)$, $S = \{(1-\zeta_3)\}$, using \Cref{S=1-stable-coeffs-from-known-solution} with $z = 1-i$ and $z = 1- \zeta_3$, respectively.

\begin{thm}
\label{Qzeta3-refined0-depth3and4}
    Let $K = \bQ(\zeta_3)$, $S = \{(1-\zeta_3)\}$, and let $p$ be a rational prime with $p \equiv 1 \bmod 3$. Then we have
    \begin{align*}
        X(\cO_K \otimes \bZ_p)_{S,\PL,3}^{(0)} &= \bigl\{ (\zeta_6,\zeta_6^{-1}), \; (\zeta_6^{-1},\zeta_6),\; (1-\zeta_3, 1-\zeta_3^{-1}),\; (1-\zeta_3^{-1},1-\zeta_3),\; (2,2) \bigr\},\\
        X(\cO_K \otimes \bZ_p)_{S,\PL,4}^{(0)} &= \bigl\{ (\zeta_6,\zeta_6^{-1}), \; (\zeta_6^{-1},\zeta_6),\; (1-\zeta_3, 1-\zeta_3^{-1}),\; (1-\zeta_3^{-1},1-\zeta_3) \bigr\} \\
        &= X(\cO_{K,S})_{(0)},
    \end{align*}
    for all $p < 2{,}000$.
\end{thm}

\begin{proof}
    Let $z=1-\zeta_3$. We have $\log(z)\ne 0$, and we checked $L_2(z)\ne 0$ and $L_3(z)+L_3(\sigma z)\ne 0$ for $p<2{,}000$. Then \Cref{S=1-stable-coeffs-from-known-solution} determines $c_3$, $c_4$, and we use the computer to filter depth $2$ solutions described in \Cref{S=1-refined0-depth1and2} through the depth $3$, $4$ equations \eqref{eq:S=1-depth-3-simple}, \eqref{eq:S=1-depth-4-simple}.
\end{proof}

Since $X(\cO_K \otimes \bZ_p)_{S,\PL,n}^{(0)} \supseteq X(\cO_{K,S})_{(0)}$ holds for all $n \geq 1$ by construction, the equality for $n = 4$ in \Cref{Qzeta3-refined0-depth3and4} implies the equality for all $n \geq 4$. It is interesting to observe that the loci $X(\cO_K \otimes \bZ_p)_{S,\PL,n}^{(0)}$ are finite already in depth~1 and yet it takes three more steps in the depth filtration to reach $X(\cO_{K,S})_{(0)}$ in depth~$4$.

\begin{thm}
\label{Qi-refined0-depth3and4}
    Let $K = \bQ(i)$, $S = \{(1-i)\}$, and let $p$ be a rational prime with $p \equiv 1 \bmod 4$. Then we have
    \begin{align*}
        X(\cO_K \otimes \bZ_p)_{S,\PL,3}^{(0)} &= X(\cO_K \otimes \bZ_p)_{S,\PL,4}^{(0)} \\
        &= \bigl\{ (1-i,1+i),\; (1+i,1-i),\; (2,2) \bigr\} \sqcup \bigl\{ (\zeta_6,\zeta_6^{-1}) : \zeta_6 \in \mu_6^{\prim}(\bZ_p) \bigr\} \\
        &= X(\cO_{K,S})_{(0)} \sqcup \bigl\{ (\zeta_6,\zeta_6^{-1}) : \zeta_6 \in \mu_6^{\prim}(\bZ_p) \bigr\},
    \end{align*}
    for all $p < 2{,}000$.
\end{thm}

The depth-3 locus in \Cref{Qi-refined0-depth3and4} is already as  small as possible: indeed, it follows from \Cref{thm:roots-of-unity} and \Cref{S=1-refined0-equations} that the pairs $(\zeta_6,\zeta_6^{-1})$ and $(\zeta_6^{-1},\zeta_6)$ are contained in the infinite depth locus $X(\cO_K \otimes \bZ_p)_{S,\PL,\infty}^{(0)}$ when $\zeta_6 \in \bQ_p$ and the period conjecture holds for~$p$.

The following table summarises what we know about the size of $X(\cO_K \otimes \bZ_p)_{S,\PL,n}^{(0)}$ for $p < 2{,}000$ from Theorems~\ref{S=1-refined0-depth1and2}, \ref{Qzeta3-refined0-depth3and4}, \ref{Qi-refined0-depth3and4}:

{\renewcommand{\arraystretch}{1.4}%
\begin{center}
\begin{tabular}{|c|c|c|}
    \hline
    & $K = \bQ(\zeta_3)$, $S = \{(1-\zeta_3)\}$ & $K = \bQ(i)$, $S = \{(1-i)\}$ \\
    \hline
    $\#X(\cO_K \otimes \bZ_p)_{S,\PL,1}^{(0)}$ & $2p-5$ & $2p-5$ \\
    $\#X(\cO_K \otimes \bZ_p)_{S,\PL,2}^{(0)}$ & $p-2$ & $p-2$ \\
    $\#X(\cO_K \otimes \bZ_p)_{S,\PL,3}^{(0)}$ & $5$ & 3 or 5 \\
    $\#X(\cO_K \otimes \bZ_p)_{S,\PL,4}^{(0)}$ & $4$ & 3 or 5\\ \hline
\end{tabular}
\end{center}
}
The entries ``3 or 5'' depend on whether $\bZ_p$ contains a primitive sixth root of unity: they are $5$ when $p \equiv 1 \bmod 3$ and otherwise~$3$. Of course we expect the results that we verified computationally for many primes to hold also for larger primes. 

\section{Real quadratic fields}
\label{sec:real-quadratic}

Let $K$ be a real quadratic field. We study the refined and unrefined Chabauty--Kim loci in small depth when $S$ is a set of primes of~$K$ of size $\leq 1$. We denote by $\sigma$ the nontrivial element of the Galois group $\Gal(K/\bQ)$.

\subsection{Integral points}
\label{sec:real-integral-points}

We start with the case $S = \emptyset$. It is known that $K=\mathbb{Q}(\sqrt{5})$ is the only real quadratic field such that $X(\mathcal{O}_K)\neq \emptyset$. In this case, $X(\mathcal{O}_K)$ is the $S_3$-orbit of $(-1+\sqrt{5})/2$ \cite[Th\'eor\`eme~1]{nagell_1970_quelques}:
\begin{equation}
\label{eq:sqrt5-orbit}
    X(\bZ[(1+\sqrt{5})/2]) = \Bigl\{ \frac{-1 \pm \sqrt{5}}2, \frac{3 \pm \sqrt{5}}2, \frac{1 \pm \sqrt{5}}2 \Bigr\}.
\end{equation}

We first discuss the choice of generators of $U_{\emptyset}^{\MT}$, the unipotent part of the mixed Tate motivic Galois group of $\cO_K$, specialising the discussion of §\ref{sec:coordinates} to the present situation. Since $K$ is real quadratic, $\mathcal{O}_K^\times\otimes\mathbb{Q}$ is 1-dimensional. Choose a generator~$\beta$ and define $\tau_{\beta} \in \Lie(U_{\emptyset}^{\MT})_{-1}$ as its dual. There are no Lie algebra generators in degree~$-2$ by~\eqref{eq:borel-dimensions}. In degree~$-3$ we have $r_1(K) + r_2(K) = 2$ generators $\sigma_3$ and $\chi_3$. Since the subspace of $\Ext^1_{\MT(K,\bQ)}(\bQ(0),\bQ(3))$ fixed by $\Gal(K/\bQ)$ has dimension $r_1(\bQ) + r_2(\bQ) = 1$, we can choose $\sigma_3$ and $\chi_3$ such that $\sigma^* \sigma_3 = \sigma_3$ and $\sigma^* \chi_3 = -\chi_3$, using the same argument as in \Cref{lem:sigmasigman}. This determines $\sigma_3$ and $\chi_3$ up to scaling. There are no generators in degree~$-4$. In summary:
\[ \Sigma_1 = \{\tau_{\beta}\}, \quad \Sigma_2 = \emptyset,\quad \Sigma_3 = \{\sigma_3,\chi_3\}, \quad \Sigma_4 = \emptyset. \]

Now let $p$ be a rational prime which splits in~$K$, let $\fp_1$, $\fp_2$ be the two primes of~$K$ above~$p$, and let $\eps_1$, $\eps_2$ be the associated period elements in $\Lie(U_{\emptyset}^{\MT})(K_{\fp_{i}})$ for $i=1,2$. In the chosen generators, $\eps_1$ can be written as
\begin{equation}
    \label{eq:realquad-integral-eps1}
    \eps_1 = c_{\tau_{\beta}} \tau_{\beta} + c_{\sigma_3} \sigma_3 + c_{\chi_3} \chi_3 + c_{\tau_{\beta}\sigma_3} [\tau_{\beta},\sigma_3] + c_{\tau_{\beta}\chi_3}[\tau_{\beta},\chi_3] + \ldots
\end{equation}
with coefficients $c_w \in K_{\fp_1} = \bQ_p$ and omitted elements having degree $< -4$. By \Cref{tau-coeffs} we have $c_{\tau_{\beta}} = \log(\beta) \neq 0$, where the $p$-adic logarithm of~$\beta$ is defined via $K \hookrightarrow K_{\fp_1} = \bQ_p$. Assuming the $p$-adic period conjecture, we also have $c_{\sigma_3}, c_{\chi_3} \neq 0$.

We now derive the equations for the Chabauty--Kim locus $X(\cO_K \otimes \bZ_p)_{\emptyset,\PL,n}$ for $n \leq 4$. It turns out that the equations in depth~$\leq 2$ do not depend on~$K$ and~$S$. At the end of this subsection, we derive the equations in arbitrary depth, see \Cref{realquad:higher-depth}.

\begin{thm}
\label{real-integral-equations}
    Let $K$ be a real quadratic field and let $p$ be a prime which splits in~$K$. Then for $n=1,2,3,4$, the depth-$n$ polylogarithmic Chabauty--Kim locus $X(\cO_K \otimes \bZ_p)_{\emptyset,\PL,n}$ is contained in the set of pairs $(z_1,z_2) \in X(\bZ_p) \times X(\bZ_p)$ satisfying the following incremental list of equations:
    \begin{itemize}
        \item in depth~1:
        \begin{align}
            \log(z_1) +\log(z_2) &= 0, \label{eq:log-equations-real}\\
            \Li_1(z_1) + \Li_1(z_2) &= 0; \label{eq:Li1-equations-real}
        \end{align}
        \item in depth~2:
        \begin{equation}
        \label{eq:L2-equations-real}
            L_2(z_1)=0, \qquad L_2(z_2)=0;
        \end{equation}
        \item in depth~3: no additional equations;
        \item in depth~4:
        \begin{align}
        \label{eq:realquadSemptydepth4-1}
             c_{\tau_{\beta}} c_{\chi_3} (L_4(z_1) + L_4(z_2)) &=  c_{\tau_{\beta} \chi_3}\log(z_1)(L_3(z_1) - L_3(z_2)),\\
        \label{eq:realquadSemptydepth4-2}
            c_{\tau_{\beta}} c_{\sigma_3}(L_4(z_1) - L_4(z_2)) &= c_{\tau_{\beta} \sigma_3} \log(z_1)(L_3(z_1) + L_3(z_2)).
        \end{align}
    \end{itemize}
    The containment is an equality for $n \leq 2$. If $c_{\sigma_3}, c_{\chi_3} \neq 0$ then it is also an equality for $n=3,4$.
\end{thm}

\begin{proof}
    The $\fp_2$-adic period element $\eps_2 = \sigma^* \eps_1$ is obtained by applying $\sigma^*$ to~\eqref{eq:realquad-integral-eps1}.
    We have $\sigma \beta = \beta^{-1}$ in $\cO_K^{\times} \otimes \bQ$ since $\beta$ has norm~$\pm 1$, hence $\sigma^* \tau_{\beta} = -\tau_{\beta}$. Recall that $\sigma_3$ and $\chi_3$ are chosen such that $\sigma^*\sigma_3 = \sigma_3$ and $\sigma^* \chi_3 = -\chi_3$. Hence we get
    \begin{equation*}
    \label{eq:realquad-integral-eps2}
        \eps_2 = -c_{\tau_{\beta}} \tau_{\beta} + c_{\sigma_3} \sigma_3 - c_{\chi_3} \chi_3 - c_{\tau_{\beta}\sigma_3} [\tau_{\beta},\sigma_3] + c_{\tau_{\beta}\chi_3}[\tau_{\beta},\chi_3] + \ldots
    \end{equation*}
    By §\ref{sec:coordinates}, the polylogarithmic Selmer scheme in depth~4 is a 4-dimensional affine space with coordinates $x_{\beta},y_{\beta},z,w$ with $z \coloneqq z_{3,1}$ corresponding to~$\sigma_3$ and $w \coloneqq z_{3,2}$ corresponding to~$\chi_3$. By \Cref{cocycle-evaluation-map}, the cocycle evaluation maps at $\eps_i$ ($i=1,2$) are given by
    \begin{align*}
        \ev_{\eps_i}^{\sharp} L_0 &= (-1)^{i+1} c_{\tau_{\beta}} x_{\beta},\\
        \ev_{\eps_i}^{\sharp} L_1 &= (-1)^{i+1} c_{\tau_{\beta}} y_{\beta},\\
        \ev_{\eps_i}^{\sharp} L_2 &= 0,\\
        \ev_{\eps_i}^{\sharp} L_3 &= c_{\sigma_3} z + (-1)^{i+1} c_{\chi_3} w,\\
        \ev_{\eps_i}^{\sharp} L_4 &= (-1)^{i+1} c_{\tau_{\beta}\sigma_3} x_{\beta}z + c_{\tau_{\beta}\chi_3} x_{\beta} w.
    \end{align*}
    Denoting by $L_n^{(i)}$ the pullback of $L_n$ along the $i$-th projection of $\Lie(\Pi_{\PL,4}^{\dR}) \times \Lie(\Pi_{\PL,4}^{\dR})$, the scheme-theoretic image of $\ev_{\eps_1} \times \ev_{\eps_2}$ is contained in the subscheme cut out by the equations
    \begin{gather*}
        L_0^{(1)} + L_0^{(2)} = 0 ,\qquad L_1^{(1)} + L_1^{(2)} = 0,\\
        L_2^{(1)} = 0, \qquad L_2^{(2)} = 0,\\
        c_{\tau_{\beta}} c_{\sigma_3}(L_4^{(1)} + L_4^{(2)}) = c_{\tau_{\beta}\sigma_3} L_0^{(1)}(L_3^{(1)} - L_3^{(2)}),\\
        c_{\tau_{\beta}} c_{\chi_3}(L_4^{(1)} - L_4^{(2)}) = c_{\tau_{\beta}\chi_3} L_0^{(1)}(L_3^{(1)} + L_3^{(2)}),
    \end{gather*}
    and pulling this back to $X(\bZ_p) \times X(\bZ_p)$ yields precisely the claimed equations. If $c_{\sigma_3}, c_{\chi_3} \neq 0$ then the scheme-theoretic image is exactly cut out by the equations above. 
\end{proof}

\begin{lemma}
\label{solving-realquad-depth1}
    Let $p$ be a prime. The common vanishing set in $X(\bZ_p) \times X(\bZ_p)$ of the equations $\log(z_1) + \log(z_2) = 0$ and $\log(1-z_1) + \log(1-z_2) = 0$ is the set of pairs of the form
    \[ \left(\frac{1+\zeta-\eta \pm \sqrt{(1+\zeta-\eta)^2 - 4\zeta}}{2}, \frac{1+\zeta-\eta \mp \sqrt{(1+\zeta-\eta)^2 - 4\zeta}}{2} \right) \]
    where $\zeta, \eta$ are roots of unity in~$\bZ_p$ for which the square root exists in~$\bZ_p$.
\end{lemma}

\begin{proof}
    The equations are equivalent to the vanishing of $\log(z_1 z_2)$ and $\log((1-z_1)(1-z_2))$, hence to $z_1z_2$ and $(1-z_1)(1-z_2)$ being roots of unity. Assume that
    \[ \zeta \coloneqq z_1 z_2 \quad \text{and} \quad \eta \coloneqq (1-z_1)(1-z_2) \]
    are roots of unity. Then we have
    \[ (X - z_1)(X - z_2) = X^2 - (1+\zeta-\eta)X + \zeta, \]
    hence $z_1$ and $z_2$ are the two roots to the quadratic equation on the right hand side.
\end{proof}

Using \Cref{solving-realquad-depth1} we can compute the locus $X(\cO_K \otimes \bZ_p)_{\emptyset,n} = X(\cO_K \otimes \bZ_p)_{\emptyset,\PL,n}$ from \Cref{real-integral-equations} in depth~$n \leq 2$. 

\begin{thm}
\label{real-integral-depth2-locus}
    Let $K$ be a real quadratic field and let $p$ be a prime that splits completely in $K$. If $p \equiv \pm 1 \bmod 5$ then $X(\mathcal{O}_K\otimes\mathbb{Z}_p)_{\emptyset,2}$ contains the $S_3$-orbit of $\left(\frac{-1+\sqrt{5}}{2},\frac{-1-\sqrt{5}}{2}\right)$. We have
    \[ 
        X(\mathcal{O}_K\otimes\mathbb{Z}_p)_{\emptyset,2} = 
        \begin{cases}
            S_3.\!\left(\frac{-1+\sqrt{5}}{2},\frac{-1-\sqrt{5}}{2}\right), & \text{if $p \equiv \pm 1 \bmod 5$},\\[1mm]
            \emptyset, & \text{otherwise,}
        \end{cases}
    \]
    for all $p < 2{,}000$.
\end{thm}

\begin{proof}
    If $p \equiv \pm 1 \bmod 5$ then $p$ splits in $\bQ(\sqrt{5})$. Since the $S_3$-orbit of $\frac{-1+\sqrt{5}}{2}$ is contained in $X(\mathcal{O}_{\mathbb{Q}(\sqrt{5})})$, the $S_3$-orbit of $\left(\frac{-1+\sqrt{5}}{2},\frac{-1-\sqrt{5}}{2}\right)$ must satisfy the equations \eqref{eq:log-equations-real}--\eqref{eq:L2-equations-real}, and hence is contained in $X(\mathcal{O}_K\otimes\mathbb{Z}_p)_{\emptyset,2}$.
    We computed the locus for all $p < 2{,}000$ and verified that it contains no additional points except this one $S_3$-orbit when $p \equiv \pm 1 \bmod 5$.
\end{proof}

\begin{rem}
    \label{rem:Qsqrt5-depth2}
    The fact that $L_2$ vanishes on the $S_3$-orbit of $(-1 \pm \sqrt{5})/2$ was already observed by Coleman \cite[Remark after Prop.~6.4]{coleman:dilogarithms}. \Cref{real-integral-equations} provides a new proof of this fact. The equation $L_2(z) = 0$ was also found by Dan-Cohen and Wewers for the single-prime Chabauty--Kim locus $X(\bZ_p)_{\emptyset,2}'$ \cite[§12.1]{DCW:explicitCK}.
\end{rem}

\begin{rem}
\label{real-selmer-section-conjecture}
    If one could prove that for any prime~$p$, the zero locus of the equations \eqref{eq:log-equations-real}--\eqref{eq:L2-equations-real} in $X(\bZ_p) \times X(\bZ_p)$ contains at most the $S_3$-orbit of $\left(\frac{-1+\sqrt{5}}{2},\frac{-1-\sqrt{5}}{2}\right)$ (i.e.\ \Cref{real-integral-depth2-locus} holds for all~$p$) then, arguing as in \Cref{selmer-section-conjecture}, we would have a proof of the $S$-Selmer Section Conjecture \cite[Conj.~1.6]{BKL:chabauty-kim-sc} for any real quadratic field $K$, $S = \emptyset$, $X = \bP^1 \smallsetminus \{0,1,\infty\}$. This would be a genuinely new result since the known cases from \cite[Corollary~6]{stix:birationalSC} assume that $K$ is $\bQ$ or imaginary quadratic.
\end{rem}

From \Cref{real-integral-depth2-locus} we immediately get the following corollary.
\begin{cor}
    When $K = \bQ(\sqrt{5})$ and $S = \emptyset$, Kim's Conjecture~\ref{conj:kim-full} holds for all $p < 2{,}000$ which split in~$K$. When $K$ is real quadratic and $\neq \bQ(\sqrt{5})$, and $S = \emptyset$,
    \begin{itemize}
        \item Kim's Conjecture~\ref{conj:kim-full} holds for all $p< 2{,}000$ which split in~$K$ and satisfy $p \not\equiv \pm 1 \bmod 5$;
        \item Kim's Conjecture~\ref{conj:kim-single} does not hold in depth~2 for all $p$ which split in $K$ and satisfy $p \equiv \pm 1 \bmod 5$.
    \end{itemize}
\end{cor}

The following lemma can be used to verify the non-vanishing of the periods $c_{\sigma_3}$ and $c_{\chi_3}$ in~$\bQ_p$ for $K = \bQ(\sqrt{5})$, and to determine the coefficients 
\[ \tilde c_{\tau_{\beta}\chi_3} \coloneqq  \frac{c_{\tau_{\beta}\chi_3}}{c_{\tau_{\beta}} c_{\chi_3}}, \quad  \tilde c_{\tau_{\beta}\sigma_3} \coloneqq \frac{c_{\tau_{\beta}\sigma_3}}{c_{\tau_{\beta}} c_{\sigma_3}}\]
appearing in the simplified depth-4 equations~\eqref{eq:realquadSemptydepth4-1}, \eqref{eq:realquadSemptydepth4-2}:
\begin{align}
    \label{eq:realquadSemptydepth4-1-simple}
     L_4(z_1) + L_4(z_2) &=  \tilde c_{\tau_{\beta} \chi_3}\log(z_1)(L_3(z_1) - L_3(z_2)),\\
     \label{eq:realquadSemptydepth4-2-simple}
     L_4(z_1) - L_4(z_2) &=  \tilde c_{\tau_{\beta} \sigma_3}\log(z_1)(L_3(z_1) + L_3(z_2)).
\end{align}

\begin{lemma}
    \label{realquad-integral-coeffs}
    Suppose $K = \bQ(\sqrt{5})$. Let $p$ be a rational prime satisfying $p \equiv \pm 1 \bmod 5$ and let $\alpha_1 \coloneqq (-1+\sqrt{5})/2$ and $\alpha_2 \coloneqq (-1-\sqrt{5})/2$. Assume that 
    \[ L_3(\alpha_1) \pm L_3(\alpha_2) \neq 0 \]
    for some choice of $\sqrt 5 \in \bQ_p$.
    Then $c_{\sigma_3}, c_{\chi_3} \neq 0$ in $\bQ_p$ and 
    \begin{align*}
        \tilde c_{\tau_{\beta} \chi_3} &= \frac{L_4(\alpha_1) + L_4(\alpha_2)}{\log(\alpha_1)(L_3(\alpha_1) - L_3(\alpha_2))},\\
        \tilde c_{\tau_{\beta} \sigma_3} &= \frac{L_4(\alpha_1) - L_4(\alpha_2)}{\log(\alpha_1)(L_3(\alpha_1) + L_3(\alpha_2))}.
    \end{align*}
\end{lemma}

\begin{proof}
    Recall from the proof of \Cref{real-integral-equations} that $\ev_{\eps_i}^{\sharp} L_3 = c_{\sigma_3} z + (-1)^{i+1} c_{\chi_3} w$ for $i=1,2$. As a consequence, $c_{\sigma_3} = 0$ would imply that $L_3(z_1) + L_3(z_2) = 0$ holds on $X(\cO_K \otimes \bZ_p)_{\emptyset,\PL,3}$. Similarly, $c_{\chi_3} = 0$ would imply that $L_3(z_1) - L_3(z_2) = 0$ holds on $X(\cO_K \otimes \bZ_p)_{\emptyset,\PL,3}$. The pair $(\alpha_1,\alpha_2)$ is contained in this locus since it comes from the integral point $(-1+\sqrt{5})/2 \in X(\cO_K)$, but we have $L_3(\alpha_1) \pm L_3(\alpha_2) \neq 0$ by assumption, thus we must have $c_{\sigma_3},c_{\chi_3} \neq 0$. Note that $\log(\alpha_1) \neq 0$ since $\alpha_1$ is not a root of unity. Hence the equations \eqref{eq:realquadSemptydepth4-1-simple}--\eqref{eq:realquadSemptydepth4-2-simple} can be solved for $\tilde c_{\tau_{\beta} \chi_3}$ and $\tilde c_{\tau_{\beta} \sigma_3}$ as claimed.
\end{proof}

For real quadratic fields other than $\bQ(\sqrt{5})$ we do not have any integral points available, which makes it more difficult to determine the coefficients $\tilde c_{\tau_{\beta} \chi_3}$ and $\tilde c_{\tau_{\beta} \sigma_3}$. However, we will show in \Cref{thm:Qsqrt2SemptyKimholds} below that for $K = \bQ(\sqrt{2})$ we can use $S'$-integral points for the larger set $S' = \{(\sqrt{2})\}$ to determine the coefficients and compute the depth 4 loci of \Cref{real-integral-equations}.

At the end of this subsection, we generalise \Cref{real-integral-equations} to depth $n$ equations for arbitrary $n\geq 1$.
\begin{thm}
\label{realquad:higher-depth}
    Let $K$ be a real quadratic field and let $p$ be a prime which splits in~$K$. Assume the $p$-adic period conjecture. Then for $n\geq 1$, the depth-$n$ polylogarithmic Chabauty--Kim locus $X(\cO_K \otimes \bZ_p)_{\emptyset,\PL,n}$ is the set of pairs $(z_1,z_2) \in X(\bZ_p) \times X(\bZ_p)$ satisfying the following incremental list of equations:
\begin{itemize}
\item in depth~1:
    \begin{equation*}
        \log(z_1) +\log(z_2) = 0, \quad
        \Li_1(z_1) + \Li_1(z_2) = 0;
    \end{equation*}
\item in depth~$m$ for $2\leq m\leq n$ and $m$ even:
\begin{align}
L_m(z_1) - L_m(z_2)
= \sum_{\substack{3\le k\le m-1\\ k\textnormal{ odd}}}
\lambda_{m,k}^{(1)}\log(z_1)^{m-k}
\left(
L_k(z_1) + L_k(z_2)
\right)
\end{align}
\begin{align}
L_m(z_1) + L_m(z_2)
= \sum_{\substack{3\le k\le m-1\\ k\textnormal{ odd}}}
\lambda_{m,k}^{(2)}\log(z_1)^{m-k}
\left(
L_k(z_1) - L_k(z_2)
\right)
\end{align}
where $\lambda_{m,k}^{(j)}$ is determined recursively by
\[
\lambda_{m,k}^{(j)} \coloneqq \frac{1}{c_{\sigma_{k,j}}} \left( \frac{c_{\tau_\beta^{m-k}\sigma_{k,j}}}{c_{\tau_\beta}^{m-k}} - \sum_{\substack{k<r<m\\ r\textnormal{ odd}}} \lambda_{m,r}^{(j)} \frac{c_{\tau_\beta^{r-k}\sigma_{k,j}}}{c_{\tau_\beta}^{r-k}} \right),
\quad
j=1,2,
\quad
k=m-1,m-3,\ldots,3.
\]
\end{itemize}
\end{thm}
\begin{proof}
Let $\fp_1$ and $\fp_2$ be the two primes of~$K$ above~$p$, and let $\eps_i \coloneqq \eps_{\fp_i}$ be the associated $\fp_i$-adic period elements ($i=1,2$). In degree $-1$, we choose $\tau_\beta$ to be the dual of $\beta$, and we have $\sigma \beta = \beta^{-1}$ in $\cO_K^{\times} \otimes \bQ$ since $\beta$ has norm~$\pm 1$. Dually, we find $\sigma^*\tau_{\beta}=-\tau_{\beta}$. In degree $\leq -2$,
$$\dim_{\bQ}\Ext^1_{\MT(K,\bQ)}(\bQ(0),\bQ(k))=\begin{cases}
    r_2(K)=0 & k\geq 2 \text{ even},\\
    r_1(K)+r_2(K)=2 & k\geq 3 \text{ odd},
\end{cases}
$$
and for odd $k\geq 3$ we have
$$\dim_{\bQ}\Ext^1_{\MT(K,\bQ)}(\bQ(0),\bQ(k))^{\Gal(K/\bQ)}=r_1(\bQ)+r_2(\bQ)=1.
$$

As a result, we can choose the Lie algebra generators $\sigma_{k,1},\sigma_{k,2}$ in degree $-k$ for odd $k\geq 3$ such that $\sigma^*\sigma_{k,1}=\sigma_{k,1}$ and $\sigma^*\sigma_{k,2}=-\sigma_{k,2}$, and there are no Lie algebra generators in degree $-k$ for even $k\geq 2$. By \Cref{lie-algebra-element-expansion}, we can write
$$\eps_1=c_{\tau_\beta}\tau_\beta+\sum_{m\geq 2}\, \sum_{\substack{3 \le k \le m \\ k \text{ odd}}}  \bigl( c_{\tau_\beta^{\,m-k} \sigma_{k,1}} \ad(\tau_\beta)^{m-k}\sigma_{k,1} + c_{\tau_\beta^{\,m-k} \sigma_{k,2}} \ad(\tau_\beta)^{m-k}\sigma_{k,2} \bigr)$$

Thus, applying $\sigma^*$, the $\fp_2$-adic period element $\eps_2 = \sigma^* \eps_1$, is given by
$$\eps_2=-c_{\tau_\beta}\tau_\beta+\sum_{m\geq 2} \,\sum_{\substack{3 \le k \le m \\ k \text{ odd}}}  (-1)^{m-k}\bigl( c_{\tau_\beta^{\,m-k} \sigma_{k,1}} \ad(\tau_\beta)^{m-k}\sigma_{k,1} - c_{\tau_\beta^{\,m-k} \sigma_{k,2}} \ad(\tau_\beta)^{m-k}\sigma_{k,2} \bigr)$$

Specialising the discussion of §\ref{sec:coordinates} to the present situation, the Selmer scheme $\Sel^{\mot}_{S,\PL,n}(X)$ is an affine space with coordinates $x_{\beta}, y_{\beta}, z_{k,j}$ for $j=1,2$ and odd $3\leq k\leq n$ where $z_{k,j}$ corresponds to $\sigma_{k,j}$. By \Cref{cocycle-evaluation-map}, the evaluation maps at $\eps_1$ and $\eps_2$ in depth~$m\geq 2$ are given by
\begin{align}
\label{eq:real_empty_eveps1}
\ev_{\eps_1}^{\sharp} L_m &= \sum_{\substack{3 \le k \le m \\ k \text{ odd}}} x_\beta^{\,m-k} \bigl( c_{\tau_\beta^{\,m-k} \sigma_{k,1}} z_{k,1} + c_{\tau_\beta^{\,m-k} \sigma_{k,2}} z_{k,2} \bigr) \\
\label{eq:real_empty_eveps2}
\ev_{\eps_2}^{\sharp} L_m &= \sum_{\substack{3 \le k \le m \\ k \text{ odd}}} (-1)^{m-k} x_\beta^{\,m-k} \bigl(  c_{\tau_\beta^{\,m-k} \sigma_{k,1}} z_{k,1} - c_{\tau_\beta^{\,m-k} \sigma_{k,2}} z_{k,2} \bigr)
\end{align} 

Let $L_m^{(j)}$ denote $L_m$ composed with the $j$-th projection of $\Lie(\Pi_{\PL,n}^{\dR}) \times \Lie(\Pi_{\PL,n}^{\dR})$ for $j=1,2$. The graph of $\ev_{\eps_1} \times \ev_{\eps_2}$ is the affine variety with coordinate ring
\begin{equation}
    \label{eq:real-empty-graph-ring}
        R \coloneqq \bQ_p[x_{\beta},y_{\beta},(z_{m,j})_{m,j}, (L_m^{(j)})_{m,j}] \bigm/ (L_m^{(j)} - \ev_{\eps_j}^{\sharp}L_m).
\end{equation}

We have derived the depth $1$ and $2$ equations in \Cref{real-integral-equations}, so we discuss only the cases $m\geq 3$ now. First consider the case where $m\geq 4$ is even. In $R$, by adding and subtracting \eqref{eq:real_empty_eveps1}, \eqref{eq:real_empty_eveps2}, 

\begin{equation}
\label{eq:Lm1Lm2}
  L_m^{(1)} - L_m^{(2)} = 2\sum_{\substack{3\le r\le m-1\\ r\ \text{odd}}} c_{\tau_\beta^{m-r}\sigma_{r,1}}\,x_\beta^{m-r}z_{r,1},\quad 
L_m^{(1)} + L_m^{(2)} = 2\sum_{\substack{3\le r\le m-1\\ r\ \text{odd}}} c_{\tau_\beta^{m-r}\sigma_{r,2}}\,x_\beta^{m-r}z_{r,2}.  
\end{equation}
For odd $k\geq 3$ which is smaller than $m$, similarly we have
$$
L_k^{(1)} + L_k^{(2)} = 2\sum_{\substack{3\le r\le k\\ r\ \text{odd}}} c_{\tau_\beta^{k-r}\sigma_{r,1}}\,x_\beta^{k-r}z_{r,1},\quad
L_k^{(1)} - L_k^{(2)} = 2\sum_{\substack{3\le r\le k\\ r\ \text{odd}}} c_{\tau_\beta^{k-r}\sigma_{r,2}}\,x_\beta^{k-r}z_{r,2}.
$$
Denote $P_k^{(1)} \coloneq L_k^{(1)} + L_k^{(2)}$ and $P_k^{(2)} \coloneq L_k^{(1)} - L_k^{(2)}$. Then we have the following unified expression
$$P_k^{(j)} = 2\sum_{\substack{3\le r\le k\\ r\ \text{odd}}} c_{\tau_\beta^{k-r}\sigma_{r,j}}\,x_\beta^{k-r}z_{r,j}.$$
Since $L_0^{(1)} = c_{\tau_\beta} x_\beta$, multiplying by $(L_0^{(1)})^{m-k}$ yields
$$(L_0^{(1)})^{m-k}P_k^{(j)} = 2\sum_{\substack{3\le r\le k\\ r\ \text{odd}}} c_{\tau_\beta}^{m-k}\,c_{\tau_\beta^{k-r}\sigma_{r,j}}\,x_\beta^{m-r}z_{r,j}.$$
Let $k_1 = m-1,\; k_2 = m-3,\; \dots,\; k_N = 3$ where $N = (m-2)/2$. The relation above becomes the matrix equation
$$\mathbf{y} = A \mathbf{x}$$
where
\begin{itemize}
    \item $\mathbf{y} = \bigl( (L_0^{(1)})^{m-k_i} P_{k_i}^{(j)} \bigr)_{i=1}^N$ and $\mathbf{x} = \bigl( x_\beta^{m-k_i} z_{k_i,j} \bigr)_{i=1}^N$;
\item $A = (a_{i\ell})$ is an $N\times N$ matrix defined by 
$$a_{i\ell} =
\begin{cases}
2\,c_{\tau_\beta}^{m-k_i}\,c_{\tau_\beta^{k_i-k_\ell}\sigma_{k_\ell,j}}, & \text{if } \ell \ge i \ (\text{i.e. } k_\ell \le k_i),\\
0, & \text{if } \ell < i \ (\text{i.e. } k_\ell > k_i).
\end{cases}$$
\end{itemize}
We find that $A$ is a lower triangular matrix with the diagonal entries $a_{ii} = 2\,c_{\tau_\beta}^{m-k_i}\,c_{\sigma_{k_i,j}}$ which are all nonzero by the $p$-adic period conjecture, hence $A$ is invertible. In other words, the set
$$\bigl\{(L_0^{(1)})^{m-k}P_k^{(j)} \mid 3\le k\le m-1,\; k\ \text{odd}\bigr\}$$
forms a basis of the space spanned by $\{x_\beta^{m-r}z_{r,j}\}_{3\le r\le m-1,\; r\ \text{odd}}$. The elements $L_m^{(1)} - L_m^{(2)}$, $L_m^{(1)} + L_m^{(2)}$ belong to the space by \eqref{eq:Lm1Lm2}. Consequently, there exist unique coefficients $\lambda_{m,k}^{(j)}$ such that
\begin{align*}
    L_m^{(1)} - L_m^{(2)} &= 2\sum_{\substack{3\le r\le m-1\\ r\ \text{odd}}} c_{\tau_\beta^{m-r}\sigma_{r,1}}\,x_\beta^{m-r}z_{r,1} = \sum_{\substack{3\le k\le m-1\\ k\ \text{odd}}} \lambda_{m,k}^{(1)} (L_0^{(1)})^{m-k}\bigl(L_k^{(1)} + L_k^{(2)}\bigr),\\
    L_m^{(1)} + L_m^{(2)} &= 2\sum_{\substack{3\le r\le m-1\\ r\ \text{odd}}} c_{\tau_\beta^{m-r}\sigma_{r,2}}\,x_\beta^{m-r}z_{r,2} = \sum_{\substack{3\le k\le m-1\\ k\ \text{odd}}} \lambda_{m,k}^{(2)} (L_0^{(1)})^{m-k}\bigl(L_k^{(1)} - L_k^{(2)}\bigr).
\end{align*}
The coefficients $\lambda_{m,k}^{(j)}$ are obtained by applying $A^{-1}$ to the vector representing the left-hand sides. Because $A$ is lower triangular, the system can be solved by back-substitution, yielding the explicit recurrence
\[
\lambda_{m,k}^{(j)} \coloneqq \frac{1}{c_{\sigma_{k,j}}} \left( \frac{c_{\tau_\beta^{m-k}\sigma_{k,j}}}{c_{\tau_\beta}^{m-k}} - \sum_{\substack{k<r<m\\ r\text{ odd}}} \lambda_{m,r}^{(j)} \frac{c_{\tau_\beta^{r-k}\sigma_{k,j}}}{c_{\tau_\beta}^{r-k}} \right),
\quad
j=1,2,
\quad
k=m-1,m-3,\ldots,3.
\]

Next consider the case where $m\geq 3$ is odd. In $R$, by adding and subtracting \eqref{eq:real_empty_eveps1}, \eqref{eq:real_empty_eveps2},
\begin{align*}
L_m^{(1)} + L_m^{(2)} &= 2\sum_{\substack{3\le r\le m\\ r\ \text{odd}}} c_{\tau_\beta^{m-r}\sigma_{r,1}}\,x_\beta^{m-r}z_{r,1}=2c_{\sigma_{m,1}}z_{m,1} + (\text{lower depth terms}),\\
L_m^{(1)} - L_m^{(2)} &= 2\sum_{\substack{3\le r\le m\\ r\ \text{odd}}} c_{\tau_\beta^{m-r}\sigma_{r,2}}\,x_\beta^{m-r}z_{r,2}= 2c_{\sigma_{m,2}}z_{m,2} + (\text{lower depth terms}).
\end{align*}
Since $c_{\sigma_{m,1}}$ and $c_{\sigma_{m,2}}$ are nonzero by the $p$-adic period conjecture, these two equations can be used to solve for the new depth variables $z_{m,1}$ and $z_{m,2}$ instead of giving relations that only involve lower depths. Hence, when $m$ is odd, no extra equations arise. This completes the proof.
\end{proof}

\subsection{Galois-stable $S$ of size~1}
\label{sec:real-S=1}

Now consider a real quadratic field~$K$ and $S = \{\fl\}$ of size~$1$ with $\fl$ a Galois-stable prime of~$K$. Let $\sigma$ be the nontrivial element of $\Gal(K/\bQ)$. Let $\beta$ be a generator of $\cO_K^{\times} \otimes \bQ$ and let $\pi_{\fl}$ be a generator of $\cO_{K,\{\fl\}}^{\times} \otimes \bQ$ modulo $\beta$ such that $\sigma \pi_{\fl} = \pi_{\fl}$ holds in $\cO_{K,\{\fl\}}^{\times} \otimes \bQ$. For example, $\pi_{\fl}$ could be the rational prime under~$\fl$. We choose the Lie algebra generators $\tau_{\beta}$ and $\tau_{\fl}$ of $\Lie(U_S^{\MT})$ in degree~$-1$ as the dual basis of $\beta$ and $\pi_{\fl}$. As before, we choose the Lie algebra generators $\sigma_3, \chi_3$ in degree $-3$ such that $\sigma^* \sigma_3 = \sigma_3$ and $\sigma^* \chi_3 = -\chi_3$. There are no generators in even degrees. In summary:
\[ \Sigma_1 = \{\tau_{\beta} ,\tau_{\fl} \}, \quad \Sigma_2 = \emptyset,\quad \Sigma_3 = \{\sigma_3,\chi_3\}, \quad \Sigma_4 = \emptyset. \]

Now let $p$ be a rational prime not divisible by~$\fl$ which splits in~$K$, let $\fp_1$ and $\fp_2$ be the two primes of~$K$ above~$p$, and let $\eps_i \coloneqq \eps_{\fp_i}$ be the associated $\fp_i$-adic period elements ($i=1,2$). By \Cref{lie-algebra-element-expansion}, we can write
\begin{align}
\label{eq:real-S=1-eps1}
    \eps_1 &= c_{\tau_{\beta}} \tau_{\beta} + c_{\tau_{\fl}} \tau_{\fl} + c_{\tau_{\fl}\tau_{\beta}} [\tau_{\fl},\tau_{\beta}] + c_{\sigma_3} \sigma_3 + c_{\chi_3} \chi_3
    + c_{\tau_{\beta}\tau_{\fl}\tau_{\beta}} [\tau_{\beta},[\tau_{\fl},\tau_{\beta}]] + c_{\tau_{\fl}\tau_{\fl}\tau_{\beta}} [\tau_{\fl},[\tau_{\fl},\tau_{\beta}]]\\ \notag
    & + c_{\tau_{\beta}\sigma_3} [\tau_{\beta},\sigma_3] + c_{\tau_{\fl}\sigma_3} [\tau_{\fl},\sigma_3]  + c_{\tau_{\beta}\chi_3} [\tau_{\beta},\chi_3] +c_{\tau_{\fl}\chi_3} [\tau_{\fl},\chi_3] \\ \notag
    &+ c_{\tau_{\beta}\tau_{\beta}\tau_{\fl}\tau_{\beta}} [\tau_{\beta},[\tau_{\beta}, [\tau_{\fl}, \tau_{\beta}]]] 
    + c_{\tau_{\beta}\tau_{\fl}\tau_{\fl}\tau_{\beta}} [\tau_{\beta},[\tau_{\fl}, [\tau_{\fl}, \tau_{\beta}]]] 
    + c_{\tau_{\fl}\tau_{\fl}\tau_{\fl}\tau_{\beta}} [\tau_{\fl},[\tau_{\fl}, [\tau_{\fl}, \tau_{\beta}]]] + \ldots
\end{align}
with coefficients $c_w \in \bQ_p$, where the omitted elements belong to the Goncharov ideal or have degree $< -4$.
We have $c_{\tau_{\fl}} = \log(\pi_{\fl})$, $c_{\tau_{\beta}} = \log(\beta)$ by \Cref{tau-coeffs}. Here, the logarithms of $\beta$ and $\pi_{\fl}$ are defined via $K \hookrightarrow K_{\fp_1} = \bQ_p$. Note that $c_{\tau_{\fl}}$, $c_{\tau_{\beta}} \neq 0$. We also define the normalised coefficients
\begin{gather*}
    \tilde c_{\tau_{\fl} \tau_{\beta}} \coloneqq \frac{c_{\tau_{\fl} \tau_{\beta}}}{c_{\tau_{\fl}} c_{\tau_{\beta}}}, \\
    \quad \tilde c_{\tau_{\beta}\sigma_3} \coloneqq \frac{c_{\tau_{\beta}\sigma_3}}{c_{\tau_{\beta}}c_{\sigma_3}}, 
    \quad \tilde c_{\tau_{\fl}\sigma_3} \coloneqq \frac{c_{\tau_{\fl}\sigma_3}}{c_{\tau_{\fl}}c_{\sigma_3}}, 
    \quad \tilde c_{\tau_{\fl}\chi_3} \coloneqq \frac{c_{\tau_{\fl}\chi_3}}{c_{\tau_{\fl}}c_{\chi_3}}, 
    \quad \tilde c_{\tau_{\beta}\chi_3} \coloneqq \frac{c_{\tau_{\beta}\chi_3}}{c_{\tau_{\beta}}c_{\chi_3}}, \\
    \tilde c_{\tau_{\beta}\tau_{\beta}\tau_{\fl}\tau_{\beta}} \coloneqq \frac{c_{\tau_{\beta}\tau_{\beta}\tau_{\fl}\tau_{\beta}}}{c_{\tau_{\beta}}^3c_{\tau_{\fl}}}, 
    \quad \tilde c_{\tau_{\beta}\tau_{\fl}\tau_{\fl}\tau_{\beta}} \coloneqq \frac{c_{\tau_{\beta}\tau_{\fl}\tau_{\fl}\tau_{\beta}}}{c_{\tau_{\beta}}^2c_{\tau_{\fl}}^2},
    \quad \tilde c_{\tau_{\fl}\tau_{\fl}\tau_{\fl}\tau_{\beta}} \coloneqq \frac{c_{\tau_{\fl}\tau_{\fl}\tau_{\fl}\tau_{\beta}}}{c_{\tau_{\beta}}c_{\tau_{\fl}}^3},
\end{gather*}
whenever $c_{\sigma_3}$, $c_{\chi_3} \neq 0$ as well, as is implied by the $p$-adic period conjecture (see \Cref{period-conjecture-implies-coeff-nonzero}). These constants are invariant under rescaling the Lie algebra generators $\tau_{\beta}, \tau_{\fl}, \sigma_3, \chi_3$.

\begin{thm}
\label{thm:realquadS=1eqs}
    Let $K$ be a real quadratic field and assume that $S=\{\fl\}$ is Galois-stable. Let $p$ be a rational prime which splits in~$K$. Assume that $c_{\sigma_3}, c_{\chi_3}, c_{\tau_{\fl}\tau_{\beta}} \neq 0$ in $\bQ_p$. Then, for $n=1,2,3,4$, the depth-$n$ Chabauty--Kim locus $X(\cO_K \otimes \bZ_p)_{\{\fl\},\PL,n}$ is exactly the set of pairs $(z_1,z_2) \in X(\bZ_p) \times X(\bZ_p)$ satisfying the following incremental list of equations:
    \begin{itemize}
        \item in depth~1: no equations;
        \item in depth~2:
        \begin{gather}
            L_2(z_1)+L_2(z_2)=0, \label{eq:realquadS=1eq1}\\
            2L_2(z_1) + c_2(\log(z_1)\Li_1(z_2) - \log(z_2)\Li_1(z_1)) = 0; \label{eq:realquadS=1eq2}
        \end{gather}
        \item in depth~3: no additional equations;
        \item in depth~4:
        \begin{gather}
        \begin{aligned}
        \label{eq:realquadS=1eq3}
            2L_4(z_1) + 2L_4(z_2) &= c_{4,1}(\log(z_1) + \log(z_2))(L_3(z_1) + L_3(z_2))\\
            &+ c_{4,2}(\log(z_1) - \log(z_2))(L_3(z_1) - L_3(z_2)) \\
            &+ c_{4,3} (\log(z_1)^2 - \log(z_2)^2) L_2(z_1),
        \end{aligned}
        \\[1mm]
        \begin{aligned}
        \label{eq:realquadS=1eq4}
            2L_4(z_1) - 2L_4(z_2) &= c_{4,4}(\log(z_1) - \log(z_2))(L_3(z_1) + L_3(z_2))\\
            &+ c_{4,5}(\log(z_1) + \log(z_2))(L_3(z_1) - L_3(z_2))\\
            &+ c_{4,6} (\log(z_1) - \log(z_2))^2 L_2(z_1)\\
            &+ c_{4,7} (\log(z_1) + \log(z_2))^2 L_2(z_1),
        \end{aligned}
        \end{gather}
    \end{itemize}where
    \begin{gather*}
        c_2 \coloneqq \tilde c_{\tau_{\fl}\tau_{\beta}},\\
        c_{4,1} \coloneqq \tilde c_{\tau_{\fl} \sigma_3}, \quad c_{4,2} \coloneqq \tilde c_{\tau_{\beta} \chi_3}, \\
        c_{4,3} \coloneqq (\tilde c_{\tau_{\beta}\tau_{\fl}\tau_{\fl}\tau_{\beta}} - \tilde c_{\tau_{\fl}\sigma_3} \tilde c_{\tau_{\beta}\tau_{\fl}\tau_{\beta}} - \tilde c_{\tau_{\beta}\chi_3} \tilde c_{\tau_{\fl}\tau_{\fl}\tau_{\beta}})/\tilde c_{\tau_{\fl}\tau_{\beta}},\\
        c_{4,4} \coloneqq \tilde c_{\tau_{\beta}\sigma_3}, \quad c_{4,5} \coloneqq \tilde c_{\tau_{\fl}\chi_3},\\
        c_{4,6} \coloneqq (\tilde c_{\tau_{\beta}\tau_{\beta}\tau_{\fl}\tau_{\beta}} -\tilde c_{\tau_{\beta} \sigma_3} \tilde c_{\tau_{\beta} \tau_{\fl} \tau_{\beta}})/\tilde c_{\tau_{\fl}\tau_{\beta}},\\
        c_{4,7} \coloneqq (\tilde c_{\tau_{\fl}\tau_{\fl}\tau_{\fl}\tau_{\beta}} - \tilde c_{\tau_{\fl}\chi_3} \tilde c_{\tau_{\fl}\tau_{\fl}\tau_{\beta}})/\tilde c_{\tau_{\fl}\tau_{\beta}}.
    \end{gather*}
\end{thm}

\begin{proof}
    We have $\sigma \beta = \beta^{-1}$ in $\cO_K^{\times} \otimes \bQ$ since $\beta$ has norm~$\pm 1$, and we have $\sigma \pi_{\fl} = \pi_{\fl}$ by our choice of~$\pi_{\fl}$. Dually, we find $\sigma^*\tau_{\beta}=-\tau_{\beta}$ and $\sigma^*\tau_{\fl}=\tau_{\fl}$. Recall that $\sigma_3$ and $\chi_3$ are chosen such that $\sigma^* \sigma_3 = \sigma_3$ and $\sigma^* \chi_3 = -\chi_3$. Thus, applying $\sigma^*$ to~\eqref{eq:real-S=1-eps1}, the $\fp_2$-adic period element $\eps_2 = \sigma^* \eps_1$, is given by
    \begin{align}
    \label{eq:real-S=1-eps2}
        \eps_2 &= - c_{\tau_{\beta}} \tau_{\beta} + c_{\tau_{\fl}} \tau_{\fl} - c_{\tau_{\fl}\tau_{\beta}} [\tau_{\fl},\tau_{\beta}] + c_{\sigma_3} \sigma_3 - c_{\chi_3} \chi_3
        + c_{\tau_{\beta}\tau_{\fl}\tau_{\beta}} [\tau_{\beta},[\tau_{\fl},\tau_{\beta}]] - c_{\tau_{\fl}\tau_{\fl}\tau_{\beta}} [\tau_{\fl},[\tau_{\fl},\tau_{\beta}]]\\ \notag
        &- c_{\tau_{\beta}\sigma_3} [\tau_{\beta},\sigma_3] + c_{\tau_{\fl}\sigma_3} [\tau_{\fl},\sigma_3] + c_{\tau_{\beta}\chi_3} [\tau_{\beta},\chi_3] - c_{\tau_{\fl}\chi_3} [\tau_{\fl},\chi_3] \\ \notag
        &- c_{\tau_{\beta}\tau_{\beta}\tau_{\fl}\tau_{\beta}} [\tau_{\beta},[\tau_{\beta}, [\tau_{\fl}, \tau_{\beta}]]] 
        + c_{\tau_{\beta}\tau_{\fl}\tau_{\fl}\tau_{\beta}} [\tau_{\beta},[\tau_{\fl}, [\tau_{\fl}, \tau_{\beta}]]] 
        - c_{\tau_{\fl}\tau_{\fl}\tau_{\fl}\tau_{\beta}} [\tau_{\fl},[\tau_{\fl}, [\tau_{\fl}, \tau_{\beta}]]] + \ldots
    \end{align}
    Specialising the discussion of §\ref{sec:coordinates} to the present situation, the Selmer scheme $\Sel^{\mot}_{S,\PL,4}(X)$ is a 6-dimensional affine space with coordinates $x_{\fl}, x_{\beta}, y_{\fl}, y_{\beta}, z, w$, with $z \coloneqq z_{3,1}$ corresponding to~$\sigma_3$ and $w \coloneqq z_{3,2}$ corresponding to~$\chi_3$. By \Cref{cocycle-evaluation-map}, the evaluation maps at $\eps_1$ and $\eps_2$ in depth~$4$ are given by
    \begin{align}
        \ev_{\eps_i}^{\sharp} L_0 &= (-1)^{i+1} c_{\tau_\beta}x_{\beta} + c_{\tau_{\fl}}x_{\fl},
        \label{eq:real-stable-ev-L0} \\
        \ev_{\eps_i}^{\sharp} L_1 &= (-1)^{i+1}c_{\tau_\beta}y_{\beta} + c_{\tau_{\fl}}y_{\fl}, \label{eq:real-stable-ev-L1}  \\
        \label{eq:real-stable-ev-L2}
        \ev_{\eps_i}^{\sharp} L_2 &= (-1)^{i+1}c_{\tau_{\fl}\tau_{\beta}}(x_{\fl} y_{\beta} - x_{\beta} y_{\fl}),\\
        \label{eq:real-stable-ev-L3}
        \ev_{\eps_i}^{\sharp} L_3 &= c_{\sigma_3} z + (-1)^{i+1} c_{\chi_3} w + \bigl(c_{\tau_{\beta}\tau_{\fl}\tau_{\beta}} x_{\beta} + (-1)^{i+1} c_{\tau_{\fl}\tau_{\fl}\tau_{\beta}} x_{\fl}\bigr)(x_{\fl} y_{\beta} - x_{\beta} y_{\fl}), \\
        \ev_{\eps_i}^{\sharp} L_4 &= (-1)^{i+1} c_{\tau_{\beta}\sigma_3} x_{\beta} z + c_{\tau_{\fl}\sigma_3} x_{\fl} z + c_{\tau_{\beta}\chi_3} x_{\beta} w + (-1)^{i+1} c_{\tau_{\fl}\chi_3} x_{\fl} w \\
        &\quad + \bigl((-1)^{i+1}  c_{\tau_{\beta}\tau_{\beta}\tau_{\fl}\tau_{\beta}} x_{\beta}^2 + c_{\tau_{\beta}\tau_{\fl}\tau_{\fl}\tau_{\beta}} x_{\beta} x_{\fl} + (-1)^{i+1}  c_{\tau_{\fl}\tau_{\fl}\tau_{\fl}\tau_{\beta}} x_{\fl}^2\bigr) (x_{\fl} y_{\beta} - x_{\beta} y_{\fl}) \notag
    \end{align}
    for $i=1,2$. We need to determine the scheme-theoretic image of
    \[ \ev_{\eps_1} \times \ev_{\eps_2}\colon \Sel^{\mot}_{S,\PL,4}(X) \to \Lie(\Pi_{\PL,4}^{\dR}) \times \Lie(\Pi_{\PL,4}^{\dR}). \]
    Let $L_n^{(i)}$ denote $L_n$ composed with the $i$-th projection of $\Lie(\Pi_{\PL,4}^{\dR}) \times \Lie(\Pi_{\PL,4}^{\dR})$ for $i=1,2$. The graph of $\ev_{\eps_1} \times \ev_{\eps_2}$ is the affine variety with coordinate ring
    \begin{equation}
        \label{eq:real-stable-graph-ring}
        R \coloneqq \bQ_p[x_{\fl},x_{\beta},y_{\fl},y_{\beta},z,w, (L_n^{(i)})_{n,i}] \bigm/ (L_n^{(i)} - \ev_{\eps_i}^{\sharp}L_n).
    \end{equation}
    In~$R$ we compute
    \begin{gather*}
        L_0^{(1)} + L_0^{(2)} = 2 c_{\tau_{\fl}} x_{\fl}, \qquad L_0^{(1)} - L_0^{(2)} = 2 c_{\tau_{\beta}} x_{\beta},\\
        L_1^{(1)} + L_1^{(2)} = 2 c_{\tau_{\fl}} y_{\fl}, \qquad L_1^{(1)} - L_1^{(2)} = 2 c_{\tau_{\beta}} y_{\beta},\\
        L_3^{(1)} + L_3^{(2)} = 2c_{\sigma_3} z + 2c_{\tau_{\beta}\tau_{\fl}\tau_{\beta}} x_{\beta}(x_{\fl} y_{\beta} - x_{\beta} y_{\fl}) \\
        L_3^{(1)} - L_3^{(2)} = 2c_{\chi_3} w + 2c_{\tau_{\fl}\tau_{\fl}\tau_{\beta}} x_{\fl}(x_{\fl} y_{\beta} - x_{\beta} y_{\fl})
    \end{gather*}
    Note that the first two lines can be solved for $x_{\fl}, x_{\beta}, y_{\fl}, y_{\beta}$ and then the last two lines can be solved for~$z$ and~$w$. This defines a splitting $\phi$ of $\ev_{\eps_1} \times \ev_{\eps_2}$ and a projection $\psi \coloneqq (\ev_{\eps_1} \times \ev_{\eps_2}) \circ \phi$ onto the scheme-theoretic image. This image is cut by the equations $f = \psi^\sharp(f)$ with $f$ running through a set of algebra generators of $\cO(\Lie(\Pi_{\PL,4}^{\dR}) \times \Lie(\Pi_{\PL,4}^{\dR}))$. Taking $f = L_2^{(1)} + L_2^{(2)}$ we compute
    \[ \psi^{\sharp}(L_2^{(1)} + L_2^{(2)}) = c_{\tau_{\fl}\tau_{\beta}}(x_{\fl} y_{\beta} - x_{\beta} y_{\fl}) - c_{\tau_{\fl}\tau_{\beta}}(x_{\fl} y_{\beta} - x_{\beta} y_{\fl}) = 0, \]
    resulting in the equation $L_2^{(1)} + L_2^{(2)} = 0$ for the scheme-theoretic image, which pulls back to equation~\eqref{eq:realquadS=1eq1} for the depth-2 Chabauty--Kim locus. For $f = 2L_2^{(1)}$ we compute
    \begin{align*}
        \psi^{\sharp}(2L_2^{(1)}) &= 2c_{\tau_{\fl}\tau_{\beta}} (x_{\fl} y_{\beta} - x_{\beta} y_{\fl}) = 2c_{\tau_{\fl}\tau_{\beta}}\left(\frac{L_0^{(1)} + L_0^{(2)}}{2 c_{\tau_{\fl}}} \frac{L_1^{(1)} - L_1^{(2)}}{2 c_{\tau_{\beta}}} - \frac{L_0^{(1)} - L_0^{(2)}}{2 c_{\tau_{\beta}}} \frac{L_1^{(1)} + L_1^{(2)}}{2 c_{\tau_{\fl}}} \right) \\
        &= -\tilde c_{\tau_{\fl}\tau_{\beta}}(L_0^{(1)} L_1^{(2)} - L_0^{(2)} L_1^{(1)}),
    \end{align*}
    resulting in~\eqref{eq:realquadS=1eq2}. With $f = 2L_4^{(1)} + 2L_4^{(2)}$ we get
    \begin{align*}
        &\psi^{\sharp}(2L_4^{(1)} + 2L_4^{(2)}) \\
        &= 4 c_{\tau_{\fl}\sigma_3} x_{\fl} z + 4 c_{\tau_{\beta}\chi_3} x_{\beta} w + 4c_{\tau_{\beta}\tau_{\fl}\tau_{\fl}\tau_{\beta}} x_{\beta} x_{\fl} (x_{\fl} y_{\beta} - x_{\beta} y_{\fl}) \\
        &= 4 c_{\tau_{\fl}\sigma_3} x_{\fl} \frac{L_3^{(1)} + L_3^{(2)} - 2c_{\tau_{\beta}\tau_{\fl}\tau_{\beta}} x_{\beta} (x_{\fl} y_{\beta} - x_{\beta} y_{\fl})}{2 c_{\sigma_3}} \\
        &\qquad + 4 c_{\tau_{\beta}\chi_3} x_{\beta} \frac{L_3^{(1)} - L_3^{(2)} - 2c_{\tau_{\fl}\tau_{\fl}\tau_{\beta}} x_{\fl}(x_{\fl} y_{\beta} - x_{\beta} y_{\fl})}{2c_{\chi_3}}\\
        &\qquad + 4c_{\tau_{\beta}\tau_{\fl}\tau_{\fl}\tau_{\beta}} x_{\beta} x_{\fl} (x_{\fl} y_{\beta} - x_{\beta} y_{\fl}) \\[3mm]
        &= \frac{2 c_{\tau_{\fl}\sigma_3}}{c_{\sigma_3}} x_{\fl} (L_3^{(1)} + L_3^{(2)}) + \frac{2c_{\tau_{\beta}\chi_3}}{c_{\chi_3}} x_{\beta}(L_3^{(1)} - L_3^{(2)}) \\
        &\qquad + 4\left(c_{\tau_{\beta}\tau_{\fl}\tau_{\fl}\tau_{\beta}} - \frac{c_{\tau_{\fl}\sigma_3}c_{\tau_{\beta}\tau_{\fl}\tau_{\beta}}}{c_{\sigma_3}} - \frac{c_{\tau_{\beta}\chi_3} c_{\tau_{\fl}\tau_{\fl}\tau_{\beta}}} {c_{\chi_3}} \right) x_{\beta}x_{\fl}(x_{\fl} y_{\beta} - x_{\beta} y_{\fl}) \\[3mm]
        &= \tilde c_{\tau_{\fl}\sigma_3} (L_0^{(1)} + L_0^{(2)})(L_3^{(1)} + L_3^{(2)}) + \tilde c_{\tau_{\beta}\chi_3}(L_0^{(1)} - L_0^{(2)})(L_3^{(1)} - L_3^{(2)}) \\
        &\qquad + \left(\frac{c_{\tau_{\beta}\tau_{\fl}\tau_{\fl}\tau_{\beta}}}{c_{\tau_{\beta}} c_{\tau_{\fl}} c_{\tau_{\fl}\tau_{\beta}}} - \frac{c_{\tau_{\fl}\sigma_3}c_{\tau_{\beta}\tau_{\fl}\tau_{\beta}}}{c_{\sigma_3} c_{\tau_{\beta}} c_{\tau_{\fl}} c_{\tau_{\fl}\tau_{\beta}}} - \frac{c_{\tau_{\beta}\chi_3} c_{\tau_{\fl}\tau_{\fl}\tau_{\beta}}} {c_{\chi_3} c_{\tau_{\beta}} c_{\tau_{\fl}} c_{\tau_{\fl}\tau_{\beta}}} \right) (L_0^{(1)} - L_0^{(2)}) (L_0^{(1)} + L_0^{(2)}) L_2^{(1)}\\[3mm]
        &= c_{4,1} (L_0^{(1)} + L_0^{(2)})(L_3^{(1)} + L_3^{(2)}) + c_{4,2} (L_0^{(1)} - L_0^{(2)})(L_3^{(1)} - L_3^{(2)}) \\
        &\qquad +c_{4,3} ((L_0^{(1)})^2 - (L_0^{(2)})^2) L_2^{(1)},
    \end{align*}
    resulting in equation~\eqref{eq:realquadS=1eq3} for the depth-4 Chabauty--Kim locus. A similar calculation with $f = 2L_4^{(1)} - 2L_4^{(2)}$ yields~\eqref{eq:realquadS=1eq4}. The remaining algebra generators $L_n^{(1)} \pm L_n^{(2)}$ of $\cO(\Lie(\Pi_{\PL,4}^{\dR}) \times \Lie(\Pi_{\PL,4}^{\dR}))$ for $n=1,3$ don't produce additional equations since they were used to define the splitting~$\phi$.
\end{proof}

If one knows some $S$-integral points, the following lemma can be used to compute the coefficients $c_2$ and $c_{4,1},\ldots,c_{4,7}$ appearing in the equations of \Cref{thm:realquadS=1eqs} and verify the non-vanishing of the $p$-adic periods $c_{\tau_{\fl}\tau_{\beta}}$ at the same time.

\begin{lemma}
\label{real-stable-coefficients}
    Let $K$ be a real quadratic field and assume that $S = \{\fl\}$ is Galois-stable. Let~$p$ be a rational prime which splits in~$K$ and fix an embedding $K \hookrightarrow \bQ_p$.
    \begin{enumerate}[label=(\alph*)]
        \item \label{item:real-stable-depth2-coeff}
        Assume that $\alpha \in X(\cO_{K,S})$ is an $S$-integral point such that $\alpha \equiv 0 \bmod \fl$ and $\alpha \neq 2$. Then the coefficient $c_2$ in~\eqref{eq:realquadS=1eq2} is given by
        \[ c_2 = -\frac{2 L_2(\alpha)}{\log(\alpha)\Li_1(\sigma \alpha) - \log(\sigma \alpha)\Li_1(\alpha)}. \]
        \item \label{item:real-stable-depth4-coeffs-1}
        Assume that $\alpha_1,\alpha_2,\alpha_3 \in X(\cO_{K,S})$ are $S$-integral points such that the matrix
        \[  M_1(\alpha_1,\alpha_2,\alpha_3) \coloneqq \begin{pmatrix}
            (\log(\alpha_i) + \log(\sigma \alpha_i))(L_3(\alpha_i) + L_3(\sigma \alpha_i))\\
            (\log(\alpha_i) - \log(\sigma \alpha_i))(L_3(\alpha_i) - L_3(\sigma \alpha_i))\\
            (\log(\alpha_i)^2 - \log(\sigma \alpha_i)^2)L_2(\alpha_i) 
        \end{pmatrix}^T_{i=1,2,3} \]
        has nonzero determinant in~$\bQ_p$. Then $c_{\tau_{\fl}\tau_{\beta}}, c_{\sigma_3}, c_{\chi_3} \neq 0$ and the coefficients $c_{4,1}, c_{4,2}, c_{4,3}$ in~\eqref{eq:realquadS=1eq3} are given by
        \[  (c_{4,1},\; c_{4,2},\; c_{4,3})^T = M_1(\alpha_1,\alpha_2,\alpha_3)^{-1} \bigl(2L_4(\alpha_i) + 2L_4(\sigma\alpha_i)\bigr)_{i=1,2,3}. \]
        \item \label{item:real-stable-depth4-coeffs-2}
        Assume that $\alpha_1,\alpha_2,\alpha_3,\alpha_4 \in X(\cO_{K,S})$ are $S$-integral points such that the matrix
        \[  M_2(\alpha_1,\alpha_2,\alpha_3,\alpha_4) \coloneqq \begin{pmatrix}
            (\log(\alpha_i) - \log(\sigma \alpha_i))(L_3(\alpha_i) + L_3(\sigma \alpha_i))\\
            (\log(\alpha_i) + \log(\sigma \alpha_i))(L_3(\alpha_i) - L_3(\sigma \alpha_i))\\
            (\log(\alpha_i) - \log(\sigma \alpha_i))^2 L_2(\alpha_i) \\
            (\log(\alpha_i) + \log(\sigma \alpha_i))^2 L_2(\alpha_i) 
        \end{pmatrix}^T_{i=1,2,3,4} \]
        has nonzero determinant in~$\bQ_p$. Then $c_{\tau_{\fl}\tau_{\beta}}, c_{\sigma_3}, c_{\chi_3} \neq 0$ and the coefficients $c_{4,4},\ldots, c_{4,7}$ in~\eqref{eq:realquadS=1eq4} are given by
        \[  (c_{4,4}, c_{4,5}, c_{4,6}, c_{4,7})^T = M_2(\alpha_1,\alpha_2,\alpha_3,\alpha_4)^{-1} \bigl(2L_4(\alpha_i) - 2L_4(\sigma\alpha_i)\bigr)_{i=1,2,3,4}. \]
    \end{enumerate}
\end{lemma}

\begin{proof}
    For part~\ref{item:real-stable-depth2-coeff}, note that every $S$-integral point $\alpha \in X(\cO_{K,S})$ yields a solution $(\alpha,\sigma\alpha)$ to equation~\eqref{eq:realquadS=1eq2}, whether the $p$-adic periods vanish or not. We can solve the equation for $c_2$ as soon as $\log(\alpha)\Li_1(\sigma \alpha) - \log(\sigma\alpha)\Li_1(\alpha) \neq 0$.    
    The assumption $\alpha \equiv 0 \bmod \fl$ implies that $1- \alpha \in \cO_K^{\times}$. So $1-\alpha$ has norm~1, which implies $\Li_1(\sigma\alpha) = -\Li_1(\alpha)$, hence
    \begin{align*}
        \log(\alpha)\Li_1(\sigma \alpha) - \log(\sigma\alpha)\Li_1(\alpha) &= -\Li_1(\alpha)((\log(\alpha) + \log(\sigma\alpha)) \\
        &= \log(1-\alpha)\log(\Nm_{K/\bQ}(\alpha)).
    \end{align*} 
    Since $\alpha \neq 2$ and $\fl \mid \alpha$, neither $1-\alpha$ nor $\Nm_{K/\bQ}(\alpha)$ are roots of unity, so this expression is nonzero and~\eqref{eq:realquadS=1eq2} can be solved for~$c_2$.

    For part~\ref{item:real-stable-depth4-coeffs-1}, observe that if $c_{\tau_{\fl}\tau_{\beta}}$ were~0 then by~\eqref{eq:real-stable-ev-L2} one would have $L_2(\alpha) = 0$ for every $S$-integral point $\alpha \in X(\cO_{K,S})$. But then the third column of the matrix $M_1 \coloneqq M_1(\alpha_1,\alpha_2,\alpha_3)$ would be zero and the determinant would vanish. Similarly, if $c_{\sigma_3}$ were~0 then by \eqref{eq:real-stable-ev-L3}, computing in the coordinate ring \eqref{eq:real-stable-graph-ring} of the graph of $\ev_{\eps_1} \times \ev_{\eps_2}$, one would have 
    \[
        L_3^{(1)} + L_3^{(2)} = 2 c_{\tau_{\beta}\tau_{\fl}\tau_{\beta}} x_{\beta} (x_{\fl}y_{\beta} - x_{\beta}y_{\fl}) 
        = 2 c_{\tau_{\beta}\tau_{\fl}\tau_{\beta}} \frac{L_0^{(1)} - L_0^{(2)}}{2c_{\tau_{\beta}}} \frac{L_2^{(1)}}{c_{\tau_{\fl}\tau_{\beta}}} 
    \]
    hence every $S$-integral point~$\alpha$ would satisfy 
    \[ L_3(\alpha) + L_3(\sigma\alpha) = \frac{c_{\tau_{\beta}\tau_{\fl}\tau_{\beta}}}{c_{\tau_{\beta}} c_{\tau_{\fl}\tau_{\beta}}} (\log(\alpha) - \log(\sigma \alpha)) L_2(\alpha). \]
    But then the first and third column of the matrix $M_1$ would be linearly dependent, again contradicting the non-vanishing of the determinant. A similar argument with $L_3(\alpha) - L_3(\sigma\alpha)$ shows $c_{\chi_3} \neq 0$. The assumptions of \Cref{thm:realquadS=1eqs} are thus satisfied and plugging each $(\alpha_i,\sigma \alpha_i)$ into Eq.~\eqref{eq:realquadS=1eq3} produces a linear equation for $c_{4,1},c_{4,2},c_{4,3}$. The non-vanishing of the determinant ensures that the linear equations can be solved for the $c_{4,i}$. The proof of part~\ref{item:real-stable-depth4-coeffs-2} is completely analogous.
\end{proof}

Using \Cref{real-stable-coefficients}, we can now verify Kim's Conjecture in the case $K = \bQ(\sqrt{2})$, $S = \emptyset$.
\begin{thm}
\label{thm:Qsqrt2SemptyKimholds}
    When $K=\mathbb{Q}(\sqrt{2})$ and $S=\emptyset$, Kim's Conjecture (\Cref{conj:kim-full}) for the polylogarithmic quotient holds in depth $4$ for $p<2{,}000$ which split completely in $K$.
\end{thm}
\begin{proof}
    By enlarging $S$ to $S'=\left\{(\sqrt{2})\right\}$ and plugging in points of $X(\mathcal{O}_{K,S'})$ into equations \eqref{eq:realquadS=1eq3}, \eqref{eq:realquadS=1eq4}, we can determine $c_{4,2}=c_{\tau_\beta\chi_3}/c_{\tau_\beta}c_{\chi_3}$, $c_{4,4}=c_{\tau_\beta\sigma_3}/c_{\tau_\beta}c_{\sigma_3}$ by \Cref{real-stable-coefficients}, which also determine the coefficients appearing in \eqref{eq:realquadSemptydepth4-1}, \eqref{eq:realquadSemptydepth4-2}. Numerical computation shows that the depth $2$ solutions (which are at most the $S_3$-orbits of $((-1+\sqrt{5})/2,(-1-\sqrt{5})/2)$ for $p<2{,}000$ by \Cref{real-integral-depth2-locus}) do not satisfy \eqref{eq:realquadSemptydepth4-1}, \eqref{eq:realquadSemptydepth4-2} for $p<2{,}000$. 
\end{proof}

We also derive the equations for the refined Chabauty--Kim locus $X(\cO_K \otimes \bZ_p)_{\{\fl\},n}^{(1)}$ for $n=1,2$, for the refinement condition $\Sigma = (1)$. The full refined locus $X(\cO_K \otimes \bZ_p)_{\{\fl\},n}^{\min}$ can be obtained from this by taking $S_3$-orbits.

\begin{thm}
\label{real-refined1-equations}
    In the situation of \Cref{thm:realquadS=1eqs}, 
    for $n=1,2,3,4$, the refined Chabauty--Kim locus $X(\cO_K \otimes \bZ_p)_{\{\fl\},n}^{(1)}$ is exactly the set of pairs $(z_1,z_2) \in X(\bZ_p) \times X(\bZ_p)$ satisfying the following incremental list of equations:
    \begin{itemize}
        \item in depth~1:
        \begin{equation}
        \label{eq:real-refined1-log}
            \log(z_1)+\log(z_2)=0;
        \end{equation}
        \item in depth~2:
        \begin{gather}
            \label{eq:real-refined1-L2}
            L_2(z_1)+L_2(z_2) =0, \\
            \label{eq:real-refined1-L2-with-c2}
            2L_2(z_1) + c_2\log(z_1)(\Li_1(z_1)+\Li_1(z_2)) = 0;
        \end{gather}
        \item in depth~3: no additional equations;
        \item in depth~4:
        \begin{align}
            L_4(z_1) + L_4(z_2) &= c_{4,2} \log(z_1)(L_3(z_1) - L_3(z_2)) + c_{4,8} \log(z_1)^2 L_2(z_1),\\
            L_4(z_1) - L_4(z_2) &= c_{4,4} \log(z_1)(L_3(z_1) + L_3(z_2)) + c_{4,9} \log(z_1)^2 L_2(z_1),
        \end{align}
    \end{itemize}
    where $c_2$, $c_{4,2}$, $c_{4,4}$ are as in \Cref{thm:realquadS=1eqs} and
    \begin{equation*}
        c_{4,8} \coloneqq -2\tilde c_{\tau_{\beta}\chi_3} \tilde c_{\tau_{\beta}\tau_{\fl}\tau_{\beta}}/\tilde c_{\tau_{\beta}\tau_{\fl}}, \qquad 
        c_{4,9} \coloneqq 2 \tilde c_{\tau_{\beta}\tau_{\beta}\tau_{\fl}\tau_{\beta}}/\tilde c_{\tau_{\fl}\tau_{\beta}}.
    \end{equation*}
\end{thm}

\begin{proof}
    By \Cref{def:refined-selmer-scheme}, the refined Selmer scheme $\Sel_{S,2}^{\mot,(1)}(X)$ is defined by the equation $x_{\fl}=0$. Setting $x_{\fl}=0$ in \eqref{eq:real-stable-ev-L0}, \eqref{eq:real-stable-ev-L1}, \eqref{eq:real-stable-ev-L2}, the restrictions of the evaluation maps $\ev_{\eps_i}$ to the refined Selmer scheme are given by
    \begin{align*}
        \ev_{\eps_1}\colon (0,x_{\beta},y_{\fl},y_{\beta}) &\mapsto (c_{\tau_\beta}x_{\beta},\; c_{\tau_{\fl}}y_{\fl}+c_{\tau_\beta}y_{\beta},\; -c_{\tau_{\fl}\tau_{\beta}}x_{\beta}y_{\fl}),\\
        \ev_{\eps_2}\colon(0,x_{\beta},y_{\fl},y_{\beta})&\mapsto (-c_{\tau_\beta}x_{\beta},\; c_{\tau_{\fl}}y_{\fl}-c_{\tau_\beta}y_{\beta},\; c_{\tau_{\fl}\tau_{\beta}}x_{\beta}y_{\fl}).
    \end{align*}
    Comparing these two, one obtains the equations defining the image of $\ev_{\eps_1} \times \ev_{\eps_2}$.
\end{proof}

The two equations \eqref{eq:real-refined1-log} and \eqref{eq:real-refined1-L2} both hold on the curve $\{z_1 z_2 = 1\} \subseteq X(\bZ_p) \times X(\bZ_p)$ by the functional equation $L_2(z) + L_2(z^{-1}) = 0$ (\Cref{coleman-sinnott-for-Ln}). Due to this unlikely intersection we do not expect $X(\cO_K \otimes \bZ_p)_{\{\fl\},2}^{(1)}$ to contain only $S$-integral points even though we have three equations in two variables: the locus contains all pairs of the form $(z,z^{-1})$ with $z \in X(\bZ_p)$ a root of the function $L_2(z) + c_2 \log(z)(2\Li_1(z) + \log(z))$. This includes $(-1,-1)$ (which is not $\{\fl\}$-integral when $\fl \nmid 2$) but typically also transcendental $p$-adic points.

\begin{example}
\label{Qsqrt2-example}
    Let $K=\mathbb{Q}(\sqrt{2})$ and $S=\{(\sqrt{2})\}$. We apply \Cref{real-stable-coefficients} with $z=\sqrt{2}$ to obtain
    \[ c_2 = \frac{L_2(\sqrt{2})}{\log(2)\log(1+\sqrt{2})}. \]
    For various $p$ which split in $K$, the numerical results for the Chabauty--Kim locus are shown in the following table, using the Sage code \url{https://github.com/steffenmueller/LinQC} implemented by Francesca Bianchi to solve the system of $2$ equations in $2$ variables.
    \begin{center}
    \begin{tabular}{ |c|c|c|c|c|c|c| } 
     \hline
      & $p=7$ & $p=17$ & $p=23$ & $p=31$ & $p=41$ & $p=47$
      \\  \hline
     $\# X(\mathcal{O}_K\otimes \mathbb{Z}_p)_{S,\PL,2}$ & $39$ & $183$ & $345$ & $867$  & $1647$ & $1929$ \\
     $\# X(\mathcal{O}_K\otimes \mathbb{Z}_p)_{S,\PL,2}^{\min}$ & $39$ & $63$ & $75$ & $171$  &  $183$ & $195$ \\  
     $\# X(\mathcal{O}_K\otimes \mathbb{Z}_p)_{S,\PL,4}$ & $33$ & $33$ & $33$ & $33$ & $33$ & $33$ \\ \hline
    \end{tabular}
    \end{center}
We found that for the primes $p$ in the table, $\# X(\mathcal{O}_K\otimes \mathbb{Z}_p)_{S,\PL,4}=33$ which is the same as the cardinality of the set of actual points $X(\mathcal{O}_{K,S})$, indicating that the Kim's conjecture holds. We did this computation up to $p<100$ summarised in the following theorem.
\end{example}

\begin{thm}
\label{thm:Qsqrt2S=1Kimholds}
    When $K=\mathbb{Q}(\sqrt{2})$ and $S=\left\{(\sqrt{2})\right\}$, Kim's Conjecture (\Cref{conj:kim-full}) for the polylogarithmic quotient holds in depth $4$ for $p<100$ which split completely in $K$.
\end{thm}

\subsection{Galois-unstable $S$ of size~1}

Now consider a real quadratic field~$K$ and $S = \{\fl\}$ of size~$1$ with $\fl$ a Galois-unstable prime of~$K$. Write $l\mathcal{O}_K=\fl\fl'$ where $l$ is the rational prime lying below $\fl$. 
As in \eqref{eq:real-S=1-eps1} above, we write the $\fp_1$-adic period element $\varepsilon_1$ as
\begin{equation}
   \eps_1 = c_{\tau_{\beta}} \tau_{\beta} + c_{\tau_{\fl}} \tau_{\fl} + c_{\tau_{\fl}\tau_{\beta}} [\tau_{\fl},\tau_{\beta}] + \ldots
\end{equation}

For a word $w$, we denote $c_w' \coloneq \per_{\fp_2}(c_w^{\fu})=\per_{\sigma^*\fp_1}(c_w^{\fu})\in K_{\fp_2}$, which equals $\per_{\fp_1}(\sigma_* c_w^{\fu})$ by \Cref{eta-sigma-p-from-eta-p}. 
Thus, applying $\sigma^*$ to $\eps_1$, the $\fp_2$-adic period element $\eps_2 = \sigma^* \eps_1$, is given by
\begin{equation}
   \eps_2 = - c_{\tau_{\beta}} \tau_{\beta} + c_{\tau_{\fl}}' \tau_{\fl} - c_{\tau_{\fl}\tau_{\beta}}' [\tau_{\fl},\tau_{\beta}] + \ldots
\end{equation}

\begin{thm}
\label{thm:realquad_unstable}
    In the setting of \Cref{thm:realquadS=1eqs} except that $S=\{\fl\}$ is not Galois-stable, for $n=1,2$, the depth-$n$ Chabauty--Kim locus $X(\cO_K \otimes \bZ_p)_{\{\fl\},n}$ is contained in the set of pairs $(z_1,z_2) \in X(\bZ_p) \times X(\bZ_p)$ satisfying the following incremental list of equations:
    \begin{itemize}
        \item in depth~1: no equations.
        \item in depth~2:
        \begin{align}
        \label{eq:realquad_unstable_depth2_1} c_{\tau_{\fl}\tau_{\beta}}' L_2(z_1)+c_{\tau_{\fl}\tau_{\beta}}L_2(z_2)&=0,\\
            \label{eq:realquad_unstable_depth2_2}        c_{\tau_{\fl}\tau_{\beta}}(\log(z_2)\Li_1(z_1)-\log(z_1)\Li_1(z_2)) &= (\log(\pi_{\fl})+\log(\sigma\pi_{\fl}))\log(\beta)L_2(z_1).
        \end{align}
    \end{itemize}
    The refined locus $X(\mathcal{O}_K\otimes\mathbb{Z}_p)_{\{\fl\},n}^{(1)}$ for $n=1,2$ is contained in the set of pairs $(z_1,z_2)$ satisfying the above list of equations plus the following equation
\begin{equation}
\label{eq:realquad_unstable_refined1}
       \log(z_1)+\log(z_2)=0,
\end{equation}
and $X(\mathcal{O}_K\otimes\mathbb{Z}_p)_{\emptyset,n}^{\min}$ is obtained by taking the $S_3$-orbits from $X(\mathcal{O}_K\otimes\mathbb{Z}_p)_{\emptyset,n}^{(1)}$ for $n=1,2$. 
    
    The containment is an equality if the $p$-adic period conjecture holds for the primes dividing~$p$.
\end{thm}

\begin{proof}
  The Selmer scheme $\Sel^{\mot}_{S,\PL,2}(X)$ is a 4-dimensional affine space with coordinates $x_{\fl}, x_{\beta}, y_{\fl}, y_{\beta}$. By \Cref{cocycle-evaluation-map}, the evaluation maps at $\eps_1$ and $\eps_2$ in depth~$2$ are given by
    \begin{align}
    \label{eq:realquad_unstable_eveps1L0}
        \ev_{\eps_1}^{\sharp} L_0 =  c_{\tau_\beta}x_{\beta} + c_{\tau_{\fl}}x_{\fl},\quad & \ev_{\eps_2}^{\sharp} L_0 = - c_{\tau_\beta}x_{\beta} + c_{\tau_{\fl}}'x_{\fl}; \\
    \label{eq:realquad_unstable_eveps1L1}
        \ev_{\eps_1}^{\sharp} L_1 = c_{\tau_\beta}y_{\beta} + c_{\tau_{\fl}}y_{\fl},\quad &  \ev_{\eps_2}^{\sharp} L_1 = -c_{\tau_\beta}y_{\beta} + c_{\tau_{\fl}}'y_{\fl}; \\
\label{eq:realquad_unstable_eveps1L2}
    \ev_{\eps_1}^{\sharp} L_2 = c_{\tau_{\fl}\tau_{\beta}}(x_{\fl} y_{\beta} - x_{\beta} y_{\fl}),\quad &  \ev_{\eps_2}^{\sharp} L_2 = -c_{\tau_{\fl}\tau_{\beta}}'(x_{\fl} y_{\beta} - x_{\beta} y_{\fl}).
\end{align}
Eliminating $x_{\fl}, x_{\beta}, y_{\fl}, y_{\beta}$ and pulling back equations to $X(\mathbb{Z}_p)\times X(\mathbb{Z}_p)$ gives \eqref{eq:realquad_unstable_depth2_1}, \eqref{eq:realquad_unstable_depth2_2}.

By \Cref{def:refined-selmer-scheme}, the refined Selmer scheme $\Sel_{S,2}^{\mot,(1)}(X)$ is defined by the equation $x_{\fl}=0$. Setting $x_{\fl}=0$ in \eqref{eq:realquad_unstable_eveps1L0}, \eqref{eq:realquad_unstable_eveps1L1}, \eqref{eq:realquad_unstable_eveps1L2} gives \eqref{eq:realquad_unstable_refined1}.
\end{proof}

	\printbibliography
	
\end{document}